\documentclass[12pt]{amsart}
\usepackage{epsfig,color}

\makeatother

\theoremstyle{definition}

\def\fnum{equation} 
\newtheorem{Thm}[\fnum]{Theorem}
\newtheorem{Cor}[\fnum]{Corollary}

\newtheorem{Lem}[\fnum]{Lemma}

\newtheorem{Rem}[\fnum]{Remark}
\newtheorem{Pro}[\fnum]{Proposition}

\numberwithin{equation}{section}
\newcommand{\Vol}{{\text{Vol}}}

\newcommand{\nn}{{\bf{n}}}
\newcommand{\Ric}{{\text{Ric}}}

\newcommand{\Identity}{{\text{Id}}}
\newcommand{\Tr}{{\text{Tr}}}

\newcommand{\Lip}{{\text {Lip}}}

\newcommand{\dist}{{\text {dist}}}
\newcommand{\cwun}{{C_1}}

\newcommand{\cK}{{\mathcal{K}}}

\newcommand{\Hess}{{\text {Hess}}}

\def\RR{{\bold R}}

\def\SS{{\bold S}}

\newcommand{\dv}{{\text {div}}}
\newcommand{\e}{{\text {e}}}

\newcommand{\Area}{{\text {Area}}}

\newcommand{\shrink}{{{R}}}
\newcommand{\graph}{{\bf{r}}_{\ell}}
\newcommand{\graphnoell}{{\bf{r}}}

\newcommand{\cC}{{\mathcal{C}}}

\newcommand{\cB}{{\mathcal{B}}}

\newcommand{\cL}{{\mathcal{L}}}

\newcommand{\cN}{{\mathcal{N}}}
\newcommand{\cM}{{\mathcal{M}}}

\newcommand{\eqr}[1]{(\ref{#1})}
\newcommand{\eps}{\epsilon}
\newcommand{\epsb}{\bar{\epsilon}}
\newenvironment{dedication}
    {\vspace{6ex}\begin{quotation}\begin{center}\begin{em}}
    {\par\end{em}\end{center}\end{quotation}}

\title[]{Uniqueness of blowups  and Lojasiewicz inequalities}

\author{Tobias Holck Colding}%
\address{MIT, Dept. of Math.\\
77 Massachusetts Avenue, Cambridge, MA 02139-4307.}
\author{William P. Minicozzi II}%
\thanks{The  authors
were partially supported by NSF Grants DMS  11040934, DMS 1206827,  and NSF FRG grants DMS 
 0854774 and DMS 0853501}

\email{colding@math.mit.edu and minicozz@math.mit.edu}

\begin{document}

\maketitle

\begin{dedication}
We dedicate this  article to Leon Simon in recognition of his fundamental contributions to analysis and geometry.
\end{dedication}

\begin{abstract}
Once one knows that singularities occur, one naturally wonders what the singularities are like. For minimal varieties the first answer, already known to Federer-Fleming in 1959, is that they 
weakly resemble cones\footnotemark[1].
%\footnote{See Brian White \cite{W4} section ``Uniqueness of tangent cone''  from which part of this discussion %is taken and where one can find more discussion of uniqueness for minimal varieties.}. 
For mean curvature flow, by the combined work of Huisken, Ilmanen, and White, singularities weakly resemble shrinkers.   Unfortunately, the simple proofs leave open the possibility that a minimal variety or a mean curvature flow looked at under a microscope will resemble one blowup, but under higher magnification, it might (as far as anyone knows) resemble a completely different blowup. Whether this ever happens is perhaps the most fundamental question about singularities.  It is this long standing open question that we settle here for mean curvature flow at all generic singularities and for mean convex mean curvature flow at all singularities. 
\end{abstract}

\section{Introduction}
We show that at each generic singularity of a mean curvature flow the blowup is unique; that is independent of the sequence of rescalings.  This settles a major open problem that was open even in the case of mean convex hypersurfaces where it was known that all singularities are generic.  Moreover, it is the first general uniqueness theorem for blowups to a Geometric PDE at a non-compact singularity.  
 
 \footnotetext[1]{See Brian White \cite{W4} section ``Uniqueness of tangent cone''  from which part of this discussion is taken and where one can find more discussion of uniqueness for minimal varieties.}
 \stepcounter{footnote}
 
 Uniqueness of blowups is perhaps the most fundamental question that one can ask about singularities and    implies regularity of the singular set; see \cite{CM5}.

 To prove our uniqueness result,
 we prove two completely new infinite dimensional Lojasiewicz type inequalities. 
  Infinite dimensional Lojasiewicz inequalities were pioneered thirty years ago by Leon Simon.  
 However, unlike all
 other infinite dimensional Lojasiewicz inequalities we know of, ours do not
 follow from a reduction to the classical finite-dimensional
 Lojasiewicz inequalities from the 1960s from algebraic geometry, rather we prove our inequalities directly and do not
 rely on Lojasiewicz's arguments or results.  
 
 It is well-known that to deal with non-compact singularities requires entirely new ideas and techniques as one cannot argue as in  Simon's work,
 and all the later work that  uses his ideas.  Partly because of this we expect that the techniques and ideas developed here have applications to other flows.
Our results hold in all dimensions.

\vskip2mm
This paper focuses on 
    mean curvature flow  (or MCF) of hypersurfaces.  This is  
a non-linear parabolic evolution equation where a hypersurface
evolves over time by locally moving in the direction of steepest
descent for the volume element.  It has been used and studied in material science for almost a century\footnote{See, e.g., the early work in material science from the 1920s, 1940s, and 1950s of T. Sutoki, \cite{Su}, D. Harker and E. Parker, \cite{HaP}, J. Burke, \cite{Bu}, P.A. Beck, \cite{Be}, J. von Neumann, \cite{N}, and W.W. Mullins, \cite{M}.}  to model things like cell, grain, and bubble growth\footnote{For instance, annealing, in metallurgy and materials science, is a heat treatment that alters a material to increase its ductility and to make it more workable. It involves heating material above its critical temperature, maintaining a suitable temperature, and then cooling. Annealing can induce ductility, soften material, relieve internal stresses, refine the structure by making it homogeneous, and improve cold working properties.  The three stages of the annealing process that proceed as the temperature of the material is increased are: recovery, recrystallization, and grain growth.   Grain growth is the increase in size of grains (crystallites) in a material at high temperature. This occurs when recovery and recrystallisation are complete and further reduction in the internal energy can only be achieved by reducing the total area of grain boundary [by mean curvature flow].}.  Unlike some of the other earlier papers in material science both von Neumann's 1952 paper and Mullins 1956 paper had explicit  equations.  In his paper von Neumann discussed soap foams whose interface tend to have constant mean curvature whereas Mullins is describing coarsening in metals, in which interfaces are not generally of constant mean curvature.  Partly as a consequence, Mullins may have been the first  to write down the MCF equation in general.  Mullins also found some of the basic self-similar solutions like the translating solution now known as the Grim Reaper.  
To be precise, suppose that
 $M_t\subset \RR^{n+1}$ is a one-parameter family of smooth hypersurfaces, then we say that $M_t$ flows by the  MCF if
\begin{equation}
x_t= -H\,\nn\, , 
\end{equation}
where $H$ and $\nn$ are   the mean curvature and unit normal, respectively, of 
$M_t$ at the point $x$.  

To understand singularities past the first singular time, we need  weak solutions of MCF.  The weak solutions that we will use   are the Brakke flows considered by White in \cite{W3}{\footnote{That is, Brakke flows in the class $S(\lambda_0 , n, n+1)$ defined in section $7$ of \cite{W3} for some $\lambda_0 > 1$.}}.   By theorem $7.4$ in \cite{W3}, this includes flows starting from any closed embedded hypersurface.
 
\subsection{Tangent flows}

By definition, a tangent flow is the limit of a sequence of rescalings
at a singularity, where the convergence is on compact subsets.{\footnote{This is analogous to a tangent cone
at a singularity of a minimal variety, cf. \cite{FFl}.}}
For instance, a tangent flow to $M_t$ at the origin in space-time is the limit of 
a sequence of rescaled flows $\frac{1}{\delta_i} \, M_{\delta_i^2 \, t}$ where $\delta_i \to 0$.
A priori, different sequences $\delta_i$ could give different tangent flows
and the question of the uniqueness of the blowup - independent of the sequence - is a major
question in many geometric problems.
By a monotonicity formula of Huisken, \cite{H1}, and an argument of Ilmanen and
White, \cite{I1},  \cite{W3}, tangent flows are   shrinkers, i.e., self-similar solutions of 
MCF that evolve by rescaling.
The only generic shrinkers are round cylinders by \cite{CM1}.

\vskip2mm 
We will say that a singular point is {\emph{cylindrical}} if at least one tangent flow is a multiplicity one cylinder $\SS^k \times \RR^{n-k}$.  
Our main application of our analytical inequalities is the 
following theorem that shows that  tangent flows at generic singularities are unique:

\begin{Thm} \label{t:main} 
Let $M_t$ be a  MCF in $\RR^{n+1}$.
  At each cylindrical singular point the tangent flow is unique.
  That is, any other tangent flow is also a cylinder with
  the same $\RR^k$ factor that points in the same direction.
\end{Thm}

This theorem solves a major open problem; see, e.g., page $534$ of \cite{W2}.  Even in the case of the evolution
of mean convex hypersurfaces where all singularities are cylindrical, uniqueness of the axis was unknown; see
\cite{HS1}, \cite{HS2}, \cite{W1}, \cite{SS}, \cite{An} and \cite{HaK}.\footnote{Our results not only give uniqueness of tangent flows but also a definite rate where the rescaled MCF converges to the relevant cylinder.  The distance to the cylinder is decaying to zero at a definite rate over balls whose radii are increasing at a definite rate to infinity.}

In recent joint work with Tom Ilmanen, \cite{CIM}, we showed
that if one tangent flow at a singular point of a MCF is
  a multiplicity one cylinder, then all are.  
  However, \cite{CIM} left open the possibility that the direction of the axis (the $\RR^k$ factor) 
  depended on the sequence of rescalings.  Our proof of Theorem \ref{t:main} and, in particular, 
  our first Lojasiewicz type inequality, 
  has its roots in some ideas and inequalities from \cite{CIM} and in fact implicitly use that cylinders are isolated among shrinkers by \cite{CIM}.
 
Uniqueness is a key question for the regularity of Geometric PDE's.
 Two of the most prominent early works on uniqueness of tangent cones are Leon 
 Simon's hugely influential paper \cite{Si1} from 1983, where he proves uniqueness 
 for tangent cones of minimal varieties with smooth cross-section. The other is 
 Allard-Almgren's 1981 paper \cite{AA}, where uniqueness of tangent cones with smooth 
 cross-section is proven under an additional integrability assumption on the cross-section; 
 see also \cite{Si2}, \cite{Hr}, \cite{CM4} for additional references.  
 
Our  results are the first general uniqueness theorems for tangent
flows to a geometric flow at a non-compact singularity.  (In fact, not only are the singularities that we deal with here non-compact but they are also non-integrable; see Section \ref{s:analysis}.)  Some special
cases of uniqueness of tangent flows for MCF were previously analyzed assuming either some sort of
convexity or that the hypersurface is a surface of rotation; see
\cite{H1}, \cite{H2}, \cite{HS1}, \cite{HS2}, \cite{W1}, \cite{SS}, \cite{AAG}, 
section $3.2$ in the book \cite{GGS}, and \cite{GK}, \cite{GKS}, \cite{GS}.  In contrast, uniqueness for
blowups at compact singularities is better understood; cf.\ \cite{AA},
  \cite{Si1}, \cite{H3}, \cite{Sc}, and \cite{Se}.
 
In fact, using the results of this paper we showed in \cite{CM5} that, for a MCF of closed embedded hypersurfaces in $\RR^{n+1}$ with only cylindrical singularities, the space-time singular set is contained in   finitely many compact embedded $(n-1)$-dimensional Lipschitz submanifolds together with a set of dimension at most $n-2$.  In particular, if the initial hypersurface is mean convex, then all singularities are generic and the results apply.   In fact, in \cite{CM5} we showed that the entire stratification of the space-time singular set is rectifiable in a very strong sense; cf., e.g., \cite{Si3}, \cite{Si4},  \cite{Si5}, \cite{BrCoL} and \cite{HrLi}.

One of the significant difficulties that we overcome in this paper, and sets it apart from all other work we know of, is that our singularities are noncompact.  This causes major analytical difficulties and to address them requires entirely new techniques and ideas. This is not so much because
of the  subtleties of analysis on noncompact domains, though this is an issue, but crucially
 because  the evolving hypersurface cannot be written as an entire graph over the singularity no matter how close we get to the singularity.
Rather,
the geometry of the situation dictates that  only part of the evolving hypersurface can be written as a graph 
over a compact piece of the singularity.{\footnote{In the end, what comes out of our analysis is that the domain the evolving hypersurface is a graph over is expanding in time and at a definite rate, but this is not all all clear from the outset; see also footnote $3$.}}
%In fact, one consequence
%of our work is that this graphical part expands over time and does so at a definite rate.
 
  \subsection{Lojasiewicz inequalities}
The main technical tools that we prove are two Lojasiewicz--type inequalities.

    In real algebraic geometry, the Lojasiewicz inequality, \cite{L}, named after Stanislaw Lojasiewicz, gives an upper bound for the distance from a point to the nearest zero of a given real analytic function. Specifically, let $f: U \to \RR$ be a real-analytic function on an open set $U$ in $\RR^n$, and let $Z$ be the zero locus of $f$. Assume that $Z$ is not empty. Then for any compact set $K$ in $U$, there exist $\alpha\geq 2$ and a positive constant $C$ such that, for all $x \in K$
  \begin{align}
  \inf_{z\in Z} |x-z|^{\alpha}\leq C\, |f(x) | \, .
  \end{align}
Here $\alpha$ can be large.

Lojasiewicz, \cite{L}, also proved the following inequality\footnote{Lojasiewicz called this inequality the gradient inequality.}: With the same assumptions on $f$, for every $p\in U$, there is a possibly smaller neighborhood $W$ of $p$ and constants $\beta \in (0,1)$ and $C> 0$ such that for all $x\in W$
 \begin{align}	\label{e:016}
  |f(x)-f(p)|^{\beta}\leq C\, |\nabla_x f | \, .
  \end{align}
  Note that this inequality is trivial unless $p$ is a critical point for $f$.
  
 An immediate consequence of \eqr{e:016} is that every critical point of $f$
  has a neighborhood where 
  every other critical point has the same value.{\footnote{This consequence of \eqr{e:016} for the
  $F$ functional near a cylinder is  implied by the rigidity result of \cite{CIM}.}}

\subsection{Lojasiewicz inequalities for non-compact hypersurfaces and MCF} 
The  infinite dimensional Lojasiewicz-type inequalities that we prove are for the $F$ functional on the space of hypersurfaces.

The $F$-functional   is given by integrating the Gaussian over a hypersurface $\Sigma\subset \RR^{n+1}$.  This is also often referred to as the Gaussian surface
area and is defined by
\begin{align}
  F (\Sigma) = (4\pi)^{-n/2} \, \int_{\Sigma} \,
  \e^{-\frac{|x|^2}{4}} \, d\mu \, .
\end{align}
The entropy $\lambda (\Sigma) $ is the supremum of the Gaussian surface areas
 over all centers and scales.

It follows from the first variation formula that the gradient of $F$ is 
\begin{align}	\label{e:gradF}
   \nabla_{\Sigma} F (\psi) = \int_{\Sigma} \, \left( H - \frac{\langle x , \nn \rangle}{2} \right)
   \, \psi \, \e^{ - \frac{|x|^2}{4} } \, .
\end{align}
Thus, the critical points of $F$ are shrinkers, i.e., hypersurfaces with $H = \frac{\langle x , \nn \rangle}{2} $.   
The most 
important shrinkers are the  generalized cylinders $\cC$; these are the generic ones by \cite{CM1}.
The space $\cC$ is the union of $\cC_k$ for $k \geq 1$, where 
 $\cC_{k}$ is the space of cylinders $\SS^k \times \RR^{n-k}$,
 where the $\SS^k$ is centered at $0$ and has radius $\sqrt{2k}$ and we allow all 
 possible rotations by $SO(n+1)$.

 \vskip2mm
A family of hypersurfaces $\Sigma_s$ evolves by the negative gradient flow for the $F$-functional 
if it satisfies  the equation 
\begin{align}
	(\partial_s x)^\perp= -H\,\nn +
x^\perp/2 \, .
\end{align}
 This flow is called the rescaled MCF since $\Sigma_s$ is 
  obtained   
from a MCF $M_t$ by setting $\Sigma_s=\frac{1}{\sqrt{-t}}M_t$,
$s=-\log(-t)$, $t<0$.  By \eqr{e:gradF},
critical points for the $F$-functional or, equivalently, stationary points for the rescaled 
MCF, are the shrinkers for the MCF that become extinct at the origin in space-time.    
A rescaled MCF has a unique asymptotic limit if and only if
the corresponding MCF has a unique tangent flow at that singularity.

\vskip2mm
We will prove versions of the two Lojasiewicz inequalities for the $F$ functional on a general hypersurface $\Sigma$.  Roughly speaking, 
we will show that
\begin{align}
   \dist (\Sigma , \cC)^2 &\leq C \, \left|\nabla_{\Sigma} F \right|\, ,   \label{e:07}  \\
    (F(\Sigma)-F(\cC))^{\frac{2}{3}} &\leq C \,\left|\nabla_{\Sigma} F \right| \, .   \label{e:08}
\end{align}
Equation \eqr{e:07} will correspond to Lojasiewicz's first inequality whereas \eqr{e:08} will correspond to his second inequality.  The precise statements of these inequalities will be much more complicated than this, but they will be of the same flavor.

  \subsection{First Lojasiewicz with $\alpha=2$ implies the second with $\beta=\frac{2}{3}$}  \label{ss:1to2}

In this subsection we will explain how the second Lojasiewicz inequality for a 
function $f$ in a neighborhood of an isolated critical point
follows from the first when the first holds  for $\nabla f$ and with $\alpha=2$.  
(We will later extend this argument to infinite dimensions.)  

Suppose that $f:\RR^n\to \RR$ is  smooth function with $f(0)=0$ and $\nabla f(0)=0$; 
without loss of generality we may assume that at $0$ the Hessian is in diagonal form 
and we will write the coordinates as $x=(y,z)$ where $y$ are the coordinates where the 
Hessian is nondegenerate.  By Taylor's formula in a small neighborhood of $0$, we have that
\begin{align}
f(x)&=\frac{a_i}{2}\, y_i^2+O(|x|^3)\, .\\
f_{y_i}(x)&=a_i\, y_i+O(|x|^2)\, .\\
f_{z_i}(x)&=O(|x|^2)\, .
\end{align}
It follows from this that the second of the two Lojasiewicz 
inequalities holds for $f$ and $\beta =\frac{2}{3}$ provided that $|z|^2\leq \epsilon\, |y|$ 
for some sufficiently small $\epsilon> 0$.  Namely, if $|z|^2\leq \epsilon\, |y|$, then
\begin{align}
  C \, |y| \leq	 |\nabla_x f| {\text{ and }} |f(x)|\leq C^{-1} \,|y|^{\frac{3}{2}}
\end{align}
for some positive constant $C$ and, hence,
%\begin{align}
%|f(x)|\leq C\,|y|^{\frac{3}{2}}
%\end{align}
%hence
\begin{align}
|f(x)|^{\frac{2}{3}}\leq C\,|\nabla_x f |\, .
\end{align}
Therefore, we only need 
to prove the second Lojasiewicz inequality for $f$ 
in the region  $|z|^2\geq \epsilon\, |y|$.  We will do this
using
 the first  Lojasiewicz inequality for $\nabla f$. 
Since $0$ is an isolated critical point for $f$, 
the first Lojasiewicz inequality for $\nabla f$ gives that 
\begin{align}
|\nabla_x f|\geq C\, |x|^2\, .
\end{align}
By assumption on the region and the Taylor expansion for $f$, we get that in this region 
\begin{align}
|f(x)|\leq C\,|y|^2+C\,|z|^3\leq C\,|z|^3\leq C\,|x|^3\, .
\end{align}
 Combining these two inequalities gives
 \begin{align}
|f(x)|^{\frac{2}{3}}\leq C\,|x|^2\leq |\nabla_x f|\, .
\end{align}
This proves the second Lojasiewicz inequality for $f$ with $\beta = \frac{2}{3}$.

\vskip2mm
In Section \ref{s:Loja2},  we extend the above argument to general Banach spaces.

\vskip2mm
 Lojasiewicz used his second inequality to show the ``Lojasiewicz theorem'':  If $f:\RR^n\to \RR$ is an analytic function, $x=x(t):[0,\infty)\to \RR^n$ is a curve with $x'(t)=-\nabla f$ and $x(t)$ has a limit point $x_{\infty}$, then the length of the curve is finite and $\lim_{t\to \infty}x(t)=x_{\infty}$.  Moreover, $x_{\infty}$ is a critical point for $f$.
 
 Even in $\RR^2$, it is easy to construct smooth functions where the Lojasiewicz theorem does not hold, but instead there are negative gradient flow lines with multiple limits.

  We will discuss the Lojasiewicz theorem in a slightly more general setting at the end of the next subsection after briefly discussing 
  infinite dimensional Lojasiewicz inequalities.

 \subsection{Infinite dimensional Lojasiewicz inequalities and applications}
 Infinite dimensional versions of Lojasiewicz inequalities were
 proven in a celebrated work of Leon Simon, \cite{Si1},
 for the area and related functionals and
 used, in particular,  to prove a fundamental result about
 uniqueness of tangent cones with smooth cross section of minimal surfaces.  
  Simon's proof of the Lojasiewicz inequality is done by reducing the infinite
 dimensional version to the classical Lojasiewicz inequality by a Lyapunov-Schmidt
  reduction argument.  Infinite dimensional  Lojasiewicz
 inequalities proven using  Lyapunov-Schmidt 
 reduction, as in the work of Simon,   have had a profound impact on various areas of analysis and geometry and are 
 usually referred to as Lojasiewicz-Simon inequalities.
 
 As already mentioned, we will also prove two infinite dimensional Lojasiewicz
 inequalities and use them to prove uniqueness of blowups for MCF
 (or, equivalently, convergence of the rescaled flow). However, unlike all
 other infinite dimensional Lojasiewicz inequalities we know of, ours do not
 follow from a  reduction   to the classical 
 Lojasiewicz inequalities; rather we prove our inequalities directly and do not
 rely on Lojasiewicz's arguments or results.  In fact, we prove our infinite
 dimensional analog of the first Lojasiewicz inequality directly and use this 
 together with an infinite dimensional analog of the argument in the previous subsection to show our
  second Lojasiewicz inequality.   The reason why we cannot argue as in   Simon's work, and all the later work that make use his ideas, comes from that our singularities are noncompact.  In particular, even near the singularities, the evolving hypersurface cannot be written as an entire graph over the singularity.  Rather, only part of the evolving hypersurface can be written as a graph 
over a compact piece of the singularity.

 \vskip2mm
 Next we will   explain how the second Lojasiewicz inequality is 
 typically used to show uniqueness.  
 Before we do that, observe first that in the second inequality we always work in 
 a small neighborhood of $p$ so that, in particular, $ |f(x)-f(p)|\leq 1$ and 
 hence smaller powers on the left hand side of the inequality imply the inequality for higher powers.  
 As it turns out, we will see that any positive power strictly less than $1$ would do for uniqueness. 
 
 Suppose now that $X$ is a Banach space and $f:X\to \RR$ is a Frechet differentiable function. 
 Let $x=x(t)$ be a curve on $X$ parametrized on $[0,\infty)$ whose velocity $x'=-\nabla f$. 
 We would like to show that if the second inequality of Lojasiewicz holds for $f$ with a power
 $1>\beta> 1/2$, then the Lojasiewicz theorem mentioned above holds.  That is, if $x(t)$ has a limit 
 point $x_{\infty}$, then the length of the curve is finite and $\lim_{t\to \infty}x(t)=x_{\infty}$.  
Since $x_{\infty}$ is a 
 limit point of $x(t)$ and $f$ is non-increasing along the curve, $x_{\infty}$ 
must be a critical point for $f$.   

To see that $x(t)$ converges to $x_{\infty}$, assume that $f(x_{\infty})=0$ and note that if we 
set $f(t)=f(x(t))$, then $f'=-|\nabla f|^2$.  Moreover, by the second Lojasiewicz 
inequality, we get that $f' \leq - f^{2\beta}$ if $x(t)$ is sufficiently close to
$x_{\infty}$.  (Assume for simplicity below that $x(t)$ stays
in a small neighborhood $x_{\infty}$ for $t$ sufficiently large so that this inequality holds; the general case follows with trivial changes.)  Then this inequality can be rewritten as $(f^{1-2\beta})' \geq (2\beta-1)$ 
which integrates to
\begin{align}    \label{e:Ldecay}
  f(t) \leq C\, t^{\frac{-1}{2\beta-1}} \, .
\end{align}

We need to show that \eqr{e:Ldecay} implies that $\int_{0}^{\infty}|\nabla f| \,ds$ is finite.  
This shows that $x(t)$ converges to $x_{\infty}$ as $t\to \infty$.  
To see that $\int_{0}^{\infty}|\nabla f| \,ds$ is finite, observe  by the Cauchy-Schwarz inequality that
\begin{align}
  \int_{0}^{\infty}|\nabla f|\,ds =\int_{0}^{\infty}\sqrt{-f'}\,ds
    \leq \left(-\int_{0}^{\infty}f'\, s^{1+\epsilon}\,ds\right)^{\frac{1}{2}}
       \left(\int_{0}^{\infty} s^{-1-\epsilon}\,ds\right)^{\frac{1}{2}} \, .
\end{align}
It suffices therefore to show that
\begin{align}
-\int_{0}^{T} f'\, s^{1+\epsilon}\,ds
\end{align}
is uniformly bounded.  
Integrating by parts gives
\begin{align}
    \int_{0}^{T} f'\, s^{1+\epsilon}\,ds=|f\, s^{1+\epsilon}|^{T}_0-(1+\epsilon)
     \int_{0}^{T} f\, s^{\epsilon}\,ds\, .
\end{align}
If we choose $\epsilon> 0$ sufficiently small depending on $\beta$, 
then we see that this is bounded independent of $T$ and hence $\int_{0}^{\infty}|\nabla f|\,ds$ is finite.

\vskip2mm
We will use an extension of this argument where the assumption $ f^{2\beta}(t)\leq -f'(t)$ is replaced by the assumption that $f^{2\beta}(t)\leq f(t-1)-f(t+1)$; see Lemma \ref{l:discreteL}.  This assumption is exactly what comes out of our analog for the rescaled MCF of the gradient Lojasiewicz inequality, i.e., out of Theorem \ref{t:ourgradloja}.

\subsection{The two Lojasiewicz inequalities}  We will now state the two Lojasiewicz-type inequalities
for the $F$ functional on the space of hypersurfaces.

\vskip2mm
Suppose that $\Sigma\subset \RR^{n+1}$ is a hypersurface and fix
some small $\epsilon_0 > 0$ (this will be chosen small enough to satisfy Lemmas \ref{l:ift}  %, \ref{l:rotate}, 
and \ref{l:frechet2}).   Given an integer $\ell$ and constant $C_{\ell}$, we let $\graph (\Sigma)$ be the maximal radius so that
\begin{itemize}
 \item  $B_{\graph (\Sigma)} \cap \Sigma$ is the graph over a cylinder in
 $\cC_k$
 of a function $u$  with $\| u \|_{C^{2,\alpha}} \leq \epsilon_0$ and $|\nabla^{\ell} A| \leq C_{\ell}$.
\end{itemize}
The parameters $\ell$ and $C_{\ell}$ will be left free until the proof of the main theorem (Theorem \ref{t:main})
and will then be chosen large.

In the next theorem,  we will use a Gaussian $L^2$ distance 
$d_{\cC}(R)$ to the space $\cC_k$ in the ball of radius $R$.    To define this, given $\Sigma_k  \in \cC_k$,  let $w_{\Sigma_k}: \RR^{n+1}\to \RR$ 
denote the distance to the axis of $\Sigma_k$ (i.e., to the space of translations that leave $\Sigma_k$ invariant).  Then we define
 \begin{align}
   d_{\cC}^2(R) = \inf_{\Sigma_k \in \cC_k} \,  \| w_{\Sigma_k} - \sqrt{2k} \|_{L^2(B_R)}^2 \equiv \inf_{  \Sigma_k \in \cC_k} \, 
      \int_{B_{R} \cap \Sigma_k} ( w_{\Sigma_k} - \sqrt{2k})^2\, \e^{-\frac{|x|^2}{4}}\, .
\end{align}
The Gaussian $L^p$ norm on the ball $B_R$ is
$\| u \|^p_{L^p(B_R)} = \int_{B_R} |u|^p \,  \e^{ - \frac{|x|^2}{4} }$.

Given a general hypersurface $\Sigma$, it is also convenient to define the function $\phi$ by
\begin{align}  
 \phi = \frac{\langle x , \nn \rangle}{2}-H \, ,
\end{align}
so that $\phi$ is minus the gradient of the functional $F$.

\vskip2mm
The main tools that we develop here are the
following two analogs for non-compact hypersurfaces of 
the well-known Lojasiewicz's  inequalities for analytic   functions on $\RR^n$.

\begin{Thm}  \label{t:ourfirstloja}
(A Lojasiewicz inequality for non-compact hypersurfaces). 
If $\Sigma \subset \RR^{n+1}$ is a hypersurface
with $\lambda (\Sigma) \leq \lambda_0 $  and $R \in [ 1 ,  \graph (\Sigma) - 1]$, then 
\begin{align}	\label{e:filoja}
	d^2_{\cC} (R) \leq     C \,  R^{\rho}\, \left\{ 
        \|  \phi \|_{L^1(B_{ R} )}^{ b_{\ell,n}  }   
       +  \e^{ - \frac{ b_{\ell , n} \, R^2}{4}   } \right\}  \, ,
\end{align}
where $C = C(n,\ell, C_{\ell}, \lambda_0)$,  $\rho = \rho (n)$ and  $b_{\ell , n} \in (0,1)$ satisfies $\lim_{\ell \to \infty} \, b_{\ell , n} = 1$.
\end{Thm}

%%old imprecise version of the theorem follows:
%\begin{Thm}  \label{t:ourfirstloja}
%(A Lojasiewicz inequality for non-compact hypersurfaces).  If $\Sigma \subset \RR^{n+1}$ is a hypersurface
%with $\lambda (\Sigma) \leq \lambda_0 $, then  
%\begin{align}	\label{e:filoja0}
%  d_{\cC}^2(\bfR) \leq C \, \gamma^{ b_{\ell , n} } \, \left( 1 + |\log \gamma |
%    \right)^{ n + \frac{9 }{ 2}}  \, ,
%\end{align}
%where the constant $C$ depends on $n$ and $\epsilon_0$, $\ell$ and $C_{\ell}$ in the definition of $\graph (\Sigma)$, 
%the exponent $b_{\ell , n} \in (0,1)$ satisfies $\lim_{\ell \to \infty} b_{\ell , n} = 1$,
%$\gamma$ is defined by
% \begin{align}
 %   \gamma = s^{-2}  \,  \lambda_0 \, [\graph (\Sigma)]^n \, \e^{ -
%      \frac{(\graph (\Sigma) -s)^2}{4} }  +       \| \Hess_{\phi} \|_{L^1(B_{\graph (\Sigma)})} +
%        \|  \phi \|_{L^1(B_{ \graph (\Sigma)} )}
%%        \, ,
 % \end{align} 
%  and the radius $\bfR$ goes to infinity as $\gamma \to 0$ and $\graph (\Sigma) \to \infty$.
  %\end{Thm}

%\vskip2mm
%We will state this theorem, including the exact dependence of $\bfR$, more precisely later in the paper (see Theorem \ref{t:ourfirstloja2}).

\vskip2mm
The   theorem  bounds the $L^2$ distance to $\cC_k$ by a
power of  $\| \phi \|_{L^1}$, with an error term that comes from  
a cutoff argument since $\Sigma$ is non-compact and is not globally a 
graph of the cylinder.{\footnote{This is a Lojasiewicz inequality for 
the gradient of the $F$ functional ($\phi$ is the gradient of $F$).  This follows since, by \cite{CIM}, cylinders 
are isolated critical points for $F$ and, thus, $d_{\cC}$ 
locally measures the distance to the nearest critical point.}} 
This theorem is essentially sharp. 
Namely, the estimate \eqr{e:filoja} does not hold 
for any exponent $b_{\ell , n}$  
larger than one, but Theorem \ref{t:ourfirstloja} lets 
us take $b_{\ell , n}$ arbitrarily close to one.

%\vskip2mm
%Here and below we use the notation that
%\begin{align}  \label{e:}
% 	\left|\nabla_{\Sigma} F \right|^2_{B_{s}} &=\int_{B_{s}\cap \Sigma} \phi^2\,\e^{-\frac{|x|^2}{4}}\, . 
% \end{align}
%(Without the subindex $B_{s}$ means that we are integrating over all of $\Sigma$.)

\vskip2mm
We will also see that the above inequality implies the following gradient type Lojasiewicz inequality.
  This inequality  bounds the difference of the $F$ functional near a critical point by two terms.  The first 
is essentially a power of $\nabla F$, while the second (exponentially decaying) term comes from
 that $\Sigma$ is not a graph over the entire cylinder.
 
 %%%here old version follows - revised 9/12/13
%  \begin{Thm}  \label{t:ourgradloja}
%(A gradient Lojasiewicz inequality for non-compact hypersurfaces).
%If $\Sigma \subset \RR^{n+1}$ is a hypersurface
%with $\lambda (\Sigma) \leq \lambda_0 $ and $\beta \in [0,1)$, then 
%\begin{align}	\label{e:fingloja}
%    \left| F(\Sigma) - F(\cC_k) \right| 
%      &\leq C 
%   \| \phi \|_{L^2 (B_{\tilde{R}})}^{\frac{3+\beta}{2}}
%	  + C  \, \e^{ - \frac{(3+\beta)(\tilde{R} -1)^2}{16}} +
%      C \, \left\{ \gamma^{\frac{ b_{\ell , n}}{2} } \, \left( 1 + |\log \gamma |
%    \right)^{ \frac{n+5}{2} } \right\}^{ \frac{3+\beta}{1+\beta}} \, , 
%\end{align}
%  where $\tilde{R}$, $b_{\ell,n}$ and $\gamma$ are as in Theorem \ref{t:ourfirstloja2}  and the constant
 % $C$ depends on $n, \epsilon_0 , \ell$ and $C_{\ell}$.
  %\end{Thm}

 \begin{Thm}  \label{t:ourgradloja}
(A gradient Lojasiewicz inequality for non-compact hypersurfaces).
If $\Sigma \subset \RR^{n+1}$ is a hypersurface
with $\lambda (\Sigma) \leq \lambda_0 $,  $\beta \in [0,1)$,  and $R \in [ 1 ,  \graph (\Sigma) - 1]$, then
\begin{align}	\label{e:fingloja} 
	 \left| F (\Sigma) - F (\cC_k) \right| 
      &\leq  C \,  R^{\rho}\, \left\{ 
        \|  \phi \|_{L^2(B_{ R} )}^{ c_{\ell,n} \, \frac{3+\beta}{2+2\beta} }   
       +  %%\e^{ - \frac{ R^2}{4} \, \min \{ c_{\ell , n} , \frac{3+\beta}{4} \} }
         \e^{ - \frac{c_{\ell,n} \,(3+\beta) R^2}{8(1+\beta)}}    +  
	 \e^{ - \frac{(3+\beta)(R -1)^2}{16} }
       \,
        \right\}  \, ,
\end{align}
where $C = C(n,\ell, C_{\ell}, \lambda_0)$,  $\rho = \rho (n)$ and  $c_{\ell , n} \in (0,1)$ satisfies $\lim_{\ell \to \infty} \, c_{\ell , n} = 1$.
\end{Thm}

When we apply the theorem, the parameters $\beta$ and $\ell$ will be chosen to make the exponent  greater
than one on  the $\nabla F$ term, essentially giving that $\left| F(\Sigma) - F(\cC_k) \right|$
is bounded by a power greater than one of $|\nabla F|$.  A separate argument will be needed to handle the 
exponentially decaying error terms.

\vskip2mm
Throughout the paper $C$ will denote a constant that can change from line to line.

\vskip2mm
We will show that when $\Sigma_t$ are  flowing by the rescaled MCF, then both terms on the right-hand side of
\eqr{e:fingloja} are bounded by a power greater than one of $\| \phi \|_{L^2}$ (the corresponding statement   holds for 
Theorem \ref{t:ourfirstloja}).  Thus, we will essentially get the inequalities
\begin{align}
d_{\cC}^2 &\leq C \, \left|\nabla_{\Sigma_t} F \right|\, ,\\
  (F(\Sigma_t)-F(\cC))^{\frac{2}{3}} &\leq C \,\left|\nabla_{\Sigma_t} F \right| \, .
\end{align}
These two inequalities can be thought of as analogs for the rescaled MCF of Lojasiewicz inequalities from real algebraic geometry; cf. \eqr{e:07} and
\eqr{e:08}.

\vskip2mm
See \cite{CM6} for a survey on Lojasiewicz inequalities and their applications.

  \section{Cylindrical estimates for a general hypersurface}	\label{s:gensi}
  
  In this section, we will prove estimates for a general hypersurface $\Sigma \subset \RR^{n+1}$.  The main results  are bounds for $\nabla \frac{A}{H}$
  when the mean curvature $H$ is positive on a large set.
  
 \subsection{A general Simons equation}

In this subsection, we will show that the second fundamental form $A$ of   $\Sigma $ satisfies an elliptic differential equation similar to Simons' equation for minimal surfaces.    The elliptic operator will be
the $L$ operator from \cite{CM1} given by
\begin{align}
	L \equiv \cL + |A|^2 + \frac{1}{2} \equiv \Delta - \frac{1}{2} \, \nabla_{x^T} + |A|^2 + \frac{1}{2} \, .
\end{align}
Namely, we will prove the following proposition:

\begin{Pro}	\label{l:gensimons}
 If  $\phi = \frac{1}{2} \langle x , \nn \rangle - H$, then
 \begin{align}
    L \, A = A + \Hess_{\phi} + \phi \, A^2 \, , 
 \end{align}
where the tensor $A^2$ is given in orthonormal frame by $\left( A^2 \right)_{ij} = A_{ik} \, A_{kj}$.
\end{Pro}

Note that $\phi$ vanishes precisely when $\Sigma$ is a shrinker and, in this case, we recover the Simons' 
equation for $A$ for shrinkers from \cite{CM1}.

\vskip2mm
We will use the following general version of Simons' equation for the second fundamental form of a hypersurface:

\begin{Lem}	\label{l:sim1}
The second fundamental form  $A$ satisfies
\begin{align}
	\left( \Delta + |A|^2\right) \, A = - H \, A^2 - \Hess_H \, .
\end{align}
\end{Lem}

See, e.g., \cite{CM3} for a proof.  %% see also page 264 of Leon's GMT; his A has opposite sign

\vskip2mm
The next lemma computes the Hessian of the support function $\langle x , \nn \rangle$.

\begin{Lem}	\label{l:support}
The Hessian of $\langle x , \nn \rangle$ is given by
\begin{align}
 \Hess_{\langle x , \nn \rangle} = -\nabla_{x^T} A - A - A^2 \, \langle x , \nn  \rangle \, .
\end{align}
\end{Lem}

\begin{proof}
 Fix a point $p \in \Sigma$.  Let $e_i$ be a local orthonormal frame for $\Sigma$ with 
 $\nabla_{e_i}^T e_j = 0$ at $p$ for every $i$ and $j$. Thus, at $p$, we have
 \begin{align}
  \nabla_{e_i} e_j = A_{ij} \, \nn \, .
  \end{align}
  Finally, using this and $\nabla_{e_i} \nn = - A_{ik} \, e_k$ (which holds at all points), 
 we compute at $p$ 
 \begin{align}
    \Hess_{\langle x , \nn \rangle} (e_i , e_j) &=  \langle x , \nn \rangle_{ij} =
    \langle x , \nabla_{e_i} \nn \rangle_j
    = -\left( A_{ik} \, \langle x , e_k \rangle \right)_j \notag \\
    &=  -A_{ikj} \, \langle x , e_k \rangle - A_{ik} \, \delta_{jk} - A_{ik} \langle x , A_{jk} \, \nn \rangle \\
    &= -\left( \nabla_{x^T} A\right) (e_i , e_j) - A(e_i , e_j) -
    \langle x , \nn \rangle A^2 (e_i , e_j) \, , \notag
    \end{align}
   where the last equality used the Codazzi equation $A_{ikj} = A_{ijk}$.
\end{proof}

\begin{proof}[Proof of Proposition \ref{l:gensimons}]
Since $L = \cL + |A|^2 + \frac{1}{2}$ and $\cL = \Delta - \frac{1}{2} \, \nabla_{x^T}$, Lemma \ref{l:sim1} gives
\begin{align}
	L A &= \left( \Delta + |A|^2 \right) \, A + \frac{1}{2} \, A -
	\frac{1}{2} \, \nabla_{x^T} A  =
	- H \, A^2 - \Hess_H + \frac{1}{2} \, A  - \frac{1}{2} \, \nabla_{x^T} A 
	 \, .
\end{align}
On the other hand,  Lemma \ref{l:support} gives
\begin{align}
  \Hess_{\phi} =   \frac{1}{2} \Hess_{\langle x , \nn \rangle} - \Hess_H =
  -\Hess_H - \frac{1}{2} \nabla_{x^T} A - \frac{1}{2} A - \frac{1}{2} A^2 \, \langle x , \nn \rangle \, ,
\end{align}
so we have
$
    L \, A - \Hess_{\phi}= A  + \phi \, A^2  $.
  \end{proof}

\subsection{An integral bound when the  mean curvature is positive}

We will show
that the tensor $\tau = A/H$ is almost parallel when $H$ is positive and $\phi$ is small.  This generalizes an estimate from
\cite{CIM} in the case where $\Sigma$ is a shrinker (i.e., $\phi \equiv 0$) with $H> 0$.
  
  \vskip2mm
Given $f>0$, define a weighted divergence operator $\dv_f$ and drift
Laplacian $\cL_f$ by
\begin{align}
  \dv_f (V) &= \frac{1}{f} \, \e^{ |x|^2/4 } \, \dv_{\Sigma}
  \, \left( f \, \e^{ -|x|^2/4 } \, V \right) \, ,
  \\
  \cL_f \, u &\equiv \dv_f (\nabla u) =\cL \, u + \langle \nabla \log
  f , \nabla u \rangle \, .
\end{align}
Here $u$ may also be a tensor; in this case the divergence traces only
with $\nabla$. Note that $\cL=\cL_1$. We recall the quotient rule (see lemma $4.3$ in \cite{CIM}):

\begin{Lem} \label{l:quotr} Given a tensor $\tau$ and a function $g$
  with $g \ne 0$, then
  \begin{align}
    \cL_{g^2} \, \frac{\tau}{g} & = \frac{ g \, \cL \, \tau - \tau \,
      \cL \, g}{g^2} = \frac{ g \, L \, \tau - \tau \, L \, g}{g^2}
    \, .
  \end{align}
\end{Lem}

\begin{Pro} \label{p:eqnsg} 
 On the set where $H > 0$, we have
  \begin{align}
    \cL_{ H^2} \, \frac{ A}{H} &=    \frac{   \Hess_{\phi} + \phi \, A^2}{H}  + 
      \frac{A \, \left(   \Delta \phi + \phi \, |A|^2 \right) }{H^2} \, , \\
    \cL_{ H^2} \, \frac{ |A|^2}{H^2} &= 2\, \left|\nabla \frac{A}{H}
    \right|^2  + 2 \, \frac{   \langle    \Hess_{\phi} + \phi \, A^2  , A \rangle}{H^2}   + 2\, \frac{|A|^2 \, \left(   \Delta \phi + \phi \, |A|^2 \right)    }{H^3} \, .
  \end{align}
\end{Pro}

\begin{proof}
The trace of 
Proposition 
\ref{l:gensimons} ($H$ is minus the trace of $A$ by   convention) gives
 \begin{align}
    L \, H &= H - \Delta \phi - \phi \, |A|^2  \, ,
 \end{align}
 where we also used that the trace of $A^2$ is $|A|^2$ since $A$ is symmetric.
 Using the 
  quotient rule (Lemma \ref{l:quotr}) and the equations for $L H $ and $LA$ (from Proposition 
\ref{l:gensimons}) gives 
  \begin{align}
  	  \cL_{ H^2} \, \frac{ A}{H} & = \frac{ H \, L A - A \, L H}{H^2} = 
	  \frac{ H \, \left( A + \Hess_{\phi} + \phi \, A^2 \right)  - A \, \left( H - \Delta \phi - \phi \, |A|^2 \right) }{H^2} \notag \\
	  &= \frac{   \Hess_{\phi} + \phi \, A^2}{H}  + \frac{A \, \left(   \Delta \phi + \phi \, |A|^2 \right) }{H^2} \, ,
  \end{align}
  giving the first claim.
   The second claim
  follows from the first since $\frac{ |A|^2}{H^2} = \langle \frac{
    A}{H} , \frac{ A}{H} \rangle$ and 
    \begin{align}
    	\frac{1}{2} \, \cL_{H^2}  \langle \frac{
    A}{H} , \frac{ A}{H} \rangle =  \langle \cL_{H^2} \, \frac{
    A}{H} , \frac{ A}{H} \rangle + \left| \nabla \frac{A}{H} \right|^2 \, .
    \end{align}
\end{proof}

The next proposition gives exponentially decaying integral bounds for
$\nabla (A/H)$ when $H$ is positive on a large ball.  It will be
important that these bounds decay rapidly.

\begin{Pro} \label{p:effective} If $B_R \cap \Sigma$ is smooth with $H
  > 0$, then for $s \in (0 , R)$ we have
  \begin{align}
    \int_{B_{R-s} \cap \Sigma} \left| \nabla \frac{A}{H} \right|^2 \,
    H^2 \, \e^{-  \frac{|x|^2}{4} } &\leq \frac{4}{s^2} \, \sup_{B_R
      \cap \Sigma} |A|^2 \, \Vol (B_R \cap \Sigma) \, \e^{ -
      \frac{(R-s)^2}{4} } \notag  \\
      &\quad + 2\, \int_{B_R \cap \Sigma}  
      \left\{  \left|  \langle \Hess_{\phi} , A \rangle
         + \frac{|A|^2}{H} \, \Delta \phi \right|
      + \left| \langle A^2 , A \rangle  + \frac{|A|^4}{H} \right| \, |\phi|
      \right\}
      \e^{- \frac{|x|^2}{4} } \, .
  \end{align}
%  where $\phi \equiv \frac{1}{2} \, \langle x , \nn \rangle - H$.
\end{Pro}

\begin{proof}
  Set $\tau= A/H$ and $u = | \tau |^2 = |A|^2/H^2$.  
     It will be convenient within this proof to use square brackets $\left[ \cdot \right]$ to denote Gaussian integrals 
  over $B_R \cap \Sigma$, i.e. $\left[ f \right] = \int_{B_R \cap \Sigma } f \, \e^{-|x|^2/4}$.

    Let $\psi$ be a function with support in
  $B_R$. Using the divergence theorem, the formula from  Proposition
  \ref{p:eqnsg}  for
  $\cL_{H^2}u$, and the absorbing inequality $4ab \leq a^2 + 4 b^2$, we get
  \begin{align}
    0 &= \left[  \dv_{H^2} \, \left( \psi^2 \, \nabla u \right) \, H^2  \left] 
    = \right[ \left(\psi^2 \, \cL_{H^2}u + 2 \psi
      \langle \nabla \psi , \nabla u \rangle \right)
    \, H^2 \right]  \notag   \\
    &= \left[ \left\{ 2\, \psi^2 \, \left| \nabla \tau \right|^2 + 2 \, \psi^2 \left( 
     \frac{   \langle    \Hess_{\phi} + \phi \, A^2  , A \rangle}{H^2}   +   \frac{|A|^2 \, \left(   \Delta \phi + \phi \, |A|^2 \right)    }{H^3}   \right) 
    + 4 \psi \langle \nabla \psi , \tau\cdot\nabla \tau
      \rangle \right\}
    \, H^2 \right]  \notag  \\
    &\geq \left[  \left(\psi^2 \, \left| \nabla \tau
      \right|^2- 4\, |\tau|^2 \, \left| \nabla \psi \right|^2 \right) \,
    H^2 \, \right]  + 2 \, \left[ \psi^2  \langle    \Hess_{\phi} + \phi \, A^2  , A \rangle \right] 
    +2 \, \left[ \psi^2 \frac{|A|^2 \, \left(   \Delta \phi + \phi \, |A|^2 \right)    }{H}  
    \right] \, , \notag
  \end{align}
  from which we obtain
  \begin{align}
    \left[ \psi^2 \, \left| \nabla \tau
      \right|^2    \, H^2 \right] 
   & \leq 4\, \left[ \left| \nabla \psi \right|^2 \, |A|^2 \, 
    \right]  - 2 \, \left[ \psi^2 \,  \langle  \Hess_{\phi} + \phi A^2, A \rangle    
     \right] 
     % \\   &\qquad
    - 2 \, \left[\psi^2 \,  \Delta \phi \, \frac{|A|^2}{H} + \psi^2 \, \phi \frac{|A|^4}{H} \right] 
    \, . \notag 
  \end{align}
  The proposition follows by choosing $\psi \equiv 1$ on $B_{R-s}$ and
  going to zero linearly on $\partial B_{R}$.
\end{proof}

We record the following corollary:

\begin{Cor} \label{c:effective} If $B_R \cap \Sigma$ is smooth with $H> \delta
  > 0$ and $|A| \leq C_1$, then there exists $C_2 = C_2 (n, \delta , C_1)$ so that for $s \in (0 , R)$ we have
  \begin{align}
     \int_{B_{R-s} \cap \Sigma} \left| \nabla \frac{A}{H} \right|^2  
   \, \e^{-  \frac{|x|^2}{4} } &\leq \frac{C_2}{s^2} \,  \Vol (B_R \cap \Sigma) \, \e^{ -
      \frac{(R-s)^2}{4} }  +  C_2 \, \int_{B_R \cap \Sigma}  \left\{   \left| \Hess_{\phi} \right| +  |  \phi|
         \right\}
      \e^{- \frac{|x|^2}{4} } \, .
  \end{align}
%  where $\phi \equiv \frac{1}{2} \, \langle x , \nn \rangle - H$.
\end{Cor}

\begin{Rem}
Corollary \ref{c:effective} essentially bounds the distance {\emph{squared}} to the space of cylinders by $\|\phi \|_{L^1}$.
This is sharp: it is not possible to get the sharper
bound where the powers are the same.  This is a general fact when  there is a non-integrable kernel.  Namely,
 if we perturb  in the direction of the kernel, then  $\phi$ vanishes
 quadratically in the distance.
\end{Rem}

The next corollary combines the Gaussian $L^2$ bound on $\nabla \tau$ from 
Corollary \ref{c:effective}  with standard interpolation inequalities to get 
pointwise bounds on $\nabla \tau$ and $\nabla^2 \tau$.

\begin{Cor}	\label{c:ptwise}
 If $B_R \cap \Sigma$ is smooth with $H> \delta
  > 0$,  $|A| + \left|\nabla^{\ell+1} A \right| \leq C_1$,  and $\lambda (\Sigma) \leq \lambda_0$, then 
  there exists $C_3 = C_3 (n , \lambda_0 , \delta , \ell , C_1)$ so that
   for   $|y| + \frac{1}{1+|y|} < R-1$,  we have
 \begin{align}
    \left|\nabla \, \frac{A}{H} \right| (y) + \left|\nabla^2  \frac{A}{H}  \right|(y)&\leq C_3 \, 
    R^{2n} \, \left\{ \e^{ - d_{\ell , n } \, \frac{(R-1)^2}{8} }+ \| \phi \|_{L^1(B_R)}^{ \frac{d_{\ell , n}}{2} } 
    \right\}  \,    \e^{ \frac{|y|^2}{8} } \, , 
 \end{align}
where the exponent $d_{\ell,n} \in (0,1)$ has $\lim_{\ell \to \infty} d_{\ell,n} 
 =1$.

%  where $\phi \equiv \frac{1}{2} \, \langle x , \nn \rangle - H$.
\end{Cor}

\begin{proof}
  Set $\tau= A/H$ and note that $\left| \nabla^{\ell +1} \tau \right|$ is bounded by a constant depending on $\delta$, $\ell$ and $C_1$.  
  Define the ball $B^y$ and constant $\delta_y$ by
  \begin{align}
     B^y = B_{\frac{1}{1+|y|} } (y)  {\text{ and }}
     \delta_y =  \int_{B^y \cap \Sigma} \left| \nabla \tau \right|  \, .
  \end{align}
 Applying   Lemma \ref{l:interp1} on $B^y$ gives
 \begin{align}
   |\nabla \tau|(y) &\leq C' \, \left\{ R^{n} \, \delta_y + \delta_y^{a_{\ell,n}} \, 
    \| \nabla^{\ell+1} \tau \|_{L^{\infty}(B^y)}^{1-a_{\ell,n}}  \right\}
       \leq C \, \left\{ R^{n} \, \delta_y + \delta_y^{a_{\ell,n}}   \right\} \, , \notag \\
     |\nabla^2 \tau|(y) &\leq    C' \, \left\{ R^{n+1} \, \delta_y +  \delta_y^{b_{\ell,n}} \, 
    \| \nabla^{\ell+1} \tau \|_{L^{\infty}(B^y)}^{1-b_{\ell,n}}   \right\} 
      \leq C \, \left\{ R^{n+1} \, \delta_y + \delta_y^{b_{\ell,n}}   \right\}  \, ,  \notag
\end{align}
where 
the powers are given by $a_{\ell,n} = \frac{2\ell}{2\ell+n}$ and $b_{\ell,n} =  \frac{2\ell-2}{2\ell+n}$,
and  $C= C(n, \delta, \ell ,C_1)$.

 To get the bound on $\delta_y$, observe that
 \begin{align}
     \inf_{B^y} \e^{ - \frac{|x|^2}{4} } \geq \e^{ - \frac{|y|^2}{4} -1 } \, , 
 \end{align}
so that Cauchy-Schwarz gives
 \begin{align}
    \left( 1 + |y| \right)^{n} \,  \e^{  \frac{-|y|^2}{4} -1 } \, \delta_y^2 &\leq C \,   \e^{  \frac{-|y|^2}{4} -1 } \,
    \int_{B^y \cap \Sigma} \left| \nabla \tau \right|^2  \leq C\,  
     \int_{B^y \cap \Sigma} \left| \nabla \tau \right|^2 \, \e^{ - \frac{|x|^2}{4} }    \leq C_2 \, \gamma \, ,
 \end{align}
where the last inequality is 
  Corollary \ref{c:effective}, $C_2 = C_2 (n,\lambda_0 , \delta , C_1)$  
  and $\gamma$ is  
 \begin{align}	\label{e:heregamma}
    \gamma =     R^n \, \e^{ -
      \frac{(R-1)^2}{4} }  +    \int_{B_{R-1/2} \cap \Sigma}  \left\{   \left| \Hess_{\phi} \right| +  |  \phi|
         \right\}
      \e^{- \frac{|x|^2}{4} } \, .
  \end{align}
  To bound the Hessian  term, first choose  balls $B^i = B_{\frac{1}{1+|z_i|}}(z_i)$  so that
   \begin{itemize}
   \item 
   $B_{R-1/2} \cap \Sigma$ is contained in the union of the half-balls $\frac{1}{2} \, B^i$.
      \item Each point is in at most $c=c(n)$ of the balls.  
      \end{itemize}
      To simplify notation, set $r_i = \frac{1}{1+|z_i|}$.
 Applying   Lemma \ref{l:interp1} on $B^i$ gives
 \begin{align}
     \sup_{ \frac{1}{2} B^i} \, |\Hess_{\phi}| &\leq    C\, \left\{ r_i^{-n-2} \, \int_{B^i} |\phi| +   \left( \int_{B^i} |\phi| \right)^{c_{\ell,n}}  \right\} 
       \, ,  
\end{align}
where $c_{\ell,n} \in (0,1)$ goes to one as $\ell \to \infty$.  Note that the Gaussian weight has bounded oscillation on $B^i$ (this is why the radius $r_i$ was chosen).  It follows that
\begin{align}
	\int_{B_{R-1/2} \cap \Sigma}     \left| \Hess_{\phi} \right|        \e^{- \frac{|x|^2}{4} }  &\leq C \, 
	\sum \, \left\{  r_i^{-2} \, \int_{B^i} |\phi| +   r_i^n \left( \int_{B^i} |\phi| \right)^{c_{\ell,n}} 
	\right\} \, \e^{ - \frac{ |z_i|^2}{4} } \notag \\
	& \leq C \, R^2 \, \| \phi \|_{L^1(B_R )} + C \, 
	\sum     \left( \int_{B^i} |\phi| \right)^{c_{\ell,n}} 
      \, \e^{ - \frac{ |z_i|^2}{4} } \\
        &\leq C \, R^2 \, \| \phi \|_{L^1(B_R )} + C \, 
	\| \phi \|_{L^1(B_R )}^{c_{\ell,n}} 
         \, , \notag
\end{align}
where the last inequality uses the H\"older inequality for sums and the bound for $F(\Sigma)$.  Since $\| \phi \|_{L^1}$ is bounded (we are interested in the case where it is much less than one), the lower power is dominant and we conclude that
\begin{align}
	 \e^{  \frac{-|y|^2}{4} -1 } \, \delta_y^2 \leq C_2 \, \gamma \leq  C\, R^n \, \e^{ -
      \frac{(R-1)^2}{4} }  + C\,  R^2 \, \| \phi \|_{L^1(B_R )}^{c_{\ell,n}} \, .
\end{align}
Arguing similarly and using this in the bounds for $\nabla \tau$ gives
 \begin{align}
    \left|\nabla \, \tau \right| (y) &\leq C \,  R^n \, \delta_y^{a_{\ell, n}} \leq C \,  R^{ \frac{3n}{2}} \, \left\{ 
      \, \e^{ \frac{ |y|^2 - (R-1)^2 }{8} }  + \e^{ \frac{|y|^2}{8}} \,  \| \phi \|_{L^1(B_R )}^{ \frac{c_{\ell,n}}{2} } \right\}^{a_{\ell, n}}   \, , \\
    \left|\nabla^2  \tau \right|(y) & \leq C \,  R^{n+1} \, \delta_y^{b_{\ell, n}} \leq C \,  R^{ \frac{3n+2}{2}} \, \left\{ 
      \, \e^{ \frac{ |y|^2 - (R-1)^2 }{8} }  + \e^{ \frac{|y|^2}{8} }\, \| \phi \|_{L^1(B_R)}^{\frac{c_{\ell,n}}{2}} \right\}^{b_{\ell, n}} \, .
 \end{align}

\end{proof}

\section{Distance to  cylinders and the first Lojasiewicz inequality}	\label{s:dcyl}

 In this section, we will prove the first Lojasiewicz inequality that bounds the distance squared to the space 
 $\cC_k$  of all rotations of the cylinder $\SS^k_{\sqrt{2k}} \times \RR^{n-k}$ 
  by a power close to one of the gradient of the $F$ functional.   This will follow from the  bounds on the  
 tensor $\tau = \frac{A}{H}$    in the previous section  together with the following proposition:
 
\begin{Pro}	\label{p:cylclose}
Given $n$, $\delta > 0$ and $C_1$, there exist 
  $\epsilon_0 > 0$, $\epsilon_1 > 0$ and $C$ so that
if
$\Sigma \subset \RR^{n+1}$ is a hypersurface (possibly with boundary)
that satisfies:
\begin{enumerate}
 \item $H \geq \delta > 0$ and $|A| + |\nabla A| \leq C_1$ on $B_R \cap \Sigma$.
 \item $B_{5\sqrt{2n}} \cap \Sigma$ is $\epsilon_0$ $C^{2}$-close to a cylinder
 in $\cC_k$ for some $k \geq 1$,
\end{enumerate}
then, for any $r \in (5 \sqrt{2n} , R)$ with
\begin{align}
	r^2 \, \sup_{B_{5\sqrt{2n}}} \left( |\phi| + |\nabla \phi | \right) + r^5 \, \sup_{B_r}  \left( |\nabla \tau | + |\nabla^2 \tau | \right)
	\leq \epsilon_1  \, , 
\end{align}
we have that 
$B_{\sqrt{r^2-3k}} \cap \Sigma$ is the graph over (a subset of) a cylinder in $\cC_k$
of   $u$ with
\begin{equation}
    |u| + |\nabla u| \leq C \, 
    \left\{ r^2 \, \sup_{B_{5\sqrt{2n}}} \left( |\phi| + |\nabla \phi | \right) + r^5 \, \sup_{B_r} \left( |\nabla \tau | + |\nabla^2 \tau | \right) \right\}  \, .
\end{equation}
 
\end{Pro}

\vskip2mm
This proposition shows that $\Sigma$ must be close to a cylinder as long as $H$ is positive, $\phi$ is small,  $\tau$ is almost parallel and $\Sigma$ is close to a cylinder on a fixed small ball.  Together with Tom Ilmanen, we proved a similar result in proposition $2.2$ in \cite{CIM} in the special case where $\Sigma$ is a shrinker (i.e., when $\phi \equiv 0$) and this proposition is inspired by that one.

\vskip2mm
We will prove the proposition over the next two subsections and then turn to the proof of the first Lojasiewicz inequality.

\subsection{Ingredients in the proof of Proposition \ref{p:cylclose}}
This subsection contains the ingredients for the proof of Proposition \ref{p:cylclose}.  The first is
  the following result from
\cite{CIM} (see corollary $4.22$ in \cite{CIM}):

\begin{Cor}[\cite{CIM}]  \label{c:lowev} If $\Sigma \subset \RR^{n+1}$ is
  a hypersurface (possibly with boundary) with  
 \begin{itemize}
  \item $0 <  \delta \leq H$ on $\Sigma$,
  \item the tensor $\tau \equiv A/H$ satisfies $\left| \nabla \tau
    \right| + \left| \nabla^2 \tau \right| \leq \eps\leq 1$,
  \item At the point $p \in \Sigma$, $\tau_p$ has at least two distinct eigenvalues $\kappa_1 \ne \kappa_2$,
 \end{itemize}
then  
\begin{align*}
  \left| \kappa_1 \kappa_2 \right| \leq \frac{2 \, \eps}{\delta^2}
  \, \left(  \frac{1}{\left| \kappa_1 - \kappa_2 \right|}  +
  \frac{1}{\left| \kappa_1 - \kappa_2 \right|^2} \right) \, .
\end{align*}

\end{Cor}

\vskip2mm
We will  use  two additional lemmas in the proof of Proposition \ref{p:cylclose}.  The next lemma 
shows that $\phi$ controls the distance to the shrinking sphere in a neighborhood of the sphere.  This, 
of course,
implies that the shrinking sphere is isolated in the space of shrinkers. The proof uses that the linearized 
operator is invertible.

\begin{Lem}	\label{l:ift}
Given $k$ and $\alpha > 0$, there exist $\eps_0 > 0$  and $C$ so that if $\Sigma_0 \subset \RR^{k+1}$ is
the graph of a $C^{2,\alpha}$ function $u$ over 
 $\SS^k_{\sqrt{2k}}$ with $\| u \|_{C^{2}} \leq \eps_0$, then
 \begin{align}
    \| u \|_{C^{2,\alpha}} &\leq C \, \| \phi \|_{C^{\alpha}} \,  .
 \end{align}
%where $\phi = \frac{1}{2} \, \langle x , \nn \rangle - H$.
\end{Lem}

\begin{proof}
On the sphere, the linearized operator $L$ for $\phi$ is given by $L = \Delta + 1$ since $|A|^2 = 1/2$ and
the drift term vanishes.  The eigenvalues for $\Delta$ on the sphere of radius one occur in clusters
with the $m$-th cluster at
$
  m^2 + (k-1) \, m $.
Scaling this to the sphere of radius $\sqrt{2k}$, the $m$-th cluster is now at
\begin{align}	\label{e:skspectral}
  \frac{m^2 + (k-1) \, m}{2k} \, ,
\end{align}
and, thus, the first three eigenvalues for $L = \Delta + 1$  occur at
$
   -1, \, - \frac{1}{2}$ and $\frac{1}{k}$.
   In particular, $0$ is not an eigenvalue and, thus, $L$ is invertible and, by the Schauder estimates, we have
\begin{align}
    \| u \|_{C^{2,\alpha}} \leq C \, \| L \, u \|_{C^{\alpha}}  \, , 
\end{align}
where $C$ depends only on $k$ and $\alpha$.  The lemma  follows from this and the fact that the linearization of $\phi$ is
$L$ and the error is quadratic (cf. Lemma \ref{l:ffu} below) so we have
\begin{align}
   \left\| \phi - L u \right\|_{C^{\alpha}} \leq C \,  \| u \|_{C^{2}} \,
      \| u \|_{C^{2,\alpha}} \, ,
\end{align}
where $C$ again depends only on $k$ and $\alpha$.  Combining the last two inequalities gives
\begin{align}
    \| u \|_{C^{2,\alpha}} \leq C \,   \| \phi \|_{C^{\alpha}}  +   C\, \| u \|_{C^{2}} \,
      \| u \|_{C^{2,\alpha}}  
    \leq C \,   \| \phi \|_{C^{\alpha}}  +   C\, \epsilon_0 \, \| u \|_{C^{2,\alpha}}  \, , 
\end{align}
which gives the claim after choosing $\epsilon_0 > 0$ so that $C\, \epsilon_0 = \frac{1}{2}$.
\end{proof}

The next lemma shows if $\Sigma$ has an approximate translation and is almost a shrinker, then slicing $\Sigma$ 
orthogonally to the translation gives a submanifold $\Sigma_0$ of one dimension less that is also almost a shrinker.
We will use this to repeatedly slice an almost cylinder to get down to the almost sphere.     We let $\phi_0$   be
the $\phi$ of $\Sigma_0$ (so $\Sigma_0 \subset \RR^k$ is a shrinker when $\phi_0 \equiv  0$).

\begin{Lem}	\label{l:slice}
Let $\Sigma \subset \RR^{k+1}$ be a hypersurface,  $\Sigma_0 = 
\{ x_{k+1} = 0 \} \cap \Sigma$, 
and $x \in \Sigma_0 $ a point where $\Sigma$ intersects the hyperplane $\{ x_{k+1} = 0 \}$ transversely.
If we have
\begin{itemize}
 \item $\left| \nabla^T x_{k+1} \right| \geq 1 - \epsilon > 1/2$,
 \item  $ \left| \nabla^T \,  \nabla^T x_{k+1} \right| \leq \epsilon$, 
 \item $\left| A ( \cdot  , \nabla^T x_{k+1} )\right| + 
 \left| \left( \nabla A \right)\, ( \cdot  , \nabla^T x_{k+1} )\right| \leq \epsilon$.
\end{itemize}
Then at $x$  
\begin{align}
    \left| \phi - \phi_0  \right| + \left| \nabla_{\Sigma_0} (\phi - \phi_0) \right| \leq 24 \, \epsilon
    \left\{ 1+   |\phi| + |\nabla \phi|  \right\}  \, .
\end{align}

\end{Lem}

\begin{proof}
Set $v =  \nabla^T x_{k+1} = \partial_{k+1}^T$.
Let $e_1 , \dots , e_{k-1}$ be an orthonormal frame for $\Sigma_0$, so that
\begin{align}	\label{e:theframe}
   e_1 , \dots , e_{k-1}, \frac{v}{|v|}
\end{align}
gives an orthonormal frame for $\Sigma$.   If  $\nn\in \RR^{k+1}$ and $\nn_0 \in \RR^k$ denote the normals to $\Sigma$ and $\Sigma_0$, respectively,
then
\begin{align}
  \nn = |v| \, \nn_0 + \langle \partial_{k+1} , \nn \rangle \, \partial_{k+1} \, .
\end{align}
(To see this, check that this unit vector is orthogonal to the frame \eqr{e:theframe}.)
Since $\langle \nabla_{e_i} e_j , \partial_{k+1} \rangle = 0$, the expression for $\nn$ gives
$\langle \nabla_{e_i} e_j , \nn \rangle = |v| \, \langle \nabla_{e_i} e_j , \nn_0 \rangle$.  It follows 
that
\begin{align}
   H - H_0 &= - \left\{  A(e_i , e_i)  + A \left(\frac{v}{|v|} 
    ,\frac{v}{|v|} \right) \right\} + \langle \nabla_{e_i} e_i ,  \nn_0 \rangle \notag \\
    &= \frac{1-|v|}{|v|} \,  A(e_i , e_i) - A \left(\frac{v}{|v|} 
    ,\frac{v}{|v|} \right)
    =\frac{|v|-1}{|v|} \,  H- \frac{1}{|v|}   A \left(\frac{v}{|v|} 
    ,\frac{v}{|v|} \right)    \, .   
\end{align}
Similarly, given $x \in \Sigma_0$, we have $x_{k+1} = 0$ and, thus,
\begin{align}
   \langle x , \nn \rangle - \langle x_0 , \nn_0 \rangle = \langle x , \nn \rangle - \langle x , \nn_0 \rangle = \frac{|v| -1}{|v|} \,  \langle x , \nn \rangle \, .
\end{align}
Combining the last two equations gives for $x \in \Sigma_0$ that
\begin{align}
	\phi - \phi_0 &= \frac{1}{2} \, \left( \langle x , \nn \rangle - \langle x_0 , \nn_0 \rangle \right) - 
	\left( H - H_0 \right)   =  \frac{|v| -1}{|v|} \, \left\{ \frac{1}{2} \,  \langle x , \nn \rangle - H \right\} +  \frac{1}{|v|}   A \left(\frac{v}{|v|} 
    ,\frac{v}{|v|} \right)   \notag \\
    &=  \frac{|v| -1}{|v|} \,\phi +  \frac{1}{|v|}   A \left(\frac{v}{|v|} 
    ,\frac{v}{|v|} \right)  
	\, .
\end{align}
Since $|v| \geq 1/2$ and $1 - |v| \leq \epsilon$, it follows that
\begin{align}
	|\phi - \phi_0|  \leq  2 \epsilon \, |\phi |+  8 \, \left| A (v,v)\right| \leq 	
	2 \epsilon \, |\phi |+  8 \, \epsilon  \, .
\end{align}
Similarly, we bound the derivative by
\begin{align}
	\left| \nabla (\phi - \phi_0 ) \right|& \leq   2 (1-|v|) \,| \nabla \phi | +2 \, |\nabla v| \, |\phi|
	+ 4\, (1-|v|) |\nabla v| \, |\phi| \notag \\
	&+ 16 \, |\nabla v| \, |A(v,v)| + 8 \, \left| \nabla A (v,v) \right| 
	+ 16\, \left| A (v , \nabla v) \right| \\
	 &\leq 2 \, \epsilon \, |\nabla \phi |  + 4 \, \epsilon \, |\phi| + 16 \, \epsilon
	 	\, .  \notag
\end{align}

\end{proof}

\subsection{The proof of Proposition \ref{p:cylclose}}

\begin{proof}[Proof of Proposition \ref{p:cylclose}]
Within the proof, it will be convenient to set
\begin{align}
   \epsilon_{\tau} (r) = \sup_{B_r} \, \left( |\nabla \tau| + |\nabla^2 \tau| \right) {\text{ and }}
   \eps_{\phi}  (r) =\sup_{B_r} \, \left(|\phi| +  |\nabla \phi |   \right) \, .
\end{align}

\vskip2mm
{\bf{Step 1: The approximate translations.}}
 Using the $C^2$-closeness in (2), at every $p$ in $\Sigma\cap B_{5\sqrt{2n}}$ there
are $n-k$ orthonormal eigenvectors
\begin{align*}
v_1 (p) , \ldots, v_{n-k} (p),
\end{align*}
of $A$ with eigenvalues $\kappa_1 , \dots \kappa_{n-k}$ with absolute value
less than $1/\sqrt{100n}$, plus  $k \geq 1$
eigenvectors with eigenvalues $\sigma_1 , \dots \sigma_k$ with absolute value
at least $1/\sqrt{4n}$. By (1), we can apply
Corollary \ref{c:lowev} to obtain
\begin{align}	\label{e:e1}
  \left|\kappa_j(p)\right|\leq C \, \eps_{\tau} (5\sqrt{2n}) \, ,\qquad j=1,\ldots,n-k,
\end{align}
where $C$ depends only on $n$ and $\delta$.

Now fix some $p$ in $\Sigma\cap B_{2\sqrt{2n}}$ and define $n-k$ linear functions $f_i$ on $\RR^{n+1}$
and  
tangential vector fields $v_i$ on $\Sigma$ by
\begin{align*}
   f_i (x) = \langle v_i (p) , x \rangle  {\text{ and }}
     v_i = \nabla^T f_i =  v_i (p) - \langle v_i (p) , \nn \rangle \, \nn \, .
\end{align*}

\vskip2mm
{\bf{Step 2: Extending the bounds away from $p$.}}
  For each $r > 5 \sqrt{2n}$, let $\Omega_r$ denote
the set of points in $B_{r} \cap \Sigma$ that can be reached from $p$ by a
path in $B_{r}\cap \Sigma$ of length at most
$3r$. 
The $v_i$'s have the following three properties on $\Omega_r$:
\begin{align}
\left| v_i-v_i(p)\right|&\leq C\,  r^2 \, \epsilon_{\tau}(r)  \,,\label{e:v1}  \\
\left| \tau (v_i )\right|&\leq C\,r^2 \,\eps_{\tau} (r) \,,\label{e:v2} \\
\left|  \nabla_{v_i} A \right|&\leq C\,r^2 \,\eps_{\tau} (r)\,,
\label{e:v3} 
\end{align}
where $C$ depends only on $n$, $\delta$ and $C_1$. 

 To prove \eqr{e:v1} and \eqr{e:v2}, suppose that
  $\gamma:[0,3r] \to \Sigma$ is a curve with $\gamma(0) = p$ and
$|\gamma' | \leq 1$ and that $w$ is a
  parallel unit vector field along $\gamma$ with $w(0) = v_i (p)$.  Therefore,  
  the bound on $\nabla \tau$ gives
  $\left| \nabla_{\gamma'} \, \tau (w) \right| \leq \epsilon_{\tau}(r)$ and, thus, 
  \begin{align}	\label{e:226}
       \left| \tau (w) \right| \leq  3 \, r  \, \epsilon_{\tau} (r) +  \left| \tau_p (v_i (p)) \right| 
       \leq (C+3r ) \, \epsilon_{\tau}(r) \leq C \, r \, \eps_{\tau}(r) \, .
  \end{align}
  In particular, we also have
    \begin{align}
       \left| A (w) \right|= |H| \, \left| \tau (w) \right| 
       \leq C \, r \, \epsilon_{\tau}(r)  \, .
  \end{align}
  Therefore, since $\nabla_{\gamma'}^{\RR^{n+1}} w  = A (\gamma' , w)\, \nn $, the fundamental theorem
  of calculus gives
  \begin{align}	\label{e:keyg}
  	\left| w(t) - v_i (p) \right| = \left| w(t)  - w(0) \right|
  	\leq \int_{0}^{3r} |A(w(s))| \, ds \leq  C\, r^2 \, \epsilon_{\tau}(r) \, .
  \end{align}
Since $w(t)$ is tangential, we see that $\left| \langle v_i(p) , \nn \rangle \right| \leq C\,  r^2  \, \epsilon_{\tau}(r)$,
giving \eqr{e:v1}.   Similarly, \eqr{e:keyg} gives that
\begin{align}
	\left| w(t) - v_i \right| = \left| \left( w(t) - v_i (p) \right)^T \right| \leq \left|  w(t) - v_i (p) \right|
	\leq  C\, r^2 \, \epsilon_{\tau}(r) \, .
\end{align}
If we combine this (and the boundedness of $\tau$) with \eqr{e:226}, the triangle inequality gives 
\begin{align}
	\left| \tau (v_i) \right| \leq \left| \tau (w) \right| + \left| \tau  \left( w - v_i \right) \right| 
	\leq  C\, r^2  \, \epsilon_{\tau} (r) \, , 
\end{align}
where we   used the lower bound on $r$ to bound $r$ by $r^2$.
This gives \eqr{e:v2}.

We will see that  \eqr{e:v2} implies \eqr{e:v3}.  Namely, given unit vector fields $x$ and $y$, 
the  Codazzi equation gives  
\begin{align}
  \left| \left( \nabla_{v_i} A \right) (x,y) \right| &= 
  \left| \left( \nabla_{x} A \right) (v_i,y) \right| = 
  \left| \left( \nabla_{x} (H \, \tau) \right) (v_i ,y) \right| \notag \\
  &= \left|H  \left( \nabla_{x}  \tau \right) (v_i,y) \right|
  + \left| \left( \nabla_{x} H \right)\, \tau  (v_i,y) \right| \leq C\, \epsilon_{\tau} (r) 
  + C \, r^2 \,  \epsilon_{\tau} \, 
  \, ,
\end{align}
where the last inequality used that $|H|$ and $|\nabla H|$ are  bounded   by (1).
This 
gives \eqr{e:v3}.

\vskip2mm
{\bf{Step 3: The sphere.}}
From the $\eps_0$ closeness to $\cC_k$ in  
$B_{5\sqrt{2n}}$ in (2), we know that
\begin{align*}
  \Sigma_0 \equiv B_{5\sqrt{2n}} \cap \Sigma \cap \{ f_1 = \cdots =
  f_{n-k} = 0 \}
\end{align*}
is a compact topological $\SS^k$ of radius fixed close to $\sqrt{2k}$.  
 Using \eqr{e:v1}--\eqr{e:v3}, we can apply Lemma \ref{l:slice} $(n-k)$ times to get that $\Sigma_0$ 
has 
\begin{align}
     \| \phi_0 \|_{C^{1}} \leq C\, \left( \eps_{\tau} + \eps_{\phi} \right) \, ,
\end{align}
where $\eps_{\tau}$ and $\eps_{\phi}$ are evaluated at $r = 5 \sqrt{2n}$.
We can now apply Lemma \ref{l:ift} to get that $\Sigma_0$ is a graph over $\SS^k_{\sqrt{2k}}$ of a function
$u_0$ with 
\begin{align}	\label{e:my0b}
    \| u_0 \|_{C^{2,\alpha}} \leq C\, \left( \eps_{\tau} + \eps_{\phi} \right) \, .
\end{align}

\vskip2mm
{\bf{Step 4: The translations and extending the bound.}}
Let $y_1 , \dots , y_{k+1}$ be an orthonormal basis of linear functions   orthogonal to the $f_i$'s.
Define the function $w$   by
\begin{align}
   w^2 \equiv \sum_{i=1}^{k+1} y_i^2  \, ,
\end{align}
so that $w$ would be identically equal to $\sqrt{2k}$ if $\Sigma$ was in $\cC_k$.
In our case, it follows from \eqr{e:my0b} that  the restriction $w_0$ of $w$ to $\Sigma_0$ satisfies
  \begin{align}
  	\| w_0 - \sqrt{2k} \|_{C^{2,\alpha}(\Sigma_0)} \leq C \, \left( \eps_{\tau} + \eps_{\phi} \right) \, .
\end{align}
We will use the $v_j$'s to extend the bounds away from $\Sigma_0$ inside $\Omega_r$. 
Namely, for each $y_i$ and  $v_j$ and any point in $\Omega_r$, we  have 
\begin{align}	\label{e:onemore}
	\left| \nabla_{v_j} \nabla^T y_i \right| & = \left| \nabla_{v_j} \nabla^{\perp} y_i \right| \leq 
	\left| A (v_j , \cdot ) \right| \leq C \, r^2 \, \epsilon_{\tau} (r) \, , 
\end{align}
where the last inequality used \eqr{e:v2} and the positive lower bound for $H$.

We will extend the bounds by constructing a
  ``radial flow''.  First, define a function $f$ by
\begin{align*}
  f^2 = \sum_{i=1}^{n-k} f_i^2 \, ,
\end{align*}
and then define the vector field $v$ by
\begin{align*}
  v = \frac{ \nabla^T f}{\left| \nabla^T f \right|^2}\, .
\end{align*}
Thus,  the flow by $v$ preserves the level sets of $f$.  Note that
\begin{align}
  \nabla^T f = \frac{ \sum f_i \nabla^T f_i}{f} = \sum \frac{f_i}{f} \, v_i 
   = \sum \frac{f_i}{f} \, v_i(p) +  \sum \frac{f_i}{f} \, (v_i- v_i (p))  \,.
\end{align}
Since the $v_i (p)$'s are orthonormal and $\sum \left( \frac{f_i}{f} \right)^2 = 1$, it follows that
\begin{align}
	\left| \sum \frac{f_i}{f} \, v_i(p)  \right| = 1 \, . \notag 
\end{align}
Combining this with the triangle inequality and 
  \eqr{e:v1}  gives that
\begin{align}
	\sup_{\Omega_r} \,  \left| 1 - \left| \nabla^T f \right| \right| & \leq  \sum   |v_i (p) - v_i| \leq C \, r^2 \, \eps_{\tau} (r) \, , 
\end{align}
where $C$ depends only on $n$, $\delta$ and $C_1$.  We will assume from now on that $r$ satisfies
\begin{align}	\label{e:myr}
	  C \, r^2 \, \eps_{\tau} (r) \leq \frac{1}{2} \, , 
\end{align}
so that  $ \left| 1 - \left| \nabla^T f \right| \right| \leq \frac{1}{2}$ and, thus, that
$
	\sup_{\Omega_r} \,  |v| \leq 2$.
	Since $v$ is in the span of the $v_i$'s and $|v| \leq 2$, it follows from \eqr{e:onemore} that
\begin{align}	\label{e:onemore2}
	\sup_{\Omega_r} \,  \left| \nabla_{v} \nabla^T y_i \right| &   \leq 
	  C \, r^2 \, \epsilon_{\tau} (r) \, .
\end{align}
%%insert on 2/26/14 below
Since $\langle \nabla y_i , v_j \rangle = 0$ at $p$ and $\left| v_i-v_i(p)\right|\leq C\,  r^2 \, \epsilon_{\tau}(r) $ on $\Omega_r$ by \eqr{e:v1}, we know that
$$\left|\nabla^T_{v_j} y_i \right| \leq C\,  r^2 \, \epsilon_{\tau}(r) {\text{ on }} \Omega_r \, . $$
Hence, since $v$ is in the span of the $v_j$'s and $|v| \leq 2$,  
 $\left|\nabla_{v} y_i \right| \leq C\,  r^2 \, \epsilon_{\tau}(r) $ on $\Omega_r$.  Combining this and \eqr{e:onemore2} gives
 \begin{align}	\label{e:onemore2new}
	\sup_{\Omega_r} \,  \left| \nabla_{v} \nabla^T w^2 \right| &  = 2\, \sup_{\Omega_r} \,  \left| \nabla_{v} (y_i \nabla^T y_i) \right|   \leq 
	2(k+1) \, \sup_{\Omega_r} \left\{ |\nabla_v y_i| \, |\nabla^T y_i| + |y_i| \, \left| \nabla_{v}  \nabla^T y_i \right| \right\}  \notag \\
	&\leq 
	  C \, r^3 \, \epsilon_{\tau} (r) \, .
\end{align}

We will now define a subset $\Omega_{r,f}$ of $\Omega_r$ given by flowing $\Sigma_0$ outwards along the 
vector field $v$.  To do this, let $\Phi (q,t)$ to be the flow by $v$ at time $t$ starting from   
$q$ and set
\begin{align}
   \Omega_{r,f}  = \left\{ \Phi (q,t) \, | \, q \in \Sigma_0 , \, 
   t^2 \leq r^2 - 3k {\text{ and }} \Phi (q,s) \in \Omega_r {\text{ for all }} s \leq t \right\} \, .
\end{align}
By integrating \eqr{e:onemore2new} up from $\Sigma_0$, we conclude   that
\begin{align}	\label{e:conq1}
	\sup_{\Omega_{r,f}} \, \left| \nabla^T w^2  \right|  \leq \sup_{\Sigma_0} \, \left| \nabla^T w^2  \right| 
	+ 6 \, r \, \sup_{\Omega_r} \left| \nabla_{v} \nabla^T w^2 \right|  \leq  C \, \eps_{\phi} + 
	C \, r^4 \, \epsilon_{\tau} (r)   \, .
\end{align}
Integrating \eqr{e:conq1} from $\Sigma_0$  gives that
\begin{align}
	\sup_{\Omega_{r,f}} \, \left|  w^2 -  2k \right| & \leq C \, r \,  \eps_{\phi} + C \, r^5 \, \epsilon_{\tau} (r)  \, . \label{e:conq3}
\end{align}
Observe next that as long as 
\begin{align}
	 C \, r \,  \eps_{\phi} + C \, r^5 \, \epsilon_{\tau} (r) \leq 
	 k \, , \label{e:conq4}
\end{align}
then we can conclude that
\begin{align}	\label{e:finalshape}
   \Omega_{r,f} = \{ f^2 \leq r^2 - 3k \} \cap \Sigma \, .
\end{align}
This gives a positive lower bound for $w$ on $ \Omega_{r,f}$ so the bound on $\nabla^T w^2$ then gives  
\begin{align}\sup_{\Omega_{r,f}} \, \left| \nabla^T w  \right|  \leq    C \, \eps_{\phi} + 
	C \, r^4 \, \epsilon_{\tau} (r) \, , 
\end{align}
so  the $C^1$ bound on $w$, and thus also on $u$, hold as claimed.

\end{proof}

\subsection{Proving the first Lojasiewicz inequality}   

In this subsection, we will prove
Theorem \ref{t:ourfirstloja}.  The proof not only gives the $L^2$ closeness to a cylinder, but also gives pointwise 
closeness on a scale that depends on $\phi$ and the initial graphical scale of $\Sigma$.

\begin{proof}[Proof of Theorem \ref{t:ourfirstloja}]
We have that $B_R \cap \Sigma$ is a smooth graph over a cylinder of a function $\bar{u}$ with 
 $\| \bar{u} \|_{C^{2,\alpha}} \leq \epsilon$ and $|\nabla^{\ell} \bar{u} | \leq C_{\ell}$ and that
$\Sigma  $  satisfies:
\begin{enumerate}
 \item $H \geq \delta > 0$ and $|A| + |\nabla A| \leq C_1$ on $B_R \cap \Sigma$.
 \item $B_{5\sqrt{2n}} \cap \Sigma$ is $\epsilon_0$ $C^{2}$-close to a cylinder
 in $\cC_k$ for some $k \geq 1$.
\end{enumerate}

 \vskip2mm
The starting point is Proposition \ref{p:cylclose} which gives,
  for any $r \in (5 \sqrt{2n} , R)$ with
\begin{align}
	r^2 \, \sup_{B_{5\sqrt{2n}}} \left( |\phi| + |\nabla \phi | \right) + r^5 \, \sup_{B_r}  \left( |\nabla \tau | + |\nabla^2 \tau | \right)
	\leq \epsilon_1  \, , 
\end{align}
we have that 
$B_{\sqrt{r^2-3k}} \cap \Sigma$ is the graph over (a subset of) a cylinder in $\cC_k$
of   $u$ with
\begin{equation}
    |u| + |\nabla u| \leq C \, 
    \left\{ r^2 \, \sup_{B_{5\sqrt{2n}}} \left( |\phi| + |\nabla \phi | \right) +
    r^5 \, \sup_{B_r} \left( |\nabla \tau | + |\nabla^2 \tau | \right) \right\}  \, .
\end{equation}
Using the a priori bounds and assuming that $\ell$ is large enough,{\footnote{We will take $\ell$ large later;
we could replace $3/4$ by any constant less than one by taking $\ell$ larger.}}
 we can use the interpolation inequalities of Lemma \ref{l:interp1} to get that
\begin{align}
   \sup_{B_{5\sqrt{2n}}} \left( |\phi| + |\nabla \phi | \right) \leq C_4 \, \| \phi \|_{L^1(B_R)}^{ \frac{3}{4} } 
   \, ,
\end{align}
where $C_4 = C_4 (n)$ and $L^1(B_R)$ denotes the Gaussian $L^1$ norm on $B_R$.

\vskip2mm
To get bounds on $\nabla \tau$ and $\nabla^2 \tau$, we apply
 Corollary \ref{c:ptwise} to get   
$C_3 = C_3 (n, \lambda_0 , \ell , C_{\ell})$ so that for  $r + \frac{1}{1+r} < R-1$,  we have
 \begin{align}
    \sup_{B_r} \left(  |\nabla \tau| + |\nabla^2 \tau | \right) &\leq C_3 \, 
    R^{2n} \, \left\{ \e^{ - d_{\ell , n } \, \frac{(R-1)^2}{8} }+ \| \phi \|_{L^1(B_R)}^{ \frac{d_{\ell , n}}{2} } 
    \right\}  \,    \e^{ \frac{r^2}{8} } \, , 
 \end{align}
where the exponent $d_{\ell,n} \in (0,1)$ has $\lim_{\ell \to \infty} d_{\ell,n} 
 =1$.

 Thus, we see that
  $B_{\sqrt{r^2-3k}} \cap \Sigma$ is the graph over (a subset of) a cylinder $\Sigma_k \in \cC_k$
of   $u$ with
\begin{align}	\label{e:bdU}
    |u| + |\nabla u| &\leq C \, 
    \left\{ r^2 \, \| \phi \|_{L^1}^{ \frac{3}{4} } +
      r^5 \,   R^{2n} \, \left\{ \e^{ - d_{\ell , n } \, \frac{(R-1)^2}{8} }+ \| \phi \|_{L^1(B_R)}^{ \frac{d_{\ell , n}}{2} } 
    \right\} 
      \, \e^{ \frac{r^2}{8} } \right\} 
   \notag  \\   &
      \leq C \,   R^{2n+5 } \,   \left\{ \e^{ - d_{\ell , n } \, \frac{(R-1)^2}{8} }+ \| \phi \|_{L^1(B_R)}^{ \frac{d_{\ell , n}}{2} } 
    \right\} 
      \, \e^{ \frac{r^2}{8} }  \, ,
\end{align}
where $C = C (n, \lambda_0 , \ell, C_{\ell})$ and this holds so long as the right hand side is at most $\epsilon_1 > 0$.   Define the radius $R_1 \leq R-1$ to be the maximal radius where this holds.  

To get the $L^2$ bound, 
 we first use \eqr{e:bdU}  on $B_{R_1}$ to get
\begin{align}
    \int_{B_{R_1}} \left| w_{\Sigma_k} - \sqrt{2k} \right|^2 \, \e^{ - \frac{|x|^2}{4} } & \leq  C \, R^{5n+10} \,  
   \left\{ \e^{ - d_{\ell , n } \, \frac{(R-1)^2}{4} }+ \| \phi \|_{L^1(B_R)}^{ d_{\ell , n} } 
    \right\}  \, ,
\end{align}
and then use that $\left| w_{\Sigma_k} - \sqrt{2k} \right|^2(x) \leq |x|^2$ to get that
 \begin{align}
    \int_{B_R \setminus B_{R_1}} \left| w_{\Sigma_k} - \sqrt{2k} \right|^2 \, \e^{ - \frac{|x|^2}{4} } & \leq  C \, R^{n+2} \,  
  \e^{ - \frac{R_1^2}{4}  } \leq C \, R^{5n+12} \, 
  \left\{ \e^{ - d_{\ell , n } \, \frac{(R-1)^2}{4} }+ \| \phi \|_{L^1(B_R)}^{  d_{\ell , n}  } 
    \right\} 
  \, ,
\end{align}
 where the last inequality is the definition of $R_1$.  Combining these completes the proof.
 
\end{proof}

We will later also need a variation on this, where we assume bounds on $A$ and $H$ on a large scale and conclude that $\Sigma$ is a graph over a cylinder on 
a large set.

\begin{Thm}  \label{t:ourfirstloja3}
There exist $R_0 , \ell_0$ and $\delta > 0$  so that if  $\Sigma \subset \RR^{n+1}$ has $\lambda (\Sigma) \leq \lambda_0 $ and
\begin{enumerate}
\item for some $R > R_0$, we have on $B_R \cap \Sigma$ that $|A| + |\nabla^{\ell_0} A| \leq C_0$ 
and $H \geq \delta_0 > 0$,
\item  $B_{R_0} \cap \Sigma$ is a $C^2$ graph over some cylinder in $\cC_k$ with norm at most $\delta$.
\end{enumerate}
Then  there is  a cylinder $\tilde{\Sigma} \in \cC_k$ so that
  \begin{itemize}
  \item[(3)] $B_{ {R}_1-2} \cap \Sigma$ is the graph of ${u}$ over $\tilde{\Sigma}$ with $\| {u} \|_{C^{2,\alpha}} \leq   \epsilon_0$,
  \end{itemize}
  where  
  ${R}_1$ is given by
\begin{align}	\label{e:defR1}
   R_1 = \max \, \left\{ r \leq R-1 \,  \big|  \,  R^{2n+5 } \,   \left( \e^{ - b_{\ell_0 , n } \, \frac{(R-1)^2}{8} }+ \| \phi \|_{L^1(B_R)}^{ \frac{b_{\ell_0 , n}}{2} } 
    \right)
      \, \e^{ \frac{r^2}{8} }  \leq   \tilde{C} \right\}    \, ,
\end{align}
the exponent $b_{\ell_0 , n} \in (0,1)$ satisfies $\lim_{\ell_0 \to \infty} b_{\ell_0 , n} = 1$ and $\tilde{C} = \tilde{C} (n , \lambda_0 ,  
 \delta_0 ,  C_0)$.
\end{Thm}

\begin{proof}%[Modifications for the proof of Theorem \ref{t:ourfirstloja3}]
We  follow  the proof of Theorem    \ref{t:ourfirstloja} up through \eqr{e:bdU} to get  $\tilde{\Sigma} \in \cC_k$ and a function $u$ so 
that $B_{R_1 - 1} \cap \Sigma$ is the graph of $u$ over $\tilde{\Sigma}$, $R_1$ is defined by \eqr{e:defR1}, and
\begin{align}
	|u| + |\nabla u| \leq 2\, \delta \, .
\end{align}
Finally, we  use interpolation and the $\nabla^{\ell_0} A$ bound to get the desired $C^{2, \alpha}$ bound when $\delta > 0$ is sufficiently small.

\end{proof}

\section{Analysis on the cylinder}	\label{s:analysis}

 In this section, we will prove estimates for the  $\cL$ and $L$ operators on a cylinder $\Sigma \in \cC_k$
 with $k \in \{ 1 , \dots , n-1 \}$.  These estimates will be used in the next section to prove our second Lojasiewicz inequality.
 Note that $
	L = \cL + 1$ on $\Sigma$ since
 $|A|^2\equiv \frac{1}{2}$.

We will use the Gaussian 
$L^2$-norm
$
\|u\|_{L^2}^2=\int u^2\,\e^{-\frac{|x|^2}{4}}$, as well as the associated Gaussian $W^{1,2}$ and $W^{2,2}$ norms  
\begin{align}
\|u\|_{W^{1,2}}^2= \int \left( u^2 + |\nabla u|^2 \right) \,\e^{-\frac{|x|^2}{4}} {\text{ and }}
\|u\|_{W^{2,2}}^2= \int \left( u^2 + |\nabla u|^2 + \left| \Hess_u \right|^2  \right) \,\e^{-\frac{|x|^2}{4}}
\, .
\end{align}

\subsection{Symmetry, the spectrum of $\cL$ and a Poincar\'e inequality}

The starting point is the following elementary lemma that summarizes
the key properties of the $\cL$ operator on $\Sigma \in \cC_k$:

\begin{Lem}	\label{l:discrete}
The operator $\cL$ on $\Sigma$ is symmetric on $W^{2,2}$ with
\begin{align}
   \int_{\Sigma} u \, \cL v \, \e^{ - \frac{|x|^2}{4} } = -
   \int_{\Sigma} \langle \nabla u , \nabla v \rangle \, \e^{ - \frac{|x|^2}{4} } \, .
\end{align}
The space $W^{1,2} $ embeds compactly into $L^2  $ and
 $\cL$ has discrete spectrum with finite multiplicity on 
  $W^{2,2} $ with a complete basis of 
  smooth $L^2 $-orthonormal eigenfunctions.
\end{Lem}

\begin{proof}
The first claim follows   from integration by parts.  The second  follows 
from \cite{BE} since $\Sigma$ has positive 
Bakry-\'Emery Ricci curvature and finite weighted
volume.  Finally, the last claim is a  consequence of the first two (cf. theorem $10.20$ in \cite{Gr}).
\end{proof}

We will also use the following Gaussian Poincar\'e inequality on $\Sigma = \SS^k_{\sqrt{2k}} \times \RR^{n-k}$.  The middle term does not use the full gradient, but   only the gradient in the translation  directions.

\begin{Lem}	\label{l:sob}
There exists $C= C(k,n)$ so that if $\Sigma \in \cC_k$ and $u \in W^{1,2}$, then
\begin{align}
    \| |x| \, u \|^2_{L^2} \leq C \, \left( \| u \|_{L^2}^2 + \| \nabla_{\RR^{n-k}} u \|_{L^2}^2 \right) \leq C \, \| u \|^2_{W^{1,2}} \, .
\end{align}

\end{Lem}

\begin{proof}
Let $y$ be coordinates on the $\RR^{n-k}$ factor, so that 
 \begin{align}
     x^T = y {\text{ and }} |x|^2 = |y|^2 +  2k \, .
 \end{align}
We compute
 \begin{align}
     \e^{  \frac{|x|^2}{4} } \, \dv_{\Sigma} \, \left( u^2 \, y \, \e^{ - \frac{|x|^2}{4} } \right) &=
     2 \, u \langle \nabla u , y \rangle + (n-k) \, u^2 - u^2 \, \frac{|y|^2}{2} \notag \\
     &\leq 4 \, \left| \nabla_{\RR^{n-k}} u \right|^2 + (n-k) \, u^2 - u^2 \, \frac{|y|^2}{4} 
    \, ,
 \end{align}
 where the inequality used the absorbing inequality $2ab \leq \frac{a^2}{4} + 4 \, b^2$.
 
 By approximation, we can assume that $u$ has 
 compact support on $\Sigma$ and, thus, Stokes' theorem gives
 \begin{align}
     \frac{1}{4} \, \int_{\Sigma} u^2 \, |y|^2 \, \e^{ - \frac{|x|^2}{4} } 
     \leq \int_{\Sigma} \left\{ (n-k)\, u^2 + 4 \, \left|\nabla_{\RR^{n-k}}  u \right|^2 \right\} \, 
     \e^{ - \frac{|x|^2}{4} } \, .
 \end{align}
The lemma follows since
$u^2 \, |x|^2 = u^2 \, \left( |y|^2 + 2k \right)$.

\end{proof}

 \subsection{Estimates for the projection onto the kernel of $L$}

Let $\cK$ be the kernel of $L$   
\begin{align}
	\cK = \{ v \in W^{2,2} \, | \, L v = 0 \} \, .
\end{align}
 Given any $u \in W^{2,2}$, we let $u_{\cK}$ denote the $L^2$-orthogonal projection of $u$ onto $\cK$
  and 
  \begin{align}
      u^{\perp} = u - u_{\cK}
  \end{align}
  the projection onto the $L^2$-orthogonal complement of $\cK$.
 
 \vskip2mm
 The next lemma shows that $L$ is bounded from $W^{2,2}$ to $L^2$,  $L$ is uniformly invertible on $\cK^{\perp}$
 and
 the projection onto $\cK$ is bounded from $L^2$ to $W^{2,2}$.

 \begin{Lem}	\label{l:propL}
 Given $n$, there exist $C$ and $\mu > 0$ so that on $\cC_k$
 \begin{align}	
        \| L u \|_{L^2} &\leq C \, \| u \|_{W^{2,2}} \, ,  \label{e:LbdW2} \\
 	\mu \, \| u^{\perp} \|_{W^{2,2}} &\leq   \|  L u \|_{L^2} \, , \label{e:invertL} \\
  	\label{e:alsoprove}
 	\| u_{\cK} \|_{W^{2,2}} &\leq C \, \| u \|_{L^2} \,  .
\end{align}

 \end{Lem}
 
 \begin{proof}
  Since $L = \Delta + \frac{1}{2} \, \nabla_{x^T} + 1$ on the cylinder, we have 
  \begin{align}
  	\| L u \|_{L^2} \leq \| \Delta u \|_{L^2} + \| u \|_{L^2} +\frac{1}{2} \, \| |x| \, |\nabla u | \|_{L^2}  \, .
  \end{align}
  The first claim follows from this and using Lemma \ref{l:sob} to get the bound
 \begin{align}
     \| |x| \, |\nabla u | \|_{L^2} \leq C \, \| |\nabla u| \|_{W^{1,2}} \leq C \,\left\{ \| |\nabla u| \|_{L^2} +  \|  \Hess_u \|_{L^2}  \right\} \, .
 \end{align}

 To get the second claim, we will 
need the ``Gaussian elliptic estimate'' 
 \begin{align}	\label{e:w22est}
 	\| v \|_{W^{2,2}} \leq C \, \left( \| v \|_{L^2} + \| \cL v \|_{L^2} \right)  \, , 
 \end{align}
 where $C$ depends on $n$ and the estimate holds for any $v \in W^{2,2}$. To prove \eqr{e:w22est}, we first integrate by parts to get 
\begin{align}
 	\| \nabla v \|^2_{L^2} = \left| \langle v , \cL v \rangle_{L^2} \right| \leq \| v \|_{L^2} \, \| \cL v \|_{L^2} 
	\leq \frac{1}{2} \| v \|^2_{L^2} + \frac{1}{2} \| \cL v \|^2_{L^2} \, .
 \end{align}
 Thus, we see that $\| v \|_{W^{1,2}}$ is bounded by the right hand side of \eqr{e:w22est}.  It remains to bound the $L^2$ norm of the Hessian of $v$.  This will follow from what we've done and the divergence theorem since 
 \begin{align}	\label{e:bochboo}
 	\e^{  \frac{|x|^2}{4} } \, \dv_{\Sigma} \, \left( \left\{ v_{ij} v_i - (\cL v) \, v_j \right\}
	\e^{ - \frac{|x|^2}{4} } \right)  &=  \frac{1}{2} \cL\,  |\nabla v|^2-\left(\cL\, v\right)^2-\langle \nabla \cL\, v, \nabla v\rangle \notag \\
	&\geq  \left| \Hess_v \right|^2  
	- \left( \cL v \right)^2 \, ,
 \end{align}
 where the last inequality used the Bochner formula for the 
 drift Laplacian on the cylinder.{\footnote{See, for instance, \eqr{e:bochnerf} below.}}   
 
 The second claim now follows by first
  applying
Lemma \ref{l:discrete} to get $\mu_0 > 0$ so that
\begin{align}
	\mu_0 \, \| u^{\perp} \|_{L^{2}} &\leq   \| L u^{\perp} \|_{L^2} = \| L u \|_{L^2} 
\end{align}
 and then using \eqr{e:w22est} to bound the $W^{2,2}$ norm.

 The final claim follows from the trivial projection bound $\| u_{\cK} \|_{L^2} \leq \| u \|_{L^2}$ and the bound
 \begin{align}	\label{e:elliptic}
 	\| u_{\cK} \|_{W^{2,2}} \leq C \, \| u_{\cK} \|_{L^2} \, .
 \end{align}
 To see \eqr{e:elliptic}, first use the equation $\cL u_{\cK} = - u_{\cK}$ to get
 $
 	\| \nabla u_{\cK}  \|_{L^2} = \| u_{\cK}  \|_{L^2} $,
 and then use the Bochner formula as in \eqr{e:bochboo} to bound the Hessian of $u_{\cK} $ in terms of $\| u_{\cK}  \|_{W^{1,2}}$.
 \end{proof}
 
  \vskip2mm
We will also need the next  lemma that bounds  the Gaussian $L^2$ norm of a quadratic expression in $u, \nabla u , \Hess_u$
that   bounds  the   error term in the linear approximation of the gradient of the $F$ functional.    When $u \in \cK$, the bound is the square of the Gaussian $L^2$ norm{\footnote{This would be  obvious if
the $C^2$ norm of $v$ were bounded by the $L^2$ norm, but this is not the case.}} while we obtain a weaker bound when $u$ is orthogonal to $\cK$.
  
%  \begin{Lem}	\label{l:CK}
%There exists $C_{K}= C_K (n)$  so that  if  $v \in \cK$, then
% \begin{align}	\label{e:CK}
 %	 \left\| v^2 + |\nabla v|^2 +\left|  \Hess_{v} (\cdot , \RR^{n-k}) \right|^2
 % +   (1 + |x|)^{-1} \left| \Hess_v \right|^2  \right\|_{L^2}  \leq C_K \, \| v \|_{L^2}^{2} \, .
% \end{align}
% \end{Lem}
 
% \vskip3mm
 %The next lemma bounds the same quadratic expression, but this time for  the part   $u^{\perp}$ that is orthogonal to $\cK$.   In this case, 
 %the bound is not as good.  % and the right hand side is the Gaussian $W^{2,2}$ norm times the $C^2$ norm of $u$.
 
 % \begin{Lem}	%\label{l:2v22}
%There exists $C_0 = C_0 (n)$ so that  if $u \in W^{2,2}$, then
 % \begin{align}  \label{e:2v22}
% 	 \left\| (u^{\perp})^2 + |\nabla u^{\perp}|^2 +\left|  \Hess_{u^{\perp}} (\cdot , \RR^{n-k}) \right|^2
%  +   (1 + |x|)^{-1} \left| \Hess_{u^{\perp}} \right|^2  \right\|_{L^2} 
%	\leq C_0 \, \| u \|_{C^2}  \, \| u^{\perp} \|_{W^{2,2}} \, .
% \end{align}
% \end{Lem}

 \begin{Lem}	\label{l:CK}
There exist $C_{K}= C_K (n)$  and $C_0 = C_0 (n)$ so that  if $u \in W^{2,2}$, then
 \begin{align}	\label{e:CK}
 	 \left\| u_{\cK}^2 + |\nabla u_{\cK}|^2 +\left|  \Hess_{u_{\cK}} (\cdot , \RR^{n-k}) \right|^2
  +   (1 + |x|)^{-1} \left| \Hess_{ u_{\cK}}\right|^2  \right\|_{L^2}  &\leq C_K \, \| u_{\cK} \|_{L^2}^{2} \, , \\
   \left\| (u^{\perp})^2 + |\nabla u^{\perp}|^2 +\left|  \Hess_{u^{\perp}} (\cdot , \RR^{n-k}) \right|^2
  +   (1 + |x|)^{-1} \left| \Hess_{u^{\perp}} \right|^2  \right\|_{L^2} 
	&\leq C_0 \, \| u \|_{C^2}  \, \| u^{\perp} \|_{W^{2,2}} \, .     \label{e:2v22q}
 \end{align}
 \end{Lem}

 \vskip3mm
 The key for proving both claims is an explicit description of $\cK$.  Namely, $\cK$
 is generated by 
  multiplying a polynomial eigenfunction of $\cL_{\RR^{n-k}}$ times
 a spherical eigenfunction of $\Delta_{\SS^k_{\sqrt{2k}}}$.
 To state this, 
  let $y_i$ be coordinates on the $\RR^{n-k}$ factor and let $\theta$ be in the $\SS^k$ factor. 
 
 \begin{Lem}	\label{l:kernel}
 Each  $v \in \cK$ can be written as
 \begin{align}
    v (y, \theta) = q (y) + \sum_i y_i f_i (\theta)  + c \, , 
 \end{align}
where $q$ is a homogeneous quadratic polynomial on $\RR^{n-k}$, each $f_i$ is an eigenfunction on $\SS^k_{\sqrt{2k}}$ with eigenvalue 
$ \frac{1}{2}$,   and $c$ is a constant.
 \end{Lem}

 \begin{proof}
  
  The operator $L$ splits as
  \begin{align}	\label{e:Lsplit}
      L = \cL + 1 = \Delta_{\theta} + \cL_y + 1 \, ,
  \end{align}
  where $\Delta_{\theta}$ is the Laplacian on $\SS^k_{\sqrt{2k}}$ and
  $\cL_y$ is the drift operator on $\RR^{n-k}$.
  
  The first obervation is that differentiating with respect to $y_i$  lowers the eigenvalue by $\frac{1}{2}$.  Thus, 
  if we set $v_i = \frac{\partial v}{\partial y_i}$ and
  $v_{ij} = \frac{\partial^2 v}{\partial y_i \partial y_j}$, then 
%   $v_i$ and $v_{ij}$ are in $W^{2,2}$ with
  \begin{align}
     \cL  v_i &= - \frac{1}{2} v_i \, , \\
     \cL v_{ij} & = 0 \, .
  \end{align}
 Since every $L^2$ $\cL$-harmonic function must be constant, we conclude that  $v_{ij}$ is constant.  
 As a consequence, the function $v$ can be written as
 \begin{align}	\label{e:decomp1}
    v = \sum_{i,j} a_{ij} \, y_i \, y_j + \sum_i f_i (\theta) \, y_i + g(\theta) \, , 
 \end{align}
where $a_{ij} \in \RR$, each $f_i$ is a function on $\SS^k_{\sqrt{2k}}$ and $g$ is a function on $\SS^k_{\sqrt{2k}}$.

Note that
\begin{align}
    \cL_{y}  y_i &= - \frac{1}{2} \, y_i \, , \\
    \cL_y \, (y_i y_j) &= 2 \delta_{ij} - y_i y_j \, .
\end{align}
Using this and the decomposition of $L$ from \eqr{e:Lsplit}, we get that
 \begin{align}	\label{e:decomp2}
    0 = L v = \sum_{i,j} a_{ij} \, \left(2 \delta_{ij} \right)
    + \sum_i \left[y_i \, \Delta_{\theta} f_i (\theta)  
    + \frac{1}{2} f_i (\theta) \, y_i \right]  +  \left( \Delta_{\theta}  + 1 \right)
    g(\theta) \, .
 \end{align}
 Observe first that only the middle terms depend on $y$.  Setting these equal to zero, we conclude that each 
 $f_i$ satisfies
 \begin{align}
    \Delta_{\theta} f_i = -  \frac{1}{2} f_i \, .
 \end{align}
 It follows that $g + 2 \, \sum a_{ii}$ is a  $\SS^k_{\sqrt{2k}}$ eigenfunction with eigenvalue one, i.e., 
 \begin{align}
     \Delta_{\theta} \, \left(g + 2 \, \sum a_{ii} \right) = - \left(g + 2 \, \sum a_{ii} \right) \, .
 \end{align}
However, one is not an eigenvalue of $\Delta_{\theta}$  (the eigenvalues jump from $1/2$ to $(k+1)/k$; see \eqr{e:skspectral}), so we have 
$g \equiv - 2 \, \sum a_{ii}$.

 \end{proof}
 
 It is interesting to note that the $y_i \, f_i$ part of $\cK$ corresponds  to rotations.  However, by \cite{CIM}, 
 the quadratic  polynomials in
 the kernel 
are not generated by one-parameter families of shrinkers.  In particular, the kernel $\cK$ contains non-integrable functions.

 \vskip2mm
 As a corollary of Lemma \ref{l:kernel}, we get   $C^2$ pointwise estimates for functions in the kernel of $L$ that grow at most quadratically in $|y|$:
 
  \begin{Cor}	\label{c:kernel}
There exists $C$ depending on $n$ so that  if $v \in \cK$, then
 \begin{align}
     \sup \,   |v| &\leq C \, (1 + |y|^2) \, \| v \|_{L^2}   \, ,  \label{e:ke1} \\
      \sup \,   |\nabla v| &\leq C \, (1 + |y|) \, \| v \|_{L^2}   \, ,   \label{e:ke2}  \\
      \sup \,    \left| \Hess_{v} (\cdot , \RR^{n-k})  \right| &\leq C   \, \| v \|_{L^2}   \, ,   \label{e:ke3} \\
       \sup \,   \left| \Hess_v \right| &\leq C \, (1 + |y|) \, \| v \|_{L^2}   \, .   \label{e:ke4} 
 \end{align}
 \end{Cor}

 \begin{Rem}
The point of \eqr{e:ke3} is that, as opposed to \eqr{e:ke4}, we get a better bound, that does not grow in $y$, if we restrict   to the Hessian in the 
Euclidean factor.  This is useful later.
 \end{Rem}

\begin{proof}[Proof of Corollary \ref{c:kernel}]
Since $\cK$ is finite dimensional, the estimates \eqr{e:ke1}--\eqr{e:ke4} will follow for all of $\cK$ from the squared 
triangle inequality once we show that there is an orthogonal basis
for $\cK$ where each element in the basis satisfies \eqr{e:ke1}--\eqr{e:ke4}.

The key for this is 
 Lemma \ref{l:kernel} which shows that
   $\cK$ can be written as
 \begin{align}	\label{e:dec123a}
    \cK  =  \cK_1 \oplus \cK_2   \, , {\text{ where}}
 \end{align}
\begin{itemize}
\item Each $v_1 \in \cK_1$ is  given by
$ \sum_i y_i \, f_i$ where
$f_i$ is a $\SS^{k}_{\sqrt{2k}}$ eigenfunction with eigenvalue 
$ \frac{1}{2}$.
\item Each  $v_2 \in \cK_2$ is a constant plus a homogeneous quadratic polynomial in $y$.  
\end{itemize}
In particular, \eqr{e:dec123a} is
a ${L^2}$-orthogonal decomposition.

\vskip2mm
{\bf{Case 1}}.   If $f_i , f_j$ are $\SS^{k}_{\sqrt{2k}}$ eigenfunctions with eigenvalue 
$ \frac{1}{2}$, then 
\begin{align}
	\langle y_i \, f_i , y_j \, f_j \rangle_{L^2} = 0 {\text{ if }} i \ne j \, , 
\end{align}
so we get an orthogonal basis for   $\cK_1$ consisting of a single $y_i$ times an $f$.   Suppose that 
\begin{align}
   v_1 = y_i \, f \, , 
\end{align}
where  $f$ is a $\SS^{k}_{\sqrt{2k}}$ eigenfunction with eigenvalue 
$ \frac{1}{2}$.  Note that
\begin{align}
   \| v_1 \|_{L^2}^2 =   \e^{- \frac{k}{2}} \, \| f \|_{L^2_{\theta}}^2 \, 
      \int_{\RR^{n-k}} y_i^2 \, \e^{ - \frac{|y|^2}{4} } \, dy 
      \equiv C_k \, \| f \|_{L^2_{\theta}}^2 \, ,
\end{align}
where the constant $C_k > 0$ depends only on $k$ and
the sub $\theta$ denotes the norms on $\SS^{k}_{\sqrt{2k}}$.

Using elliptic estimates for the compact manifold $\SS^{k}_{\sqrt{2k}}$,  we have $c_0 = c_0 (k)$ so that
\begin{align}
	\| f \|_{C^2_{\theta}} \leq c_0 \, \| f \|_{L^2_{\theta}} \, .
\end{align}
Therefore, at each point, we have that
\begin{align}
   |v_1|^2 &= y_i^2 \, f^2 \leq c_0^2 \, y_i^2 \, \| f \|_{L^2_{\theta}}^2 \, , \\
   |\nabla v_1|^2&=   y_i^2 \, |\nabla_{\theta} f|^2
   + f^2  \leq   c_0^2 \, \left(1 + y_i^2 \right) \, \| f \|_{L^2_{\theta}}^2 \, , \\
   |\Hess_{v_1}|^2 &= y_i^2 \, \left| \Hess_{f} \right|^2
   \leq c_0^2 \, y_i^2 \, \| f \|_{L^2_{\theta}}^2 \, , \\
   \left|   \Hess_{v_1} (\cdot , \RR^{n-k})  \right|^2 &\leq \left| \nabla_{\theta} f \right|^2 \leq c_0^2 \, , 
\end{align}
giving the desired bounds in this case (the first bound is even better than needed).

\vskip2mm
{\bf{Case 2}}.  
It is easy to see that an orthogonal basis for $\cK_2$  is given by
\begin{align}
	\{ y_i y_j - 2 \delta_{ij} \, | \, i \leq j \} \, .
\end{align}
Therefore, 
  it suffices to show \eqr{e:ke1}--\eqr{e:ke4} when
\begin{align}
 	v_2 = y_i y_j - 2 \delta_{ij} \, .
\end{align}
However, this follows immediately since the $L^2$ norms are nonzero and $v_2$ is a quadratic polynomial in $y$ (in this case, the Hessian bound
is even better than needed).
\end{proof}
 
 \vskip2mm
We will now use the estimates from the corollary to prove Lemma \ref{l:CK}.

 \begin{proof}[Proof of Lemma \ref{l:CK}]
 To simplify notation,    set 
 \begin{align}
     \| v \|_2 \equiv  \left\| v^2 + |\nabla v|^2 + \left| \Hess_{v} (\cdot , \RR^{n-k})  \right|^2
  +   (1 + |x|)^{-1} \left| \Hess_v \right|^2  \right\|_{L^2}  \, .   
 \end{align}
  Given $a \in \RR$, note that $\| a\, v \|_2 = a^2 \, \| v \|_2$.
 
We will show that there is a constant $C_K$ so that 
 \begin{align}	\label{e:wife}
 	C_K \equiv \sup \left\{ \| w \|_{2} \, \big| \,  w \in \cK
	{\text{ and }}   \| w \|_{L^2} = 1  \right\}  < \infty \, .
 \end{align}
 Once we have this, then 
 for a general $v \in \cK$, we set $w =  \frac{v}{\| v\|_{L^2}}  $ so that 
  \begin{align}	\label{e:CK2}
 	\| v  \|_{2} = \| \, \| v \|_{L^2} \, w \|_2 = \| v \|^2_{L^2}
	\,  \| w \|_{2} 
	\leq C_K \, \| v \|_{L^2}^{2} \, ,
 \end{align}
 giving the first claim \eqr{e:CK}.
  
To establish  \eqr{e:wife},   apply
  Corollary \ref{c:kernel}  to get $C = C(n)$ so that   
 \begin{align}
     |w|^4 +  |\nabla w|^4 +  \left| \Hess_w \right| ^4  &\leq C \, (1 + |y|^2)^4
      \, .    
 \end{align} 
Integrating
  this polynomially growing bound against the exponential decaying Gaussian weight gives the desired uniform
 bound on $\|w \|_2^2$.
 
 To prove \eqr{e:2v22q}, 
 we will show that
     \begin{align}	\label{e:h2vv}
 	 \left\| (u^{\perp})^2 \right\|_{L^2}  , \, \left\|  |\nabla u^{\perp}|^2 \right\|_{L^2}  , \, \left\|  \left| \Hess_{u^{\perp}} (\cdot , \RR^{n-k})  \right|^2
	 \right\|_{L^2}  {\text{ and }} \left\| 
      (1 + |x|)^{-1} \left| \Hess_{u^{\perp}} \right|^2  \right\|_{L^2} 
 \end{align}
 are each bounded by $C_0 \, \| u \|_{C^2}  \, \| u^{\perp} \|_{W^{2,2}}$ for a constant $C_0$ depending only on the dimension $n$.
 %  \begin{align}	\label{e:h2vv}
 %	 \| (u^{\perp})^2 + |\nabla u^{\perp}|^2 + \left| \Hess_{u^{\perp}} (\cdot , \RR^{n-k})  \right|^2
  %+   (1 + |x|)^{-1} \left| \Hess_{u^{\perp}} \right|^2  \|_{L^2} 
 %\end{align}
% is bounded by $C_0 \, \| u \|_{C^2}  \, \| u^{\perp} \|_{W^{2,2}}$
   The key point will be the bounds \eqr{e:ke1}--\eqr{e:ke4} on $u_{\cK}$ from
 Corollary \ref{c:kernel}.
 %We will bound each of the four terms in \eqr{e:h2vv} by $C_0 \, \| u \|_{C^2}  \, \| u^{\perp} \|_{W^{2,2}}$.

 For the first term, we use \eqr{e:ke1} to get
 \begin{align}
 	( u^{\perp} )^2 &= ( u - u_{\cK}) \, u^{\perp}  
	\leq  \left(  \| u \|_{C^0} +  C \, (1+ |x|^2) \| u_{\cK} \|_{L^2} \right) \, |u^{\perp}| \notag \\
	&\leq  C\,  \| u \|_{C^0}\,     (1+ |x|^2)   \, |u^{\perp}| \, , 
 \end{align}
 where the last inequality used the projection inequality $ \| u_{\cK} \|_{L^2} \leq \| u \|_{L^2}$ and the trivial inequality
 $\| u \|_{L^2} \leq C \, \| u \|_{C^0}$ that follows since $\Sigma$ has finite Gaussian area.  Integrating and applying
 Lemma \ref{l:sob} twice gives
  \begin{align}
 	\| ( u^{\perp} )^2 \|_{L^2} \leq  C\,  \| u \|_{C^0}\,   \|   (1+ |x|^2)   \, u^{\perp}\|_{L^2} 
	\leq C\,  \| u \|_{C^0}\,   \|    u^{\perp}\|_{W^{2,2}}  \, .
 \end{align}
 
  For the second term, we use the triangle inequality and  \eqr{e:ke2} to get
 \begin{align}
 	|\nabla u^{\perp} |^2 &\leq \left( |\nabla u| + |\nabla u_{\cK}| \right) \, |\nabla u^{\perp} | \leq 
	\left(  \| u \|_{C^1} + C (1+|x|) \,  \| u_{\cK} \|_{L^2} \right) \, |\nabla u^{\perp} |  \notag \\
	&\leq 
	 C \,   \| u \|_{C^1}   (1+|x|)   \, |\nabla u^{\perp} |  \, , 
 \end{align}
 where the last inequality follows as above.
 Integrating and applying
 Lemma \ref{l:sob} gives
  \begin{align}
 	\|  |\nabla u^{\perp} |^2 \|_{L^2} \leq  C\,  \| u \|_{C^1}\,   \|   (1+ |x|)   \, |\nabla u^{\perp}|\|_{L^2} 
	\leq C\,  \| u \|_{C^1}\,   \|    u^{\perp}\|_{W^{2,2}}  \, .
 \end{align}
 
   For the third term, we use the triangle inequality and \eqr{e:ke3} to get 
     \begin{align}
 	\left| \Hess_{u^{\perp}} (\cdot , \RR^{n-k})  \right|^2  &\leq  \left\{ \left| \Hess_{u} (\cdot , \RR^{n-k})  \right|
	+ \left| \Hess_{u_{\cK}} (\cdot , \RR^{n-k})  \right| \right\} \, 
	\left| \Hess_{u^{\perp}} (\cdot , \RR^{n-k})  \right| \notag \\
	&\leq \left\{ \| u \|_{C^2} + C \, \| u_{\cK} \|_{L^2} \right\} \, 
	\left| \Hess_{u^{\perp}} (\cdot , \RR^{n-k})  \right| \\
	&\leq C \, \| u \|_{C^2} \, \left| \Hess_{u^{\perp}} (\cdot , \RR^{n-k})  \right|  \, . \notag
 \end{align}
 Integrating this gives 
 \begin{align}
 	\left\| \left| \Hess_{u^{\perp}} (\cdot , \RR^{n-k})  \right|^2 \right\|_{L^2} \leq 
	C \, \| u \|_{C^2} \,  \| \Hess_{u^{\perp}} \|_{L^2} \, .
 \end{align}

 Finally, for the fourth (last) term, we use the triangle inequality and \eqr{e:ke4} to get
 \begin{align}
 	 (1 + |x|)^{-1} \left| \Hess_{u^{\perp}} \right|^2  &\leq  (1 + |x|)^{-1} \left( \left| \Hess_{u} \right|
	 + \left| \Hess_{u_{\cK}} \right| \right) 
	 \left| \Hess_{u^{\perp}} \right| \notag \\
	 &\leq (1 + |x|)^{-1} \left\{ \| u \|_{C^2}
	 + C \, (1+|x|) \, \| u_{\cK} \|_{L^2}  \right\} 
	 \left| \Hess_{u^{\perp}} \right| \\
	 &\leq  C \, \| u \|_{C^2} \,  \left| \Hess_{u^{\perp}} \right|
	 \notag \, .
 \end{align}
 To bound the last term and complete the proof of \eqr{e:2v22q}, we integrate this to get
 \begin{align}
 	\left\|  (1 + |x|)^{-1} \left| \Hess_{u^{\perp}} \right|^2 \right\|_{L^2}  & 
	\leq  C \, \| u \|_{C^2} \, \|   \Hess_{u^{\perp}} \|_{L^2} 
	 \notag \, .
 \end{align}

 \end{proof}

\section{The gradient Lojasiewicz inequality for $F$}	\label{s:Loja2}

In this section, we will prove a gradient Lojasiewicz inequality for $F$ in a neighborhood of a cylinder 
$\Sigma \in \cC_k$.  The inequality will hold for graphs over part of $\Sigma$ with small $C^2$ norm.   
The key technical ingredient is the next proposition which shows that our first Lojasiewicz implies our gradient Lojasiewicz inequality.    
 
 \begin{Pro}		\label{p:firstsecond}
There exist $C=C(n,\lambda_0)$ and $\epsb = \epsb (n) > 0$ so that if  $\lambda (\Sigma) \leq \lambda_0$ and
$B_{\tilde{R}} \cap \Sigma$ is the graph of $\tilde{u}$  over a cylinder in $\cC_k$ with $\| \tilde{u}\|_{C^{2}} \leq \epsb$, then
for any $\beta \in [0,1)$
 \begin{align}	 	\label{e:mydarnbeta}
	 \left| F (\Sigma) - F (\cC_k) \right| 
      &\leq C 
   \| \phi \|_{L^2 (B_{\tilde{R}})}^{\frac{3+\beta}{2}}
	  + C  \, (1 + {\tilde{R}}^{n-1}) \, \e^{ - \frac{(3+\beta)(\tilde{R} -1)^2}{16}} +
      C \, \| \tilde{u} \|_{L^2(B_{\tilde{R}})}^{ \frac{3+\beta}{1+\beta}}  \, .
\end{align}
 
 \end{Pro}

\vskip1mm
The proof of Proposition \ref{p:firstsecond}
  is an infinite dimensional version
of the model argument using Taylor expansion   given in Subsection \ref{ss:1to2}.  The simple  model  was done with $\beta = 0$, but would have
worked with any $\beta \in [0,1)$.  However, the simple model did not include a cutoff   and, to bound the exponential term in \eqr{e:mydarnbeta}, we will need to choose $\beta$ close to one.{\footnote{The 
$ \| \phi \|_{L^2 (B_{\tilde{R}})}^{\frac{3+\beta}{2}}$  term is fine for any $\beta > 0$ and the  $\| \tilde{u} \|_{L^2(B_{\tilde{R}})}^{ \frac{3+\beta}{1+\beta}}$
term is fine if $\beta < 1$.}}

\subsection{The linearization of the gradient of the $F$ functional}

Given a graph $\Sigma_u$ of a function $u$ over a cylinder $\Sigma \in \cC_k$, we let $F(u) \equiv F (\Sigma_u)$
and then let $\cM (u)$ be the gradient of $F$.  
The next lemma gives  linear and quadratic  approximations for $\cM$ and  $F$, respectively.
 
 \begin{Lem}	\label{l:frechet2}
  There exists $\cwun$ so that if the $C^2$ norm of $u$ is sufficiently small and $u$ is defined on the entire cylinder, then
  \begin{align}
       \left\| \cM (u) - L u \right\|_{L^2} &\leq \cwun \,   \left\| u^2 + |\nabla u|^2 +\left| \nabla_{\RR^{n-k}} |\nabla u| \right|^2
  +   (1 + |x|)^{-1} \left| \Hess_u \right|^2  \right\|_{L^2} \,  ,   \\
      \left| F(u) - F(\cC_k) - \frac{1}{2} \, \langle u , L u \rangle_{L^2} \right| 
      &\leq \cwun \, \| u \|_{L^2} \,  \left\| u^2 + |\nabla u|^2 +\left| \nabla_{\RR^{n-k}} |\nabla u| \right|^2
  +   (1 + |x|)^{-1} \left| \Hess_u \right|^2  \right\|_{L^2}  \, .
  \end{align}

 \end{Lem}

\vskip2mm
The  bound in the first inequality in Lemma \ref{l:frechet2} is essentially  quadratic in $u$.  For example,
it is bounded by $C \, \| u \|_{C^2} \, \| u \|_{W^{2,2}}$.   
Ideally, we would have liked the bound to be quadratic in $\| u \|_{W^{2,2}}$, but the exponential decay in the Gaussian norm
makes this impossible and, thus, leads to technical complications.

\vskip2mm
We will prove Lemma \ref{l:frechet2} in this subsection.  The starting point is the next lemma   computing  
$\cM (u)$   in terms of $u, \nabla u$ and $\Hess_u$.

 \begin{Lem}	\label{l:gradcm2}
  If $\Sigma$ is a cylinder in $\cC_k$ and $p \in \Sigma$,
 then   
 \begin{align}	\label{e:gradcm2}
 	\cM (u) (p) =  f(u(p), \nabla u(p)) + \langle p , V(u(p) , \nabla u(p)) \rangle 
	+ \Phi^{\alpha \beta} (u(p) , \nabla u(p)) \, u_{\alpha \beta} (p) 
	     \, ,
 \end{align}
 where $f$, $V$ and $\Phi^{\alpha \beta}$ depend smoothly on $(s,y)$ for $|s|$ small.
 \end{Lem}
 
 Lemma \ref{l:gradcm2} is proven in Appendix \ref{s:append}.

\vskip2mm
The next lemma shows, for a general map $u \to \cM (u)$ of the form \eqr{e:gradcm2},  that the linearization 
gives a good approximation up to quadratic error.
To state this precisely, 
 consider  
a general map $\cN(u)$  of the form
\begin{align}	\label{e:geng}
   \cN (u) (p)  = f (p, u(p) , \nabla u (p)) +
      \Phi^{\alpha \beta} (p, u(p) , \nabla u (p)) \, u_{\alpha \beta} (p) \, ,
\end{align}
where $f$ and $\Phi^{\alpha \beta}$ are smooth functions of $(p,s,y)$ where $p$ is the point, $s \in \RR$, and
$y$ is a tangent vector at $p$.  The linearization of $\cN$ at $u$ is defined to be
\begin{align}	\label{e:Luv}
   L_u \, v &= \frac{d}{dt} \big|_{t=0} \, \cN (u+tv)  
   %\notag \\   &
    = f_s \, v + f_{y_{\alpha}} v_{\alpha} + \Phi^{\alpha \beta} \, v_{\alpha \beta}
    +  u_{\alpha \beta} \, \left( \Phi^{\alpha \beta}_s \, v +  \Phi^{\alpha \beta}_{y_{\gamma}} \, v_{\gamma}
    \right) \, ,
\end{align}
where all functions are evaluated at the same point $p$ and we have left out the obvious dependence 
of $f$ and $\Phi$ on $(p, u(p) , \nabla u(p))$.

\begin{Lem}	\label{l:ffu}
If $\cN (u)$ is given by \eqr{e:geng}, then we get at each point $p$ that
\begin{align}	\label{e:frechet1}
      \left|  \cN (u+v) - \cN (u) - L_u v \right| &\leq   C_1 \, \left( |v| + |\nabla v| \right)^2 + C_2 \, 
	  \left( |v| + |\nabla v| \right) \, \left| \Hess_v \right|    \, , 
 \end{align}
where the constants $C_1 = C_1 (p)$ and $C_2= C_2(p)$ are given by
\begin{align}	\label{e:lip2}
 C_1 &=\Lip_p (f_s)  + \left| u_{\alpha \beta} \right| \, 
	\Lip_p \left( \Phi^{\alpha \beta}_s \right) +  \Lip_p (f_{y_{\gamma}})  + \left| u_{\alpha \beta} \right|
	  \, \Lip_p \left( \Phi^{\alpha \beta}_{y_{\gamma}} \right) \, ,
    \\
     C_2&= \left| \Phi^{\alpha \beta}_{s} \right| +    \left| \Phi^{\alpha \beta}_{y_{\gamma}} \right|
      + \Lip_p \left( \Phi^{\alpha \beta} \right) 
        \, .
\end{align}
Here $\Lip_p$ denotes the Lipschitz norm at $p$ with respect to the $s$ and $y$ variables.  
\end{Lem}

\begin{proof}
Using \eqr{e:Luv}, we get at $p$ that for any $w$ that
\begin{align}	
	\left| L_{u+w} v - L_{u} v \right| &\leq 
	 \left| f_s (p, u+ w , \nabla u + \nabla w) - f_s (p,u,\nabla u) \right| \, |v|  \notag \\
	&\qquad
	+  
	\left| \left(u_{\alpha \beta} + w_{\alpha \beta} \right) \, \Phi^{\alpha \beta}_s  (p, u+w , \nabla u + \nabla w)
	-  u_{\alpha \beta}   \Phi^{\alpha \beta}_s  (p, u , \nabla u )\right|
	  \, |v| \notag \\
	&\qquad +  \left| f_{y_{\alpha}} (p, u+ w , \nabla u + \nabla w) - f_{y_{\alpha}}  (p,u,\nabla u) \right| \, 
	\left| v_{\alpha} \right| \notag \\
	&\qquad +
	\left| \left(u_{\alpha \beta} + w_{\alpha \beta} \right) \, \Phi^{\alpha \beta}_{y_{\gamma}}  (p, u+w , \nabla u + \nabla w)
	-  u_{\alpha \beta}   \Phi^{\alpha \beta}_{y_{\gamma}}  (p, u , \nabla u )\right| \, \left| v_{\gamma} \right| \notag \\
	&\qquad + \left| \Phi^{\alpha \beta} (p, u+w , \nabla u + \nabla w) - \Phi^{\alpha \beta} (p, u , \nabla u) \right| 
	\, \left| v_{\alpha \beta} \right| \, .
\end{align}
Bounding these terms gives
\begin{align}  \label{e:lip1}
	\left| L_{u+w} v - L_{u} v \right| &\leq  \left\{  \left[ \Lip_p (f_s)  + \left| u_{\alpha \beta} \right| \, 
	\Lip_p \left( \Phi^{\alpha \beta}_s \right) \right] \, (|w| + |\nabla w|)  +  \left| w_{\alpha \beta} \right| \, \left| \Phi^{\alpha \beta}_s \right|
	\right\} \, |v| \, 
	 \notag \\
	  &\qquad + \left\{ \left[  \Lip_p (f_{y_{\gamma}})  + \left| u_{\alpha \beta} \right|
	  \, \Lip_p \left( \Phi^{\alpha \beta}_{y_{\gamma}} \right) \right] \, (|w| + |\nabla w| ) 
	  + \left| w_{\alpha \beta} \right| \, \left| \Phi^{\alpha \beta}_{y_{\gamma}} \right|
	  \right\} \, | v_{\gamma}| \notag \\
	  &\qquad + \Lip_p \left( \Phi^{\alpha \beta} \right) \, ( |w| + |\nabla w|) \, \left| v_{\alpha \beta} \right| 
	  \, .
\end{align}

 The fundamental theorem of calculus in one variable gives
\begin{align}
  \cN (u+v) - \cN (u)  =  \int_0^1 \,  \left(  \frac{d}{dt} \big|_{t=0} \, \cN (u+tv) \right) \, dt 
    =  \int_0^1 \,  L_{u+tv} \, v \, dt \, .
\end{align}
Finally, combining this with \eqr{e:lip1}  gives that (again at $p$)
\begin{align}	 
      &\left|  \cN (u+v) - \cN (u) - L_u v \right| \leq \sup_{t\in [0,1]} \,    \left|  L_{u+tv} \, v  - L_u \, v \right| \notag \\
      &\qquad \leq  \left\{  \left| \Phi^{\alpha \beta}_{s} \right| +    \left| \Phi^{\alpha \beta}_{y_{\gamma}} \right|
      + \Lip_p \left( \Phi^{\alpha \beta} \right) 
        \right\} \, \left( |v| + |\nabla v| \right) \, 
	   \left|\Hess_{ v} \right|  \\
	   &\qquad   + \left\{ \Lip_p (f_s)  + \left| u_{\alpha \beta} \right| \, 
	\Lip_p \left( \Phi^{\alpha \beta}_s \right) +  \Lip_p (f_{y_{\gamma}})  + \left| u_{\alpha \beta} \right|
	  \, \Lip_p \left( \Phi^{\alpha \beta}_{y_{\gamma}} \right)
	\right\} \, \left( |v| + |\nabla v| \right)^2
	   \notag  \, .
 \end{align}
 \end{proof}

 \begin{proof}[Proof of Lemma \ref{l:frechet2}]
 By Lemma \ref{l:gradcm2}, $\cM (u)$  is of the form \eqr{e:geng}.
 Since $0$ is a critical point for $F$, we have $\cM (0) = 0$.  Therefore,
  Lemma \ref{l:ffu} gives
\begin{align}
  \left| \cM (u) - L u \right| \leq C_1 \, \left( |u| + |\nabla u| \right)^2 +
  C_2 \, 
    \left( |u| + |\nabla u|\right)   \left| \Hess_u \right|    \, ,
\end{align}
where the constant $C_1 = C_1 (x)$ is bounded by $C \,(1+ |x|)$ and the constant $C_2$ is 
uniformly bounded independent of $x \in \Sigma$ (these bounds follow from Lemma \ref{l:gradcm2}).

Integrating in space (against the Gaussian weight) gives
\begin{align}	\label{e:intinsp}
  \| \cM (u) - L u \|_{L^2} &\leq C \,  \left( \| (1+|x|) \, u^2\|_{L^2} + 
  \| (1+|x|) \, \,|\nabla u|^2 \|_{L^2} \right) +
  C_2 \, 
      \|  \left( |u| + |\nabla u|\right)  |\Hess_u| \|_{L^2}  \notag \\
      &\leq C \,  \left( \| (1+|x|) \, u^2\|_{L^2} + 
  \| (1+|x|) \,\,|\nabla u|^2 \|_{L^2} \right) +
  C \, 
      \|  (1+|x|)^{-1} \,  |\Hess_u|^2 \|_{L^2}   \, ,
\end{align}
where the last inequality used the absorbing inequality
\begin{align}
	2\, \left( |u| + |\nabla u|\right)  |\Hess_u| \leq (1+|x|) \, \left( |u| + |\nabla u|\right)^2
	+ (1 + |x|)^{-1} \,   |\Hess_u|^2 \, .
\end{align}
To get rid of the $|x|$'s in the first two terms, we use
Lemma \ref{l:sob} to get 
\begin{align}
		 \| |x| \, u^2 \|_{L^2} &\leq C \, \| u^2 \|_{W^{1,2}} \leq C \, \| u^2 + |\nabla u|^2 \|_{L^2} \, , \\
		 \| |x| \, |\nabla u|^2 \|_{L^2} &\leq  C \, \| |\nabla u|^2 +  \left| \nabla_{\RR^{n-k}} |\nabla u| \right|^2   \|_{L^2} \, .
\end{align}
To get the first claim, we substitute these bounds back into \eqr{e:intinsp}  
\begin{align}
  \| \cM (u) - L u \|_{L^2} &\leq C      \left\| u^2 + |\nabla u|^2 +\left| \nabla_{\RR^{n-k}} |\nabla u| \right|^2
  +   (1 + |x|)^{-1} \left| \Hess_u \right|^2  \right\|_{L^2}  
    \, .
\end{align}
  
To get the second claim, 
 we first use the fundamental theorem of calculus and the definition of the gradient to get
   \begin{align}
        F(u) - F(\cC_k) - \frac{1}{2} \, \langle u , L u \rangle_{L^2}  
      &= \int_0^1 \, \frac{d}{dt} \, \left[ F(tu) - \frac{t^2}{2}  \, \langle u , L u \rangle_{L^2} 
      \right] \, dt \notag \\
      &= \int_0^1 \langle u , \cM (tu) - t \, L u \rangle_{L^2} \, dt \, .
  \end{align}
Since the Cauchy-Schwarz inequality bounds the integrand   by
$
     \| u \|_{L^2} \, \| \cM (tu) - t \, L u \|_{L^2}$,  the second claim now follows from the first.
   \end{proof}

 \subsection{The gradient Lojasiewicz inequality}

 We will   prove Proposition \ref{p:firstsecond} and then use it to prove our
 gradient Lojasiewicz inequality using our first   Lojasiewicz inequality.

 \begin{proof}[Proof of Proposition \ref{p:firstsecond}]
 {\bf{Step 1}}: {\emph{Cutting off to get a compactly supported perturbation of the cylinder.}}
 Unlike this proposition, both Lemma \ref{l:frechet2} and
  the results of   the previous section are for entire graphs over a cylinder.  Thus, we fix a   cutoff function
 $\psi$ with  $0 \leq \psi \leq 1$ that is one on $B_{\tilde{R} -1}$ and zero outside of $B_{\tilde{R}}$ and set
 \begin{align}
    u = \psi \, \tilde{u} \, .
 \end{align}
 Observe that $u$ has
 $\| u \|_{C^{2}} \leq C_{n} \, \| \tilde{u} \|_{C^{2}} \leq  C_n \, \epsilon_0$, 
 where  $C_n$ depends on the $C^{2}$ norm of $\psi$   and, thus, depends only on $n$.  Since $\psi$ is supported in $B_{\tilde{R}}$ and $|\psi| \leq 1$, we have
 \begin{align}	\label{e:filoja4}
      \| u \|^2_{L^2} \leq \| \tilde{u} \|_{L^2(B_{\tilde{R}})}^2  \, .
\end{align}
Finally, using the exponential decay of the Gaussian, we see that
\begin{align}	\label{e:tail1}
	\left| F(\Sigma) - F( {\text{Graph}}_u ) \right| &\leq C \, \lambda_0 \, \tilde{R}^{n-1} \, 
	\e^{ -\frac{ (\tilde{R} - 1)^2}{4} } \, , \\
	\| \cM (u) \|_{L^2} &\leq C\,  \| \phi \|_{L^2 (B_{\tilde{R}})} 
	  + C_n \, \e^{ - \frac{(\tilde{R} -1)^2}{8}}
	\label{e:tail2} 
	\, ,
\end{align}
where $\phi$ here is the $\phi$ for $\Sigma$ and $C, C_n$ depend only on $n$.

 \vskip2mm
 {\bf{Step 2}}: {\emph{The gradient Lojasiewicz inequality for the compact perturbation.}}
  To simplify notation, define $F_0 (u)$ by
 \begin{align}
 	F_0 (u) = F({\text{Graph}}_u) - F(\cC_k) \, ,
 \end{align}
and, given a function $v$,   set 
 \begin{align}
     \| v \|_2 \equiv  \left\| v^2 + |\nabla v|^2 +
     \left| \Hess_{v} (\cdot , \RR^{n-k})  \right|^2
  +   (1 + |x|)^{-1} \left| \Hess_v \right|^2  \right\|_{L^2}  \, .   
 \end{align}
 Assuming that $\| u \|_{C^2}$ is sufficiently small, then  Lemma \ref{l:frechet2} 
 gives $\cwun$  so that
 \begin{enumerate}
 \item[(L1)]  $\left| \| \cM (u)\|_{L^2}  - \| L u \|_{L^2} \right| \leq \cwun \, 
 \|  u  \|_{2} $.
  \item[(L2)]   $   \left| F_0(u)  - \frac{1}{2} \, \langle u , L u \rangle_{L^2} \right| 
      \leq   \cwun \, \| u \|_{L^2} \, \|  u  \|_{2} $.
 \end{enumerate}
 Here we also used the Kato inequality 
\begin{align}
	  \left| \nabla_{\RR^{n-k}} |\nabla v| \right|^2 \leq
	   \left| \Hess_{v} (\cdot , \RR^{n-k})  \right|^2 \, .
\end{align}

\vskip2mm
We will divide into cases depending on the projection of $u$ to the kernel $\cK$ of $L$.
 Let $\cwun$ be the constant from (L1) and (L2).

 \vskip2mm
 \noindent
 {\bf{Case 1}}:  Suppose first that $u$ satisfies
 \begin{align}	\label{e:case1}
 	    \| u_{\cK}  \|_{2} \leq \eps \, \| u^{\perp} \|_{W^{2,2}}^{1+\beta}  \, ,
 \end{align}
 where $\epsilon > 0$ will be chosen below and $\beta \in [0,1)$.  
 Using the squared triangle inequality and then 
 \eqr{e:case1} plus{\footnote{Note that  $\| u^{\perp} \|_{W^{2,2}}$ is small so
 $ \| u^{\perp} \|_{W^{2,2}}^{1+\beta} \leq  \| u^{\perp} \|_{W^{2,2}}$.}}
 Lemma \ref{l:CK} gives  
  \begin{align}	\label{e:cf1}
 	\| u \|_2 \leq 2 \, \| u_{\cK} \|_2 +2 \, \| u^{\perp} \|_2 \leq 2\,
	\left( \epsilon + C_0 \| u \|_{C^2} \right)  \, \| u^{\perp} \|_{W^{2,2}}   \, .
 \end{align}
  Using (L1) and  \eqr{e:invertL}   and then using \eqr{e:cf1} gives    
   \begin{align}		\label{e:firststep}
      \| \cM (u)  \|_{L^2} &\geq   \| L u  \|_{L^{2}}    -
         \cwun \, 
      \| u \|_{ 2}  \geq  \mu \, \, \| u^{\perp} \|_{W^{2,2}}    -
         \cwun \, 
      \| u \|_{ 2}  \notag \\
        &\geq 
     \left( \mu - 2\, \cwun   \left[ \epsilon + C_0 \| u \|_{C^2} \right]
     \right) \, \| u^{\perp} \|_{W^{2,2}}    \, .
  \end{align}
  We now choose $\epsilon > 0$ and a bound for $\| u \|_{C^2}$ so that 
 $ 2\, \cwun \left[ \epsilon + C_0 \| u \|_{C^2} \right] \leq \frac{\mu}{2}$ and, thus, 
 \begin{align}		\label{e:firststep2}
      \| \cM (u)  \|_{L^2} & \geq 
    \frac{\mu}{2} \, \| u^{\perp} \|_{W^{2,2}}    \, .
  \end{align}

We will show  that $F_0 (u)$ is higher order in $\| u^{\perp} \|_{W^{2,2}}$.  Since
 $L$ is symmetric and $L u_{\cK} = 0$,    Cauchy-Schwarz   and
the   bound on $L$ from $W^{2,2}$ to $L^2$ by \eqr{e:LbdW2} give
  \begin{align}
  	\left| \langle u , L u \rangle_{L^2} \right| =  \left| \langle u^{\perp} , L u^{\perp} \rangle_{L^2} \right| \leq C \,  \| u^{\perp} \|_{L^2} 
	\, \| u^{\perp} \|_{W^{2,2}} \, .
\end{align}
 Substituting this into (L2),  using that $\| u \|_2 \leq C \, \| u^{\perp} \|_{ W^{2,2} }$ (by \eqr{e:cf1}),  and applying the triangle inequality
 $\| u \|_{L^2} \leq \| u_{\cK} \|_{L^2} + \| u^{\perp} \|_{L^2}$ gives
    \begin{align}
      \left| F_0(u) \right| 
     & \leq  C \,  \| u^{\perp} \|_{L^2} 
	\, \| u^{\perp} \|_{W^{2,2}} + C \, \| u \|_{L^2} 
	\, \| u  \|_{2}   \notag \\
	&\leq
	     C \,  \| u^{\perp} \|_{L^2} 
	\, \| u^{\perp} \|_{W^{2,2}} + C \,  \| u_{\cK} \|_{L^2} 
	\, \| u^{\perp} \|_{W^{2,2}} 
	 \, .  
  \end{align}
The first term on the right side is trivially  bounded by $C \, \| u^{\perp} \|^2_{W^{2,2}}$.  To bound the last
term, we use that  the cylinder has finite Gaussian area so that
\begin{align}
	  \| u_{\cK} \|^2_{L^2} \leq C \, \| u^2_{\cK} \|_{L^2} \leq C \, \| u_{\cK} \|_2  \, ,
\end{align}
to get that
\begin{align}
	 \| u_{\cK}  \|_{L^2} 
	\, \| u^{\perp} \|_{W^{2,2}}  \leq C \, \| u_{\cK} \|_{2}^{\frac{1}{2}} \,   
       \| u^{\perp} \|_{W^{2,2}} \leq C \, \| u^{\perp} \|^{\frac{3+\beta}{2}}_{W^{2,2}} \, ,
\end{align}
where the last inequality used
  \eqr{e:case1}.  Putting all of this together (and noting that $\| u^{\perp} \|_{W^{2,2}} $ is bounded) gives
   \begin{align}	
 	   \left| F_0(u) \right| 
     \leq C \,  \| u^{\perp} \|^{\frac{3+\beta}{2}}_{W^{2,2}}  
     \leq C\, \left\| \cM (u) \right\|^{\frac{3+\beta}{2}}_{L^2} \, ,
 \end{align}
 where the last inequality is \eqr{e:firststep2}.  Combining this with the bound on 
 $\| \cM (u) \|_{L^2}$ from
 \eqr{e:tail2} gives
 \begin{align}   \label{e:bothcases0}
    \left| F_0(u) \right| 
     \leq C 
   \| \phi \|_{L^2 (B_{\tilde{R}})}^{\frac{3+\beta}{2}}
	  + C  \, \e^{ - \frac{(3+\beta)(\tilde{R} -1)^2}{16}}
	  \, .
\end{align}

  \vskip2mm
 \noindent
 {\bf{Case 2}}:  Suppose now that $u$ satisfies
\begin{align}	\label{e:case2}
 	   \| u_{\cK}  \|_{2} > \eps \, \| u^{\perp} \|_{W^{2,2}}^{1+\beta}  \, .
 \end{align}
 Lemma \ref{l:CK} gives $C_0$ so that
 \begin{align}	\label{e:2v22}
 	\| u^{\perp} \|_2 \leq C_0 \, \| u \|_{C^2} \,  \| u^{\perp} \|_{W^{2,2}}
 	 \leq C \,  \| u_{\cK}  \|_2^{ \frac{1}{1+\beta}}   \, , 
\end{align}
where the last inequality is \eqr{e:case2}.
  Using the squared triangle inequality and
\eqr{e:2v22}  gives
 \begin{align}	\label{e:45one}
 	\| u \|_2 \leq 2 \, \| u_{\cK} \|_2 + 2 \, \| u^{\perp} \|_2 \leq
	C \, \| u_{\cK} \|_2^{ \frac{1}{1+\beta}}   \, ,
 \end{align}
 where the last inequality uses that $\| u_{\cK} \|_2$ is bounded.
  Using (L2) and \eqr{e:LbdW2},  then \eqr{e:case2} and \eqr{e:45one} (and the projection
  inequality $\| u^{\perp} \|_{L^2} \leq \| u \|_{L^2}$), 
we get
  \begin{align}	\label{e:incase2}
      \left| F_0(u) \right| 
     & \leq  C \,  \| u^{\perp} \|_{L^2} 
	\, \| u^{\perp} \|_{W^{2,2}} + \cwun \, \| u \|_{L^2} \, \| u \|_{2}
	    \leq  2\, C\, \| u \|_{L^2} \, \| u_{\cK} \|_2^{ \frac{1}{1+\beta}}  \, .
  \end{align}
  However, 
  since Lemma \ref{l:CK} 
   and  the
 projection inequality $\| u_{\cK} \|_{L^2} \leq \| u \|_{L^2}$ give that
 \begin{align}
  	\| u_{\cK} \|_2 \leq C_K \, \| u_{\cK} \|_{L^2}^2 \leq C_K \, \| u \|_{L^2}^2 \, , 
\end{align}
we conclude that
\begin{align}
	 \left| F_0(u) \right| 
      \leq  C \,   \| u \|_{L^2}^{ \frac{3+\beta}{1+\beta}} 
      \leq   C \, \| \tilde{u} \|_{L^2(B_{\tilde{R}})}^{ \frac{3+\beta}{1+\beta}}   \, ,
\end{align}
where the last inequality is the Gaussian $L^2$ bound on $u$ from \eqr{e:filoja4}.

 \vskip2mm
 If we now combine the bounds from the two cases, then we see that 
 \begin{align}	\label{e:finalstep3}
	 \left| F_0(u) \right| 
      &\leq C 
   \| \phi \|_{L^2 (B_{\tilde{R}})}^{\frac{3+\beta}{2}}
	  + C  \, \e^{ - \frac{(3+\beta)(\tilde{R} -1)^2}{16}} + C \, \| \tilde{u} \|_{L^2(B_{\tilde{R}})}^{ \frac{3+\beta}{1+\beta}}  \, .
\end{align}
 Finally, we use the triangle inequality to combine this with the bound \eqr{e:tail1}   on the
$F$ functional from Step 2 and Step 3
to complete the proof.

 \end{proof}

We will use the following elementary lemma to control graphical bounds when we write a surface as a graph over two nearby cylinders.

\begin{Lem}	\label{l:rotate}
There exists $\epsilon_0 = \epsilon_0 (n) > 0$ so that if  $\Sigma_1 , \Sigma_2 
\in \cC_k$, $5 \sqrt{2n} \leq R_1 < R_2$ and 
\begin{itemize}
\item  $B_{R_1} \cap \Sigma$ is the graph of $u_1$ over $\Sigma_1$ with $|u_1| + |\nabla u_1| \leq \epsilon_0$,
\item   $B_{R_2} \cap \Sigma$ is the graph of $u_2$ over $\Sigma_2$ with $\| u_2 \|_{C^{2,\alpha}} \leq \epsilon_0$, 
\end{itemize}
then we get for $R = \min \{ 2\, R_1 , R_2 \}$ that
\begin{itemize}
\item   $B_{R} \cap \Sigma$ is the graph of $u_1$ over $\Sigma_1$ with $\| u_1 \|_{C^{2}} \leq   \epsb$.
\end{itemize}
\end{Lem}

\begin{proof}
Since $B_{R_1} \cap \Sigma$ is $\epsilon_0$ $C^1$-close to $\Sigma_1$ and $\epsilon_0$ close to $\Sigma_2$, we get that the distance between $\Sigma_1$ and $\Sigma_2$ in $B_{R_1}$ is at most $2\epsilon_0$.  Since the distance between cylinders grows linearly in the radius, we conclude that the distance between $\Sigma_1$ and $\Sigma_2$
in $B_R$ is at most $4\epsilon_0$.  The lemma follows easily from this.
\end{proof}

 \begin{proof}[Proof of Theorem \ref{t:ourgradloja}]  
 The result will follow by combining the $L^2$ closeness to a cylinder given by the first Lojasiewicz inequality 
and  Proposition
\ref{p:firstsecond}.    Note that we can assume that $R$ is large and $\| \phi \|_{L^2(B_R)}$ is small since the inequality is otherwise trivially true.
 
 \vskip2mm
 {\bf{Step 1}}:  {\emph{Fixing the nearby cylinder.}}  The Lojasiewicz inequality 
 of Theorem \ref{t:ourfirstloja} 
 gives a cylinder $\Sigma_k \in \cC_k$  
 so that $B_{\tilde{R}} \cap \Sigma$ is the graph of $\tilde{u}$ over $\Sigma_k$ with $\| \tilde{u} \|_{C^{1}} \leq  \epsilon_0$,
 where  $b_{\ell , n} \in (0,1)$ satisfies $\lim_{\ell \to \infty} \, b_{\ell , n} = 1$, 
  and{\footnote{We choose $\tilde{R}$ to make $\| \tilde{u} \|_{C^1}$ small by \eqr{e:bdU}.}}
 \begin{align}	\label{e:bdUnew}
  \tilde{R} = \max \, \left\{ r \leq R \, \big| \, R^{2n+5 } \,   \left\{ \e^{ - b_{\ell , n } \, \frac{R^2}{8} }+ \| \phi \|_{L^2(B_R)}^{ \frac{b_{\ell , n}}{2} } 
    \right\} 
      \, \e^{ \frac{r^2}{8} }  \leq \tilde{C} \right\} \, ,
\end{align}
where $\tilde{C}$ depends on $n , \lambda_0 , \ell , C_{\ell}$.  Combining this with Lemma \ref{l:rotate}, we extend $\tilde{u}$ out 
to $\bar{R} = \min \, \{ 2 \tilde{R} , R \}$ so that
  \begin{itemize}
  \item[($\star_1$)] $B_{\bar{R}} \cap \Sigma$ is the graph of $\tilde{u}$ over $\Sigma_k$ with $\| \tilde{u} \|_{C^{2}} \leq  \epsb$,
  \item[($\star_2$)]  
  $\| \tilde{u} \|^2_{L^2( B_{\bar{R} } )} \leq  C \,  R^{\rho}\, \left\{ 
        \|  \phi \|_{  L^2(B_R )  }^{ b_{\ell,n}  }   
       +  \e^{ - \frac{ b_{\ell , n} \, R^2}{4}   } \right\}$,

  \end{itemize}
where     $C = C(n,\ell, C_{\ell}, \lambda_0)$ and  $\rho = \rho (n)$.
 
\vskip2mm
{\bf{Step 2}}: {\emph{Using the first Lojasiewicz to get the second}}.
  Proposition
\ref{p:firstsecond} gives 
 \begin{align}	 \label{e:gradLgives}
	 \left| F (\Sigma) - F (\cC_k) \right| 
      &\leq C \left\{ 
   \| \phi \|_{L^2 (B_{\bar{R}})}^{\frac{3+\beta}{2}}
	  +  (1 + \bar{R}^{n-1}) \, \e^{ - \frac{(3+\beta)(\bar{R} -1)^2}{16}} +
        \| \tilde{u} \|_{L^2(B_{\bar{R}})}^{ \frac{3+\beta}{1+\beta}}  \right\} 
         \, ,
\end{align}
where   $C = C (n , \lambda_0)$.
To bound the last term in \eqr{e:gradLgives}, we use ($\star_2$) to get
 \begin{align}	 \label{e:star2bd}
  \| \tilde{u} \|_{L^2(B_{\bar{R}})}^{ \frac{3+\beta}{1+\beta}}       &\leq C \,R^{ \frac{3+\beta}{2 + 2\beta} \, \rho} \left\{ 
   \| \phi \|_{L^2 (B_{R})}^{b_{\ell , n} \, \frac{3+\beta}{2+2 \beta}}
	  +    \e^{ - \frac{b_{\ell,n} \,(3+\beta) R^2}{8(1+\beta)}}  
        \right\} 
         \, .
\end{align}
%Since $\beta \in [0,1)$, the  $\| \phi \|_{L^2}$ term  in \eqr{e:gradLgives} can be absorbed in the $\| \phi \|_{L^2}$ term in \eqr{e:star2bd}.  

To deal with the exponential term in \eqr{e:gradLgives}, we consider two cases.  Suppose first that $\bar{R} < R$, so that $\bar{R} = 2 \tilde{R}$ and  the definition of $\tilde{R}$ gives  
\begin{align}
	R^{2n+5 } \,   \left\{ \e^{ - b_{\ell , n } \, \frac{R^2}{8} }+ \| \phi \|_{L^2(B_{R})}^{ \frac{b_{\ell , n}}{2} } 
    \right\} 
      \, \e^{ \frac{\tilde{R}^2}{8} }  =  \tilde{C} \, .
\end{align}
Since  $\bar{R} = 2 \tilde{R}$ in this case, we have
\begin{align}	\label{e:raisethis}
	\e^{ - \frac{\bar{R}^2}{8} } =\left[  \e^{ - \frac{\tilde{R}^2}{8} } \right]^4 \leq C \, R^{8n+20 }     \left\{ \e^{ - b_{\ell , n } \, \frac{R^2}{2} }+ \| \phi \|_{L^2(B_{R})}^{2  b_{\ell , n} } 
    \right\} 
 \, .
\end{align}
We can assume that $\bar{R} > 4$ so that  $\left( \frac{ \bar{R} -1}{\bar{R}} \right)^2 > 1/2$.  
Raising \eqr{e:raisethis} to the $\frac{3+\beta}{2} \, \left( \frac{ \bar{R} -1}{\bar{R}} \right)^2 > \frac{3+\beta}{4} $ power, we bound the exponential term in 
\eqr{e:gradLgives} by a constant times a power of $R$ times
\begin{align}
	     \e^{ - b_{\ell , n } \, \frac{(3+\beta)R^2}{8} }+ \| \phi \|_{L^2(B_{R})}^{ b_{\ell , n} \frac{(3+\beta) }{2} }  \leq \e^{ - \frac{(3+\beta)(R -1)^2}{16} }
	     +  \| \phi \|_{L^2(B_{R})}^{ b_{\ell , n} \frac{(3+\beta) }{2} }
 \, ,
\end{align}
where the last inequality used also that  $b_{\ell , n}$ is close to one, and in particular at least $1/2$.  We proved this inequality in the case where
$\bar{R} < R$, but it also obviously holds in the case when
when $\bar{R} = R$ (and the $  \phi $ term is unnecessary).

\vskip1mm
Putting it all together,   $ \left| F (\Sigma) - F (\cC_k) \right| $ is bounded by  $C \, R^{\rho'}$ times 
 \begin{align}	 % \label{e:cbndf}
	   \| \phi \|_{L^2 (B_{\bar{R}})}^{\frac{3+\beta}{2}} +  \left\{  \| \phi \|_{L^2 (B_{R})}^{b_{\ell , n} \, \frac{3+\beta}{2+2 \beta}}
	  +    \e^{ - \frac{b_{\ell,n} \,(3+\beta) R^2}{8(1+\beta)}}  \right\} +  
	  \left\{ \e^{ - \frac{(3+\beta)(R -1)^2}{16} }
	     +  \| \phi \|_{L^2(B_{R})}^{ b_{\ell , n} \frac{(3+\beta) }{2} }
	  \right\}   \, ,   \notag
\end{align}
where we have grouped terms together based on where they came from in \eqr{e:gradLgives}.  Finally,   the first and fifth  terms can be absorbed in the second term.

\end{proof}

\section{Compatibility of the shrinker and cylindrical scales}

One of the main difficulties in this paper is that the singularities are not compact and, thus,  surfaces cannot generally be written as entire graphs over a cylinder.
As a result, our estimates include ``error terms'' coming from cut off functions.   Thus,  a surface is close to the cylinder if a large part of it can be written as a small graph over the cylinder.   

Given a hypersurface $\Sigma\subset \RR^{n+1}$, we will prove a lower bound for the scale on which it is ``roughly cylindrical'' in Theorem \ref{t:rR} below.
  This essentially bounds the error terms in our Lojasiewicz inequalities by a power greater than one of $|\nabla_{\Sigma} F|$, which is crucial in the next section when 
we prove uniqueness of tangent flows.  It will also imply that the size of the graphical region is growing at a definite rate under the rescaled MCF.

\subsection{The cylindrical scale and the shrinker scale}

 Recall that the cylindrical scale $\graph (\Sigma)$  is the largest radius where $\Sigma$ can be written as a small $C^{2,\alpha}$ graph over a cylinder 
 with a uniform bound on
$\nabla^{\ell} A$.  Namely, given a fixed $\epsilon_0 > 0$, an integer $\ell$ and a  constant
$C_{\ell}$,  $\graph (\Sigma)$ is   the maximal radius where
\begin{itemize}
 \item  $B_{\graph (\Sigma)} \cap \Sigma$ is the graph over a cylinder in
 $\cC_k$
 of a function $u$  with $\| u \|_{C^{2,\alpha}} \leq \epsilon_0$ and $|\nabla^{\ell} A| \leq C_{\ell}$.
\end{itemize}
The constant $\epsilon_0$ is fixed, but we have yet to choose $\ell$ and $C_{\ell}$.   (The constant $\ell$ will be chosen large to get good bounds on lower derivatives by interpolation and then $C_{\ell}$ will   be chosen.)

The point of this section is to prove that these cylindrical scales are large enough that the error terms in 
our Lojasiewicz inequalities can be absorbed.   The scale $R$  that we have to beat{\footnote{To see this,   note that the larger exponentially decaying term on the right in Theorem \ref{t:ourgradloja} is essentially $\e^{ - \frac{R^2}{4} }$.  We need to bound this by a power greater than one of $|\nabla_{\Sigma} F| = \| \phi \|_{L^2}$. }} is roughly given by
$\e^{ - \frac{R^2}{4} } =  |\nabla_{\Sigma} F| $.  Thus, we define
  a ``shrinker scale''  $\shrink (\Sigma)$  by
\begin{align}
  \e^{ - \frac{\shrink^2(\Sigma)}{2} } =  |\nabla_{\Sigma} F|^2 \, ,
\end{align}
with the convention that $\shrink (\Sigma)$ is infinite  when $\Sigma$ is a complete shrinker. 
When $\Sigma_t$ flows by the rescaled MCF, we define the shrinker scale (also denoted by $\shrink (\Sigma_t)$) to be 
\begin{align}	\label{e:shrinkdef}
  \e^{ - \frac{\shrink^2(\Sigma_t)}{2} } = \int_{t-1}^{t+1}|\nabla_{\Sigma_s} F|^2 \, ds =F(\Sigma_{t-1})-F(\Sigma_{t+1})\, .
\end{align}

 \vskip2mm
The main  result of this section is the following theorem that shows that the cylindrical scale is a fixed factor larger than the shrinker scale :

\begin{Thm}     \label{t:rR}
There exist $\mu > 0$ and $C$   so that
if $\Sigma_t$ flows by the rescaled MCF and $\lambda (\Sigma_t)\leq \lambda_0$, then, given any $\ell$, there exists $C_{\ell}$ (depending on $\ell$) so that
% \begin{align}
 %  (1+\mu) \, \shrink  (\Sigma_t)\leq \min_{t-1/2\leq s\leq t+1}\graphnoell (\Sigma_s) +C \, .
%\end{align}
%Furthermore
 \begin{align}
   (1+\mu) \, \shrink  (\Sigma_t)\leq \min_{t-1/2\leq s\leq t+1}\graph (\Sigma_s) +C \, .
\end{align}
\end{Thm}

\vskip3mm
To understand this, observe that Theorem
\ref{t:ourfirstloja3} gives uniform graphical estimates on  any scale less than $\shrink (\Sigma)$.  To apply
Theorem
\ref{t:ourfirstloja3}, we need uniform curvature bounds and 
  a   lower bound for $H$ on this larger scale.  We will establish these uniform bounds on the larger scale 
   by an extension and improvement argument, where Theorem
\ref{t:ourfirstloja3} gives uniform bounds on  larger scales (this is the improvement part).  Roughly speaking, the extension argument will  
  use curvature estimates for MCF to get bounds on a larger scale forward in time, then use the bounds on $\phi$ to pull these bounds backwards in time
  under the rescaled MCF.   Repeating this gets us as close as we want to the scale $\shrink (\Sigma)$ and gets us uniform curvature bounds on a larger scale than $\shrink (\Sigma)$.  The final step is to get graphical estimates on a larger scale too; for this, we cannot use 
  Theorem
\ref{t:ourfirstloja3}.  Rather, we get these graphical estimates from estimates for MCF and a scaling argument to relate MCF and
 rescaled MCF.{\footnote{This final step can not be iterated since it has a loss in the estimates and we can no longer apply 
 Theorem
\ref{t:ourfirstloja3}  to get rid of the loss.}}

\subsection{Backward curvature estimates}

Recall that, when $\Sigma\subset \RR^{n+1}$ is a hypersurface,  $\phi$ is defined to be
$
\phi=  \frac{\langle x , \nn \rangle}{2} - H$, 
so that $\Sigma_t$ flows by the rescaled MCF if $\partial_t x = \phi \, \nn$ and 
\begin{align}
	|\nabla_{\Sigma} F|^2= \| \phi \|^2_{L^2}  \equiv \int_{\Sigma}\phi^2\,\e^{-\frac{|x|^2}{4}}\, .
\end{align}

In this subsection, we will show the following curvature estimate for the rescaled MCF:

\begin{Pro}	\label{p:backwards}
Given $n$, $\lambda_0$, there exists $s\geq 2$ and $\delta > 0$ so that the following holds:  Given $1/2\geq \tau>0$, there exists $\mu >0$,  such that if $\Sigma_t$ flows by the rescaled MCF, $\lambda_0\geq \lambda (\Sigma_t)$, $t_2 \geq t_1 + \tau$, $x_0\in B_{R-s}$ and
\begin{align}
\int_{t_1}^{t_2}\int_{B_{R+2}\cap \Sigma_t}\phi^2\,\e^{-\frac{|x|^2}{4}}&\leq \frac{\mu^2\, \e^{-\frac{(R+2)^2}{4}}}{R^2\, (t_2-t_1+1)}\, ,\\
\sup_{B_{s \sqrt{\tau}}(x_0)\cap\Sigma_{t_2}}|A|^2&\leq \delta\,\tau^{-1}\, ,
\end{align}
then for all $t \in \left[     t_1 - \log (1 - 7\tau/8) ,   t_1 - \log (1-\tau)
	\right] $ and any $\ell$
\begin{align}
\sup_{B_{\frac{\sqrt{\tau}}{3}}  \left(  \e^{ \frac{1}{2} (t-t_1) } \, x_0 \right)\cap \Sigma_{t}} \left\{  |A|^2 + \tau^{\ell} \, \left| \nabla^{\ell} A \right|^2
\right\} &\leq \frac{C_{\ell}}{ \tau }\, ,
\end{align}
where $C_{\ell}$ depends on $n$ and $\ell$.
\end{Pro}

\vskip2mm
We will need the following elementary lemma:

\begin{Lem}	\label{l:lemmon}
If $\Sigma_t$ flows by the rescaled MCF and $0\leq \eta\leq 1$ is a smooth compactly-supported function on $\RR^{n+1}$, then for $t_1\leq t_2$ and $\tau>0$
\begin{align}
		\int_{\Sigma_{t_2}}\eta\,\e^{-\frac{|x-x_0|^2}{4\tau}}&-\int_{\Sigma_{t_1}}\eta\,\e^{-\frac{|x-x_0|^2}{4\tau}}\notag\\
	=&\int_{t_1}^{t_2}\int_{\Sigma_t}\langle \nabla \eta,\nn\rangle\,\phi\,\e^{-\frac{|x-x_0|^2}{4\tau}}
	   +\int_{t_1}^{t_2}\int_{\Sigma_t}\frac{\langle x_0,\nn\rangle}{2\tau}\,\phi\,\eta\,\e^{-\frac{|x-x_0|^2}{4\tau}}\\
	& +\left(1-\frac{1}{\tau}\right)\int_{t_1}^{t_2}\int_{\Sigma_t}\frac{\langle x,\nn\rangle}{2}\,\phi\,\eta\,\e^{-\frac{|x-x_0|^2}{4\tau}}
		-\int_{t_1}^{t_2}\int_{\Sigma_t}\,\phi^2\,\e^{-\frac{|x-x_0|^2}{4\tau}}\, .\notag
\end{align}
\end{Lem}

\begin{proof}
If $f\in \RR^{n+1}\to \RR$ is a smooth function with compact support, then
\begin{align}
\frac{d}{dt}\int_{\Sigma_t}f\,\e^{-\frac{|x|^2}{4}}&= \int_{\Sigma_t}\langle \nabla f,\nn\rangle\,\phi\,\e^{-\frac{|x|^2}{4}}-\int_{\Sigma_t}f\,\phi^2\,\e^{-\frac{|x|^2}{4}}\notag\\
&=  \int_{\Sigma_t}\langle \nabla \log f,\nn\rangle\,\phi\,f\,\e^{-\frac{|x|^2}{4}}-\int_{\Sigma_t}\phi^2\,f\,\e^{-\frac{|x|^2}{4}} \, .
\end{align}
If we set
$
f(x)=\eta\,\e^{\frac{|x|^2}{4}}\,\e^{-\frac{|x-x_0|^2}{4\tau}}
$, 
then
\begin{align}
	\nabla \log f&=\nabla \log \eta+\frac{x}{2}-\frac{x-x_0}{2\tau}=\nabla \log \eta+\frac{x_0}{2\tau}+\frac{1}{2}\,\left(1-\frac{1}{\tau}\right)\,x\, .
\end{align}
Therefore 
\begin{align}
\frac{d}{dt}\int_{\Sigma_t}\eta\, \e^{-\frac{|x-x_0|^2}{4\tau}}&=\int_{\Sigma_t}\langle \nabla \eta,\nn\rangle\,\phi\,\e^{-\frac{|x-x_0|^2}{4\tau}}+ \int_{\Sigma_t}\frac{\langle x_0,\nn\rangle}{2\tau}\,\phi\,\eta\,\e^{-\frac{|x-x_0|^2}{4\tau}}\notag\\
&+ \left(1-\frac{1}{\tau}\right)\int_{\Sigma_t}\frac{\langle x,\nn\rangle}{2}\,\phi\,\eta\,\e^{-\frac{|x-x_0|^2}{4\tau}}
-\int_{\Sigma_t}\,\phi^2\,\eta\,\e^{-\frac{|x-x_0|^2}{4\tau}}\, .
\end{align}
The lemma now follows by integrating from $t_1$ to $t_2$.
\end{proof}

\begin{Cor}	\label{c:gaussianbound}
Given $\epsilon>0$, $1\geq \tau>0$, and $\lambda_0$, there exists $\mu=\mu(\epsilon, \tau, \lambda_0)>0$, $s=s(\epsilon, \lambda_0)\geq 2$ such that if $\Sigma_t$ flows by the rescaled MCF, $\lambda_0\geq \lambda (\Sigma_t)$, $x_0\in B_{R-s}$, and $t_2>t_1$
\begin{align}
\int_{t_1}^{t_2}\int_{B_{R+2}\cap \Sigma_t}\phi^2\,\e^{-\frac{|x|^2}{4}}&\leq \frac{\mu^2\, \e^{-\frac{(R+2)^2}{4}}}{R^2\, (t_2-t_1+1)}\, ,\\
(4\pi\tau)^{-\frac{n}{2}}\int_{\Sigma_{t_2}}\e^{-\frac{|x-x_0|^2}{4\tau}}&\leq 1+\frac{\epsilon}{2}\, ,
\end{align}
then
\begin{align}
(4\pi\tau)^{-\frac{n}{2}}\int_{\Sigma_{t_1}}\e^{-\frac{|x-x_0|^2}{4\tau}}&\leq 1+\epsilon\, .
\end{align}
\end{Cor}

\begin{proof}
Observe first that by the entropy bound $\lambda (\Sigma_t)\leq \lambda_0$, there exists $s>0$ such that for all $y\in \RR^{n+1}$ and all $t$ 
\begin{align}   \label{e:away}
(4\pi\tau)^{-\frac{n}{2}}\int_{\Sigma_{t}\setminus B_{s \sqrt{\tau}}(y)}  \e^{-\frac{|x-y|^2}{4\tau}}\leq \frac{\epsilon}{4}\, .
\end{align}
 If we choose a non-negative function $\eta$ with $\eta\leq 1$, $|\nabla \eta|\leq 1$, $\eta=1$ on $B_R$,  and $\eta=0$ outside 
$B_{R+2}$, then Lemma \ref{l:lemmon} gives
\begin{align}  
\int_{B_R\cap \Sigma_{t_1}}  \e^{-\frac{|x-x_0|^2}{4\tau}}&\leq \int_{\Sigma_{t_1}}\eta\,  \e^{-\frac{|x-x_0|^2}{4\tau}}\leq \int_{B_{R+2}\cap\Sigma_{t_2}}\e^{-\frac{|x-x_0|^2}{4\tau}}+\int_{t_1}^{t_2}\int_{(B_{R+2}\setminus B_R)\cap \Sigma_t}|\phi |\, \e^{-\frac{|x-x_0|^2}{4\tau}}\notag\\
&+\left(\frac{1}{\tau}-1\right)\int_{t_1}^{t_2}\int_{B_{R+2}\cap\Sigma_t}\frac{|\langle x,\nn\rangle|}{2}\,|\phi| \,\e^{-\frac{|x-x_0|^2}{4\tau}}
+\int_{t_1}^{t_2}\int_{B_{R+2}\cap\Sigma_t}\frac{|\langle x_0,\nn\rangle|}{2\tau}\,|\phi|\,\e^{-\frac{|x-x_0|^2}{4\tau}}\notag\\
&+\int_{t_1}^{t_2}\int_{B_{R+2}\cap \Sigma_t}\,\phi^2\,\e^{-\frac{|x-x_0|^2}{4\tau}} \, .
\end{align}
Combining the terms that are linear in $\phi$ gives
\begin{align}    \label{e:bulk}
   \int_{B_R\cap \Sigma_{t_1}}  \e^{-\frac{|x-x_0|^2}{4\tau}} &\leq  \int_{B_{R+2}\cap\Sigma_{t_2}}\e^{-\frac{|x-x_0|^2}{4\tau}}+\left(1+\frac{ \frac{|x_0|}{\tau}+\left(\frac{1}{\tau}-1\right)(R+2)}{2}\right)\,\int^{t_2}_{t_1}\int_{B_{R+2}\cap\Sigma_t}|\phi|\, \e^{-\frac{|x-x_0|^2}{4\tau}} \notag \\
   &\qquad 
+\int^{t_2}_{t_1}\int_{B_{R+2}\cap \Sigma_t}\,\phi^2\notag\\
&\leq\int_{B_{R+2}\cap\Sigma_{t_2}}\e^{-\frac{|x-x_0|^2}{4\tau}}+\frac{R+2}{\tau}\,\int^{t_2}_{t_1}\int_{B_{R+2}\cap\Sigma_t}|\phi|\, \e^{-\frac{|x-x_0|^2}{4\tau}}
+\int^{t_2}_{t_1}\int_{B_{R+2}\cap \Sigma_t}\,\phi^2\, .
\end{align}
By the entropy bound $\lambda (\Sigma_t)\leq \lambda_0$ and the Cauchy-Schwarz inequality, we have
\begin{align}   \label{e:}
\int^{t_2}_{t_1}\int_{B_{R+2}\cap\Sigma_t}|\phi|\, \e^{-\frac{|x-x_0|^2}{4\tau}}&\leq \left(\int^{t_2}_{t_1}\int_{B_{R+2}\cap\Sigma_t}\phi^2\, \e^{-\frac{|x-x_0|^2}{4\tau}}\right)^{\frac{1}{2}}\,\sqrt{ (4\pi\tau)^{\frac{n}{2}}\,(t_2-t_1)\, \lambda_0}\notag\\
&\leq \sqrt{  (4\pi\tau)^{\frac{n}{2}}\,(t_2-t_1)\,\lambda_0}\left(\int^{t_2}_{t_1}\int_{B_{R+2}\cap\Sigma_t}\phi^2\right)^{\frac{1}{2}}\\
&\leq  \sqrt{ (4\pi\tau)^{\frac{n}{2}}\,(t_2-t_1)\, \lambda_0}\, \, \e^{\frac{(R+2)^2}{8}}\,\left(\int^{t_2}_{t_1}\int_{B_{R+2}\cap\Sigma_t}\phi^2\,\e^{-\frac{|x|^2}{4}}\right)^{\frac{1}{2}}
\, .\notag
\end{align}
We can therefore bound the two last terms in \eqr{e:bulk} as follows
\begin{align}   \label{e:boundbulk}
\frac{R+2}{\tau}&\,\int^{t_2}_{t_1}\int_{B_{R+2}\cap\Sigma_t}|\phi|\, \e^{-\frac{|x-x_0|^2}{4\tau}}
+\int^{t_2}_{t_1}\int_{B_{R+2}\cap \Sigma_t}\,\phi^2\notag\\
&\leq C\, \frac{R+2}{\tau}\,\sqrt{(4\pi\tau)^{\frac{n}{2}}\,(t_2-t_1)}\, \e^{\frac{(R+2)^2}{8}}\,\left(\int^{t_2}_{t_1}\int_{B_{R+2}\cap\Sigma_t}\phi^2\,\e^{-\frac{|x|^2}{4}}\right)^{\frac{1}{2}}
+\e^{\frac{(R+2)^2}{4}}\,\int^{t_2}_{t_1}\int_{B_{R+2}\cap\Sigma_t}\phi^2\,\e^{-\frac{|x|^2}{4}}\notag\\
&\leq  C\, (\mu/\tau +\mu^2)\, .   
\end{align}
Using  \eqr{e:away}, \eqr{e:bulk}, and \eqr{e:boundbulk} we get that
\begin{align}   \label{e:}
(4\pi\tau)^{-\frac{n}{2}}&\int_{\Sigma_{t_1}}  \e^{-\frac{|x-x_0|^2}{4\tau}}= (4\pi\tau)^{-\frac{n}{2}}\int_{B_R\cap \Sigma_{t_1}}  \e^{-\frac{|x-x_0|^2}{4\tau}}+ (4\pi\tau)^{-\frac{n}{2}}\int_{\Sigma_{t_1}\setminus B_R}  \e^{-\frac{|x-x_0|^2}{4\tau}}\notag\\
&\leq (4\pi\tau)^{-\frac{n}{2}}\int_{\Sigma_{t_2}}\e^{-\frac{|x-x_0|^2}{4\tau}}+C\, (\mu/\tau
+\mu^2)+\frac{\epsilon}{4}\, .
\end{align}
Choosing $\mu$ sufficiently small gives the corollary.
\end{proof}

We will apply this corollary in combination with Brakke's regularity result to get curvature estimates at an earlier time-slice in terms of curvature estimates at a later time-slice.  By White's \cite{W3} version of Brakke's regularity result \cite{B}, there exist  constants $\epsilon$ and $C_B$ depending on $n$ and $ \lambda_0 $  such that if $M_s\subset \RR^{n+1}$ flow ($s<0$) by the MCF, $\lambda (M_s)\leq \lambda_0$, and for some $s_0<0$ 
\begin{align}
(-4\pi s_0)^{-\frac{n}{2}}\int_{M_{s_0}}  \e^{\frac{|x-x_0|^2}{4s_0}}\leq 1+\epsilon\, ,
\end{align}
then for all $s \in [ -\frac{s_0}{4},0]$
\begin{align}
\sup_{M_s \cap B_{\frac{1}{2}\sqrt{-s_0}}(x_0)}|A|^2\leq \frac{C_B}{-s_0}\, .
\end{align}
We can use the correspondence between MCF and rescaled MCF to translate this into a similar curvature estimate for rescaled MCF.  Namely, if 
$\Sigma_t$ is a rescaled MCF with entropy at most $\lambda_0$ and there is some $\tau \in (0, 1/2)$ so that
\begin{align}   \label{e:brakke}
(4\pi\tau)^{-\frac{n}{2}}\int_{\Sigma_{t_0}}  \e^{-\frac{|x-x_0|^2}{4\tau}}\leq 1+\epsilon\, ,
\end{align}
then for all $t \in [ t_0 - \log (1 - 3\tau/4) ,   t_0 - \log (1-\tau) ]$ we have
\begin{align}	\label{e:brombrakke}
  \sup_{ \Sigma_t \cap B_{ \frac{ \sqrt{\tau} }{2} } \left( \e^{ \frac{1}{2} \, \left( t - t_0 \right) } x_0 \right) }  |A|^2 \leq \frac{C_B}{ \tau } \, .
\end{align}
This is proven by writing the rescaled flow $\Sigma_t$ as  $\e^{ \frac{1}{2} \, \left( t - t_0 \right) }  \, M_s$ where  where $s = 1 -  \e^{t_0-t} - \tau$ and $M_s$ is the MCF with $M_{-\tau} = \Sigma_{t_0}$.
(Here we have used that the result of Brakke/White is uniform in $\Sigma$ or more precise uniform in the point $x_0$ where it is centered as  for the rescaled MCF when the point $x_0$ is fixed this mean that the original ``fixed'' point $x_0$ for the MCF evolves by $\e^{ \frac{1}{2} \, \left( t - t_0 \right) } x_0$.) 

\vskip2mm
\begin{proof}[Proof of Proposition \ref{p:backwards}]
Combining the above consequence of Brakke's theorem 
with Corollary \ref{c:gaussianbound} gives the $|A|$ bound in Proposition \ref{p:backwards} for $t$ in the time interval
\begin{align}
	\left[     t_1 - \log (1 - 3\tau/4) ,   t_1 - \log (1-\tau)
	\right] \, .
\end{align}

The bounds on higher derivatives of $A$ then follow from this and the interior estimates of Ecker and Huisken, \cite{EH}.
\end{proof}

\subsection{A mean value inequality}

In the next lemma, we will use that if $\Sigma_t$ flow by the rescaled MCF, then (see section 2 of \cite{CIMW})
\begin{align}
(\partial_t-L)\,\phi =0   {\text{ where }}  L =\cL+|A|^2+\frac{1}{2} \, . \label{e:eqrphi} 
\end{align}
Hence,
\begin{align}	\label{e:usephi}
(\partial_t-\cL)\,\phi^2&=2\,\phi\, (\partial_t-\cL)\,\phi-2\,|\nabla \phi|^2=\phi^2\,(2\,|A|^2+1)-2\,|\nabla \phi|^2\, .
\end{align}

\begin{Lem}  \label{l:meanvalue}
There exists a constant $C$ so that if
 $\Sigma_t$ flow by the rescaled MCF for $t \in [t_1 , t_2]$,  $r + 1 \leq \min_{t_1 \leq s\leq t_2}\graphnoell (\Sigma_s)$ and  $0 < \beta < (t_2-t_1)/2$, 
then  
\begin{align}
  \max_{s \in [t_1 + \beta , t_2 ]} \, \, \left|\nabla_{\Sigma_s} F \right|^2_{B_r} 
&\leq
(C+1/\beta) \,  \left(F(\Sigma_{t_1})-F(\Sigma_{t_2})\right)
\, .\label{e:e1a}\\
\int_{t_1+ \beta}^{t_2} \int_{B_r\cap\Sigma_s}|\nabla \phi|^2\,\e^{-\frac{|x|^2}{4}}
&\leq
(C+ 1/\beta)  \,\left(F(\Sigma_{t_1})-F(\Sigma_{t_2})\right)
\, .\label{e:e1b}
\end{align}
\end{Lem}

\begin{proof}
Fix a compactly supported function $\eta$ on $\RR^{n+1}$ with
  $1\leq \eta\leq 0$, $\eta$ identically one on $B_r$, $\eta$ vanishes outside $B_{r+1}$,  and $| \nabla \eta | \leq 2$.  If we restrict $\eta$ to $\Sigma_t$, then 
  the flow equation and \eqr{e:usephi} give
  \begin{equation}
  	\partial_t \, \left( \phi^2 \, \eta^2 \right) = (\eta^2)_t \, \phi^2 + \eta^2 \, \partial_t \phi^2 =
	\phi^3 \, \langle \nabla \eta^2 , \nn \rangle + \eta^2 \, \left( \cL \phi^2 
	+\phi^2\,(2\,|A|^2+1)-2\,|\nabla \phi|^2 \right) \, .
  \end{equation}
  Using this and the equation for the derivative of the weighted measure, and integrating by parts to take $\cL$ off of $\phi^2$,
  we get 
\begin{align}       
	\partial_t\left(\int_{\Sigma_t}\phi^2\,\eta^2 \,\e^{-\frac{|x|^2}{4}}\right)&=-\int_{\Sigma_t}\phi^4\,\eta^2 \,\e^{-\frac{|x|^2}{4}}+\int_{\Sigma_t}\left(2\,|A|^2+1\right)\phi^2\,\eta^2\,\e^{-\frac{|x|^2}{4}} \notag\\
	&-2\,\int_{\Sigma_t}|\nabla \phi|^2\,\eta^2 \,\e^{-\frac{|x|^2}{4}}+ \int_{\Sigma_t}\phi^3\,\langle \nabla \eta^2, \nn\rangle\,\,\e^{-\frac{|x|^2}{4}}  -   \int_{\Sigma_t}  \langle \nabla \phi^2 , \nabla  \eta^2 \rangle \,\e^{-\frac{|x|^2}{4}} \, .
\end{align}
Using the absorbing inequalities $2 \phi^3 \eta |\nabla \eta| \leq \phi^4 \eta^2 + \phi^2 |\nabla \eta|^2$ and
$4 \, \eta |\phi| |\nabla \eta | | \nabla \phi| \leq  \eta^2 |\nabla \phi|^2 +4 \phi^2 |\nabla \eta|^2$, we get
\begin{align}	 \label{e:enu}
	\partial_t\left(\int_{\Sigma_t}\phi^2\,\eta^2 \,\e^{-\frac{|x|^2}{4}}\right) \leq &\int_{\Sigma_t}
	\left\{ \left(2\,|A|^2+1 \right)\eta^2 + 5 |\nabla \eta|^2 \right\} \phi^2\, \e^{-\frac{|x|^2}{4}}  - 
	\int_{\Sigma_t} |\nabla \phi|^2 \, \eta^2 \, \e^{ - \frac{|x|^2}{4} } 
			\notag\\
	\leq&\, C \int_{\Sigma_t}\phi^2\,\e^{-\frac{|x|^2}{4}} - \int_{\Sigma_t} |\nabla \phi|^2 \, \eta^2 \, \e^{ - \frac{|x|^2}{4} }  \, . 
\end{align}
Suppose that $s \in [t_1 + \beta , t_2]$.
To prove  \eqr{e:e1a}, we 
integrate \eqr{e:enu} to get
\begin{align}
	\int_{\Sigma_s}\phi^2\,\eta^2 \,\e^{-\frac{|x|^2}{4}} &\leq  
	\min_{[t_1 , t_1 + \beta ]} \int_{\Sigma_t}\phi^2\,\eta^2 \,\e^{-\frac{|x|^2}{4}} + C \int_{t_1}^s   \int_{\Sigma_t}\phi^2\,\e^{-\frac{|x|^2}{4}} 
	\leq (C+1/\beta) \, \int_{t_1}^s   \int_{\Sigma_t}\phi^2\,\e^{-\frac{|x|^2}{4}}  \notag \\
	&\leq (C+1/\beta) \,  \left(F(\Sigma_{t_1})-F(\Sigma_{t_2})\right) \, .
\end{align}
Finally, to get \eqr{e:e1b}, we integrate \eqr{e:enu} from $t_1 + \beta$ to $t_2$ and use \eqr{e:e1a} to bound
the contributions at the end points.
\end{proof}

\subsection{Uniform short time stability of the cylinder}

The last result that we will need for proving Theorem \ref{t:rR}
 is the following elementary short time uniform stability of the cylinder under MCF with bounded curvature:

\begin{Lem}	\label{l:cylind}
Given $R> \sqrt{2n}$, $\epsilon > 0$ and $C_0$, there exist $\delta > 0$  and $\theta > 0$ so that if
$M_t$ is a MCF with
\begin{enumerate}
\item $B_{R+2} \cap M_{-1}$ is a $C^{2,\alpha}$ graph over $\Sigma \in \cC_k$ with norm at most $\delta$.
\item $|A| + |\nabla A| + |\nabla^2 A | + |\nabla^3 A| \leq C_0$ on $B_{R+2} \cap M_t$ for $t \in [-1-1/C_0 , -1 + 1/C_0]$.
\end{enumerate}
Then for each $t \in [-1,\theta-1]$ we have that
\begin{itemize}
\item  $B_R \cap M_t$ is a $C^{2,\alpha}$ graph over $\sqrt{-t} \, \Sigma$ with norm at most $\epsilon$.
\end{itemize}
\end{Lem}

\begin{proof}
Since $|A|$ is bounded, the MCF equation implies that $|\partial_t x|$ is also bounded.  Likewise, the bound on $|\nabla A|$ (and thus on $|\nabla H|$)
and the evolution equation for the normal (see lemma $7.5$ in \cite{HP}) imply that $|\partial_t \nn|$ is also uniformly bounded.  Combining these two 
bounds, it follows that $B_{R+1} \cap M_t$ remains a graph over $\Sigma$ of a function $u$ with a uniform bound
\begin{align}
	 \left| \partial_t u \right| + \left| \partial_t \nabla u \right| \leq C_1 {\text{ for }}  t \in [-1-\theta_0 , -1 + \theta_0] \, , 
\end{align}
where $\theta_0 > 0$ and $C_1$ depend on $C_0 , \epsilon , n$.  Similarly, the higher derivative bounds on $A$ then yield bounds on higher derivatives of $u$
and the lemma follows immediately.
\end{proof}

\subsection{Proof of Theorem \ref{t:rR}}

We are now prepared to prove Theorem \ref{t:rR} which shows that the cylindrical scale is a fixed factor larger
than the shrinker scale.

\begin{proof}
(of Theorem \ref{t:rR}).    The theorem follows by an extension and improvement argument that is inspired by a similar argument for shrinkers in \cite{CIM}.
\vskip2mm
\noindent
{\bf{(1) Extending the scale.}}
Given $\ell$, we will show that there exist $\delta > 0$, $\bar{s}>0$, $\theta > 0$, $R_0$, $C_2$  and $C_{\ell}$ so that if 
\begin{enumerate}
\item[(A1)] $B_R \cap \Sigma_s$ is a graph of $u_1$ over some $\Sigma_1 \in \cC_k$ with $\| u_1 \|_{C^{2,\alpha}} \leq \delta$ for each $s \in [t_0-\bar{s} , t_0+ \bar{s}]$ for some $R \in [R_0 , \shrink (\Sigma_t)]$ and $t_0 \in [t-1/2 , t+1 - \bar{s}]$
\end{enumerate}
then, for every $s \in [t_0-\bar{s} , t_0+ \bar{s}]$, we have 
\begin{enumerate}
\item[(A2)]  $\graph(\Sigma_t) \geq (1+\theta)R$ and 
$\left|\nabla_{\Sigma_s} F \right|^2_{ B_{ (1+\theta)R} }
 \leq
C_2 \,  \left(F(\Sigma_{t-1})-F(\Sigma_{t+1})\right)$.
\end{enumerate}
 
The key observation is that the cylindrical estimates and global entropy bound imply that the local Gaussian densities on some fixed scale are almost one.  Thus, White's    Brakke estimate  \cite{W3} gives  a curvature bound on a larger region
$B_{(1+\kappa) R}$ with $\kappa > 0$ but at the cost of moving forward in time.  However, Proposition \ref{p:backwards}   pulls this curvature bound backwards in time {\emph{while only coming in by a fixed additive amount}}.  As long as $R$ is sufficiently large, the multiplicative gain beats the additive loss and, thus, the bound on $A$ extends to a larger scale with no loss in time.
 Ecker-Huisken \cite{EH} then gives  uniform higher derivative bounds on $A$.  We can now apply Lemma \ref{l:cylind} on a unit scale but centered at points out to the  extended scale to get the cylindrical estimates on the larger scale.
 Finally, 
using the curvature bounds,  Lemma \ref{l:meanvalue}
gives the constant $C_2$ so that $\left|\nabla_{\Sigma_s} F \right|^2_{ B_{ (1+\theta)R} }
 \leq
C_2 \,  \left(F(\Sigma_{t-1})-F(\Sigma_{t+1})\right)$.

\vskip2mm
\noindent
{\bf{(2)  The improvement below the shrinker scale}}.   The Lojasiewicz 
inequality of Theorem \ref{t:ourfirstloja3} will give an improved bound on the larger scale if we are below the shrinker scale:
\vskip1mm
\noindent 
Given $\tau > 0$,  $\delta > 0$,  $C_2$, $\ell$ and $C_{\ell}$, there exist $\ell_1$ and $R_1$ so that if 
  $\ell \geq \ell_1$ and  $R \in [ R_1 , \shrink (\Sigma_t)]$ satisfies
  \begin{align}
         R \leq  r_{\ell} (\Sigma_s)  {\text{ and }}
    \left|\nabla_{\Sigma_s} F \right|^2_{ B_{  R} }
 \leq
C_2 \,  \left(F(\Sigma_{t-1})-F(\Sigma_{t+1})\right) \, , 
\end{align}
then 
 $B_{(1-\tau)R} \cap \Sigma_s$ is a graph of $u_3$ over some $\Sigma_3 \in \cC_k$ with $\| u_3 \|_{C^{2,\alpha}} \leq \delta$.
 
 \vskip2mm
 \noindent
 {\bf{Putting it together}}:
 The point is to choose $\tau$ much smaller than $\theta$, so that the gain in scale from extending in (1) beats the loss in scale from the improvement in (2).  We can then apply the two steps iteratively to get a fixed factor greater than one beyond the shrinker scale, giving the theorem.
\end{proof}

\section{The gradient Lojasiewicz inequality and uniqueness}

In this section, we will use the gradient Lojasiewicz inequality of Theorem \ref{t:ourgradloja} and the compatibility of the shrinker and cylindrical scales of the previous section to prove a gradient Lojasiewicz inequality for rescaled MCF.  We will show that this inequality
implies uniqueness of the tangent flow at a cylindrical singularity, thus completing the proof of Theorem \ref{t:main}.

\subsection{Mean value inequalities}

 In this subsection, we will prove a mean value inequality  that is needed for the  gradient Lojasiewicz inequality.  The argument follows that of Lemma \ref{l:meanvalue}
 in the previous section
 with  the gradient of $\phi$ in place of $\phi$ essentially using the equation that one gets from taking the derivative of equation 
 \eqr{e:eqrphi} to get the following:

\begin{Lem}  \label{l:meanvalue3}
There exists a constant $C$ so that if
 $\Sigma_t$ flow by the rescaled MCF, and $r\leq \min_{t-1/2\leq s\leq t+1}\graph (\Sigma_s)$, 
then  
\begin{align}
  \max_{s \in [t- \frac{1}{4} , t + 1 ]} \, \, \int_{B_r\cap \Sigma_t} |\nabla \phi|^2\,\e^{-\frac{|x|^2}{4}}
&\leq
C \, (1+r)  \,  \left(F(\Sigma_{t-1})-F(\Sigma_{t+1})\right)
\, .\label{e:e1aa}\\
\int_{t-\frac{1}{4}}^{t+1} \int_{B_r\cap\Sigma_s}|\Hess_{\phi}|^2\,\e^{-\frac{|x|^2}{4}}
&\leq
C \, (1+r) \,\left(F(\Sigma_{t-1})-F(\Sigma_{t+1})\right)
\, .\label{e:e1bb}
\end{align}
\end{Lem}

Lemma 
\ref{l:meanvalue3} will follow from the same argument as in the proof of Lemma \ref{l:meanvalue} (together with the result of Lemma \ref{l:meanvalue}) provided we have the following:

\begin{Lem}
If  $\Sigma_t$ flow by the rescaled MCF, then
\begin{align}
(\partial_t-\cL)\,|\nabla \phi|^2\leq-2\, |\Hess_\phi|^2+C\,   |\nabla \phi|^2+\phi^2\, ,  \label{e:diffineq}
\end{align}
where $C$ depends only on $n$ and the bounds for $A$ and $\nabla A$.
\end{Lem}

\begin{proof}
To prove this, note first that if $\Sigma\subset \RR^{n+1}$ is a hypersurface, $f:\RR^{n+1}\to \RR$  is a smooth function, $X$, $Y\in T_x\Sigma$, then
\begin{align}
\Hess_f^{\RR^{n+1}} (X,Y)=\langle \nabla_{X}\left(\nabla ^Tf+\langle \nabla f,\nn\rangle\,\nn\right), Y\rangle=\Hess_f^{\Sigma}(X,Y)-\langle \nabla f,\nn\rangle\,A(X,Y)\, .   \label{e:hess}
\end{align}
Recall also that if $\Sigma$ is a manifold (not necessarily embedded in Euclidean space),
then the Bochner formula for the drift Laplacian $\Delta_f\,u=\Delta\,u-\langle\nabla f, \nabla u\rangle$ is 
 \begin{align}	\label{e:bochnerf}
\frac{1}{2}\Delta_f\,|\nabla u|^2=|\Hess_u|^2+\langle \nabla\Delta_f u,\nabla u\rangle+\Ric_f(\nabla u,\nabla u)\, .
\end{align}
Here $\Ric_f=\Ric+\Hess_f$ is the Bakry-\'Emery Ricci curvature.    

We will use that if $\Sigma_t\subset \RR^{n+1}$ is a one-parameter family of hypersurfaces moving by the rescaled MCF and $u=u(x,t):\Sigma_t\times \RR\to \RR$ is a smooth function, then
\begin{align}  \label{e:eee}
	\partial_t \,|\nabla^Tu|^2=2\,\langle \nabla^T \partial_t u,\nabla^T u\rangle+2\, \left(\frac{\langle x,\nn\rangle}{2} - H\right)\,  		A(\nabla^Tu,\nabla^T u)\, .
\end{align}
{\bf{Proof of \eqr{e:eee}}}:
To see this, extend $u$ to a function on $\RR^{n+1}\times\RR$ so that on $\Sigma_t$  
  $\nabla u=\nabla^T u$ 
and $\partial_t u = \langle \nabla u , \partial_t x \rangle + u_t = u_t$, where $u_t$ is the $t$ derivative of $u$ as a function on $\RR^{n+1}\times\RR$ and the rescaled MCF equation
is
\begin{align}
		 \partial_t x  = \left( \frac{1}{2} \, \langle x , \nn \rangle - H\right) \, \nn \, .
\end{align}
Therefore, differentiating $\nabla u = \nabla u (x(t) , t)$ on $\Sigma_t$,
the chain rule gives
\begin{align}
  	\partial_t \nabla u = \nabla_{\partial_t x} \nabla u  + \nabla u_t
		= \left( \frac{1}{2} \, \langle x , \nn \rangle - H\right) \, \nabla_{\nn} \nabla u  +
		\nabla \partial_t u \, .
\end{align}
Using this, the symmetry of the Hessian, and the definition of $A$ gives
\begin{align}
\frac{1}{2} \, \partial_t \,|\nabla^Tu|^2 &= \frac{1}{2} \, \partial_t \, |\nabla u |^2 =
\langle \partial_t \nabla u , \nabla u \rangle = 
\left( \frac{1}{2} \, \langle x , \nn \rangle - H\right) \, \langle \nabla_{ \nabla u } \nabla u , \nn \rangle
		+ \langle \nabla \partial_t u , \nabla u \rangle
\notag \\
&= \left( \frac{1}{2} \, \langle x , \nn \rangle - H\right) \, A( \nabla u ,  \nabla u)
		+ \langle \nabla \partial_t u , \nabla u \rangle
		\, ,
\end{align}
completing the proof of  \eqr{e:eee}.

\vskip2mm
Let $f=\frac{|x|^2}{4}$ so that \eqr{e:hess} gives
 \begin{align}
\Ric_f(\nabla u,\nabla u)
%&=\Ric (\nabla u,\nabla u)+\Hess^{\Sigma}_f(\nabla u,\nabla u)\notag\\ &
 =\Ric (\nabla u,\nabla u)+\frac{|\nabla u|^2}{2}+ \frac{\langle x,\nn\rangle}{2}\,A(\nabla u,\nabla u)\, .
\end{align}
Therefore,  combining  \eqr{e:eee} and the Bochner formula \eqr{e:bochnerf} gives
\begin{align}
	(\partial_t-\cL)\,|\nabla u|^2&=-2\,|\Hess_u|^2+2\,\langle \nabla (\partial_t-\cL)\, u,\nabla u\rangle\notag\\
	&-2\,\Ric(\nabla u,\nabla u)-|\nabla u|^2 
	- 2\,  H \,  A(\nabla u,\nabla u)\, .   \label{e:diffineq0}
\end{align}
Suppose now that   $u=\phi$, so that  $(\partial_t-\cL)\,\phi=(|A|^2+\frac{1}{2})\,\phi$ by \eqr{e:eqrphi}.
Finally,  \eqr{e:diffineq} follows from \eqr{e:diffineq0} 
and the absorbing inequality $2\,|\phi\, \langle \nabla |A|^2,\nabla \phi\rangle|\leq \phi^2+C\,|\nabla \phi|^2$.
\end{proof}

%Differentiating equation \eqr{e:eqrphi} twice and applying the argument in (and results of) Lemmas \ref{l:meanvalue} and \ref{l:meanvalue3} gives:

%\begin{Lem}  \label{l:meanvalue2}
%There exist constants $C$ and $\alpha$ so that if
% $\Sigma_t$ flow by the rescaled MCF, and $r\leq \min_{t-1\leq s\leq t+1}\graph (\Sigma_s)$, 
%then  
%\begin{align}
%  \max_{s \in [t- \frac{1}{8} , t + 1 ]} \, \, \int_{B_r\cap \Sigma_t} |\Hess_{\phi}|^2\,\e^{-\frac{|x|^2}{4}}
%&\leq
%C \, (1+r)^{\alpha}  \,  \left(F(\Sigma_{t-1})-F(\Sigma_{t+1})\right)
%\, .\label{e:e1aaa}
%\end{align}
%\end{Lem}

%We leave the details of the proof of this (that is very similar to that of the two lemmas above) to the reader.

\subsection{A discrete gradient Lojasiewicz inequality for  rescaled MCF}

The next theorem gives a discrete version of a gradient Lojasiewicz inequality for  rescaled MCF.  

\begin{Thm}	\label{t:gradMCF}
Given $n$ and $\lambda_0$, 
there exist constants $ K, \bar{R} , \epsilon$ and $\tau \in (1/3,1)$ so that if 
  $\Sigma_s$ is a rescaled MCF for $s \in [t-1, t+1]$ satisfying
  \begin{itemize}
  \item $\lambda (\Sigma_s)  \leq \lambda_0$.
  \item $B_{\bar{R}} \cap \Sigma_s$ is a $C^{2,\alpha}$ graph over  some cylinder in $\cC_k$   with  norm at most $\epsilon$ for each $s$.
\end{itemize}
Then we have
\begin{align}
	(F(\Sigma_t)-F(\cC))^{1+\tau} &\leq K\, \left(F(\Sigma_{t-1})-F(\Sigma_{t+1})\right) \, .
\end{align}
\end{Thm}

\begin{proof}
Given any $\beta \in [0,1)$ and $R \in [ 1 ,  \graph (\Sigma_t) - 2]$, 
Theorem \ref{t:ourgradloja}
 gives    
\begin{align}	\label{e:finglojakk} 
	 \left| F (\Sigma_t) - F (\cC) \right| 
      &\leq  C \,  R^{\rho}\, \left\{ 
        \|  \phi \|_{L^2(B_{ R \cap \Sigma_t} )}^{ c_{\ell,n} \, \frac{3+\beta}{2+2\beta} }   
       +  \e^{ - \frac{ R^2}{4} \,   c_{\ell , n}  } 
       +  \e^{ - \frac{ R^2}{4} \, \left( \frac{3+\beta}{4} \right)  }  \right\}  \, ,
\end{align}
where $C = C(n,\ell, C_{\ell}, \lambda_0)$,  $\rho = \rho (n)$ and  $c_{\ell , n} \in (0,1)$ satisfies $\lim_{\ell \to \infty} \, c_{\ell , n} = 1$.  We will bound each term by a power greater than $1/2$  of 
$\left(F(\Sigma_{t-1})-F(\Sigma_{t+1})\right) $.

We defined the shrinker scale $\shrink (\Sigma_t)$ in \eqr{e:shrinkdef}   by
\begin{align}	 
  \e^{ - \frac{\shrink^2(\Sigma_t)}{2} } = \int_{t-1}^{t+1}|\nabla_{\Sigma_s} F|^2 \, ds =F(\Sigma_{t-1})-F(\Sigma_{t+1})\, .
\end{align}
If we set
$
  R +2 \equiv  \min_{t-1/2\leq s\leq t+1}\graph (\Sigma_s) $, 
  then  Theorem \ref{t:rR}  gives
$\mu > 0$ and $C$  so that
 \begin{align}
   R \geq (1+\mu) \, \shrink (\Sigma_t) - C \, ,
\end{align} 
as long as we are willing to choose $C_{\ell}$ sufficiently large depending on $\ell$.  The crucial point   is that $\mu$ does not change when we take
$\ell$ larger, although   $C_{\ell}$ does depend on $\ell$.

  Lemmas \ref{l:meanvalue} %and \ref{l:meanvalue2}
gives a constant $C$  so that  
\begin{align}
   \|  \phi \|_{L^2(B_{ R \cap \Sigma_t} )}^2
&\leq
C \,  \left(F(\Sigma_{t-1})-F(\Sigma_{t+1})\right) %= C \, \e^{ - \frac{\shrink^2(\Sigma_t)}{2} } 
\, .
% % \max_{s \in [t- \frac{1}{8} , t + 1 ]} \, \, \int_{B_r\cap \Sigma_t} |\Hess_{\phi}|^2\,\e^{-\frac{|x|^2}{4}}
%&\leq
%C \, (1+r)^{\alpha}  \,  \left(F(\Sigma_{t-1})-F(\Sigma_{t+1})\right) \, .
\end{align}

We first choose $\beta \in [0,1)$ so that 
\begin{align}
	(1+\mu) \, \left( \frac{3+\beta}{4} \right) > 1 \, .
\end{align}
This takes care of the third term in \eqr{e:finglojakk}.  Now we choose $\ell$ large so that
\begin{align}
	c_{\ell , n} \, \left( \frac{3+\beta}{2+2\beta} \right) > 1 {\text{ and }}
	(1+\mu) \, c_{\ell , n} > 1 \, .
\end{align}
 This takes care of the first two terms.  Once we choose $\ell$, then  Theorem \ref{t:rR} gives $C_{\ell}$ and, thus, determines the multiplicative factor $K$.

\end{proof}

 \subsection{An extension of ``Lojasiewicz theorem''}

 Lojasiewicz used the 
gradient Lojasiewicz inequality to prove convergence of flow lines for the negative gradient flow of an analytic function $f$.
We will prove an analogous convergence result where the differential  inequality $ f^{2\beta}(t)\leq -f'(t)$ (that follows from 
the 
gradient Lojasiewicz) is replaced by the discrete inequality
$f^{2\beta}(t)\leq f(t-1)-f(t+1)$.  
This assumption is exactly what comes out of our analog of the
gradient Lojasiewicz inequality, i.e., out of Theorem \ref{t:ourgradloja}.
 
 %any integral 
%curve of the negative gradient flow of a function satisfying the 
%gradient Lojasiewicz inequality converges at infinity
%provided it has a limit point at infinity. 

  The extension will rely on the following elementary lemma:

\begin{Lem}	\label{l:discreteL}
If $f:[0,\infty)\to [0,\infty)$ is a  non-increasing function,  $\epsilon , K >0$  and
for $t \geq 1$
\begin{align}
K\,f^{1+\epsilon}(t)\leq f(t-1)-f(t+1)  \, ,
\end{align}
then there exists a constant $C$ such that
\begin{align}	\label{e:elemdecay}
f(t) \leq  C\,t^{-\frac{1}{\epsilon}} \, .
 \end{align}
 Moreover, if $\epsilon < 1$, then 
 \begin{align}	\label{e:claimtoo}
 	\sum_{j=1}^{\infty} \, \left(  f(j) - f(j+1)   \right)^{ \frac{1}{2} } < \infty \, .
 \end{align}
\end{Lem}

\begin{proof}
After replacing $f$ by $f/C_0$ for some positive constant $C_0$,
we can assume without loss of generality that $0<f(0)\leq 1$ and $K=1$.  
Set $t_0=4\,2^{\epsilon}\,f^{-\epsilon}(0)/\epsilon+2$ and $C=f(0)\, t_0^{\frac{1}{\epsilon}}$, then $f(0)=C\,t_0^{-\frac{1}{\epsilon}}$ and hence \eqr{e:elemdecay} holds for all $t\leq t_0$.  Next note that by assumption for all $t\geq 2$
\begin{align}
f^{1+\epsilon}(t)\leq f^{1+\epsilon}(t-1)\leq f(t-2)-f(t)\, .
\end{align}
Or, equivalently, for all $t\geq 2$
\begin{align}
f(t-2) \geq f(t)\,(1+f^{\epsilon}(t))\, .
\end{align}
We would like to show that \eqr{e:elemdecay} holds; so suppose not and let $t$ be a $t$ where inequality \eqr{e:elemdecay} fails.  After possibly replacing $t$ by $t-2$ a finite number of times we may assume that \eqr{e:elemdecay} fails for $t$, but holds for $t-2$.  From the choice of $C$ it follows that $t>t_0\geq 2$.    Moreover,  
\begin{align}
f(t-2) \geq f(t)\,(1+f^{\epsilon}(t))> C\,t^{-\frac{1}{\epsilon}}\,(1+C^{\epsilon}\,t^{-1})\, .
\end{align}
Combining this with the elementary inequality that $(1+h)^{-\epsilon}\leq 1-2^{-1-\epsilon}\,\epsilon\, h$ for all $h\leq1$ and that both $C^{\epsilon}\,t_0^{-1}=f^{\epsilon}(0)\leq 1$ and $2^{-1-\epsilon}\,\epsilon\, C^{\epsilon}=2^{-1-\epsilon}\,\epsilon\, f^{\epsilon}(0)\,t_0\geq 2$ gives 
\begin{align}
f^{-\epsilon}(t-2) 
< C^{-\epsilon}\,t\,(1+C^{\epsilon}\,t^{-1})^{-\epsilon}
\leq C^{-\epsilon}\,(t-2^{-1-\epsilon}\,\epsilon\, C^{\epsilon})
\leq C^{-\epsilon}\,(t-2)\, .
\end{align}
Contradicting that \eqr{e:elemdecay} holds for $t-2$ and, thus, completing the proof of the first claim.

Suppose now that $\epsilon < 1$ and  fix  some $p \in (1, 1/\epsilon)$.   Cauchy-Schwarz gives that
 \begin{align}	\label{e:claimtoo2}
 	\left[ \sum_{j=1}^{\infty} \, \left(  f(j) - f(j+1)   \right)^{ \frac{1}{2} } \right]^2 \leq
	\left[ \sum_{j=1}^{\infty} \,    (f(j) - f(j+1) )\, j^p    \right] \,
	\left[ \sum_{j=1}^{\infty} \,j^{-p} \right]
	 \, .
 \end{align}
 The last term is finite since $p>1$, so 
 it suffices to prove that $(f(j) - f(j+1) )\, j^p$ is summable.  However, this  follows from the summation by parts formula
 \begin{align}
 	\sum_{j=1}^n b_j \, (a_{j+1} - a_{j}) = \left[  b_{n+1} a_{n+1} - b_1 a_1 \right] -  \sum_{j=1}^{n-1} a_{j+1} \, (  b_{j+1}-b_j )  
 \end{align}
 with $a_j = f(j)$ and $b_j = j^p$ since the decay \eqr{e:elemdecay} implies that
 \begin{align}
 	(n+1)^p \, f(n+1) &\leq C \, (n+1)^{ - \frac{1}{\epsilon} } \, (n+1)^p  \to 0 \, , \\
	 \sum_{j=1}^{\infty}  f(j+1) \, [ (j+1)^p - j^p]  &\leq C \, p\,  \sum_{j=1}^{\infty} (j+1)^{ - \frac{1}{\epsilon} + p -1} < \infty \, .
 \end{align}
The first inequality in the second line used that $[ (j+1)^p - j^p] \leq p \, (j+1)^{p-1}$.

\end{proof}

\subsection{Uniqueness of tangent flows}

We are now prepared to prove the uniqueness of cylindrical tangent flows.
 
\begin{proof}[Proof of Theorem \ref{t:main}]
Let $\Sigma_t$ be the rescaled MCF associated to the cylindrical singularity. It follows from the uniqueness theorem of \cite{CIM} that if
a sequence $t_j \to \infty$, then there is a  subsequence 
  $t_j' \to \infty$ so that
$\Sigma_{t_j'}$ converges with multiplicity one to a cylinder $\Sigma \in \cC_k$.    It follows from White's Brakke-type theorem, \cite{W3}, that this convergence is smooth on compact subsets.  A priori, different sequences could lead to different cylinders (i.e., different rotations of the same cylinder); the point of this theorem is that this does not occur.

Given any fixed large $\rho$ and small $\epsilon > 0$, it follows from the previous paragraph that there must be some $T$ so that 
  \begin{itemize}
  \item For each $t \geq T$, there is a cylinder in $\cC_k$ so that, for each $s \in [t-1,t+1]$, 
   $B_{\rho} \cap \Sigma_s$ is a $C^{2,\alpha}$ graph over  this cylinder   with  norm at most $\epsilon$.
\end{itemize}
Therefore, we can apply
Theorem \ref{t:gradMCF} to $\Sigma_t$ for $t \geq T$ to get 
  $ K$ and $\mu \in (1/3,1)$ so that
 \begin{align}
	(F(\Sigma_t)-F(\cC))^{1+\mu} &\leq K\, \left(F(\Sigma_{t-1})-F(\Sigma_{t+1})\right)  \, .
\end{align}
This  ``discrete differential inequality''  allows us to apply
Lemma \ref{l:discreteL} to conclude that 
\begin{align}	\label{e:discre}
	\sum_{j=1}^{\infty} \left( F(\Sigma_j) - F(\Sigma_{j+1} \right)^{ \frac{1}{2} } < \infty \, .
\end{align}
Using Cauchy-Schwarz and  that rescaled MCF is the negative gradient flow for $F$, we have
\begin{align}
	\int_1^{\infty} \| \phi \|_{L^1 (\Sigma_t)} \, dt  &\leq  \sum_{j=1}^{\infty} \left( F(\Sigma_j) \,  \int_{j}^{j+1} \| \phi \|^2_{L^2 (\Sigma_t)} \, dt \right)^{ \frac{1}{2} }  \notag \\
	&\leq   \sqrt{F(\Sigma_0)} \, \sum_{j=1}^{\infty} \left(    F(\Sigma_j) - F(\Sigma_{j+1}    \right)^{ \frac{1}{2} }   < \infty \, , 
\end{align}
where the last inequality is \eqr{e:discre} and the $L^1$ and $L^2$ norms are all weighted Gaussian norms.  The uniqueness now follows immediately from
Lemma \ref{l:dist}.
\end{proof}

\appendix

\section{Geometric quantities on a graph}	\label{s:append}

In this appendix, we will prove some technical results for the geometry of normal exponential graphs over a 
hypersurface.  As one consequence, we will prove Lemma \ref{l:gradcm2} which computes
the gradient of the $F$ functional on graphs over cylinders.

Throughout this appendix,  
 $\Sigma_u$ will denote the graph of a function $u$ over a fixed hypersurface $\Sigma$ (in  most applications $\Sigma$ 
 will be a cylinder),
where $\Sigma_u$ is given by
\begin{equation}
	\Sigma_u = \{ x + u(x) \, \nn (x) \, | \, x \in \Sigma \} \, .
\end{equation}
We will assume that $|u|$ is small so  $\Sigma_u$ is contained in a tubular neighborhood of $\Sigma$ where the normal exponential map is invertible.  Let $e_{n+1}$ be the gradient of the (signed) distance function to $\Sigma$; note that $e_{n+1}$ equals $\nn$ on $\Sigma$.

The geometric quantities that we need to compute on $\Sigma_u$   are:
\begin{itemize}
\item The relative area element $\nu_u (p) = \sqrt{\det g^u_{ij}(p)}/ \sqrt{\det g_{ij}(p)}$, where $g_{ij}(p)$ is the metric for $\Sigma$ at $p$ and
 $g^u_{ij}(p)$ is the pull-back metric from the graph of $u$ at $(p+ u(p) \, \nn(p))$.
 \item The mean curvature $H_u(p)$ of $\Sigma_u$ at $(p+ u(p) \, \nn(p))$.
 \item The support function $\eta_u (p) = \langle p + u (p) \, \nn (p) , \nn_u \rangle$, where $\nn_u$ is the normal to $\Sigma_u$.
 \item The speed function $w_u (p) = \langle e_{n+1} , \nn_u \rangle^{-1}$ evaluated at $(p+ u(p) \, \nn(p))$.
 \end{itemize}

\vskip2mm
The mean curvature and the support function directly appear in the shrinker equation.  The speed function enters indirectly when we rewrite the equation in graphical form;   the speed function adjusts for  that the normal direction and vertical directions may not be the same.  The relative area element will be used to compute the mean curvature and to relate the gradient of $F$  to $\phi = \frac{1}{2} \langle x , \nn \rangle - H$.

\subsection{Calculations}

The next lemma gives the expressions for the $\nu_u$, $\eta_u$ and $w_u$ on a   graph $\Sigma_u$ over a general hypersurface
$\Sigma$.  The statement is rather technical and it is helpful to keep in mind  the special case where
  $\Sigma$ is the  hyperplane $\RR^n$ and the  quantities are given by
\begin{align}
	\nu_u = \sqrt{1 + |\nabla u|^2} = w_u  {\text{ and }}
	% \, , \\
	%H_u &= \dv_{\RR^n} \, \left( \frac{ \nabla u }{\sqrt{ 1 + |\nabla u|^2}} \right)  \, , \\
	\eta_u   =  \frac{ u  - \langle p , \nabla u   \rangle }{\sqrt{ 1 + |\nabla u|^2}} \,  .
\end{align}
The first part of the lemma  gives similar formulas for a general $\Sigma$.  The second part 
uses the formulas to compute  Taylor expansions of the quantities.  Some of these computations are used to compute linear approximations here,
while others are not used in this paper but  are recorded for future reference and will be used elsewhere.
 
\vskip2mm

\begin{Lem}	\label{l:areau}
There are   functions $w, \nu , \eta$ depending  on $(p , s , y)  \in \Sigma \times \RR \times T_p\Sigma$  that are smooth for $|s|$ less than the normal injectivity radius of $\Sigma$  so that:
\begin{align}
	w_u (p) = w(p,s,y) &=  \sqrt{ 1 + \left| B^{-1}(p,s) (y) \right|^2 }   \, ,  \label{e:herew} \\
	\nu_u (p) = \nu (p,s,y) &= w(p,s,y) \, \det \, \left( B(p,s) \right)   \, , \label{e:hereG}   \\
	\eta_u (p) = \eta (p,s,y) &= \frac{ \langle p , \nn (p) \rangle  + s - 
	\langle p , B^{-1} (p,s) (y) \rangle }{ w(p,s,y)
	} \, , \label{e:hereeta}
\end{align}
where the linear operator $B(p,s) \equiv \Identity - s \, A(p)$.
Finally, the functions $w$, $\nu$, and $\eta$ satisfy:
\begin{itemize}
\item $w(p,s,0) \equiv 1$,  $\partial_s w(p,s,0) = 0$, $\partial_{y_{\alpha}}  w  (p,s,0) = 0$, and $ 
\partial_{y_{\alpha}} \partial_{  y_{\beta} }w (p,0,0) = \delta_{\alpha \beta}$.
\item  $\nu(p,0,0) =1$; the non-zero first and second order terms are $ \partial_s \nu  (p,0,0) = H(p)$,
$\partial_s^2 \nu(p,0,0) = H^2 (p) - |A|^2 (p)$, 
 $\partial_{p_j} \partial_s \nu  (p,0,0) = H_j(p)$,
and 
$\partial_{y_{\alpha}} \partial_{  y_{\beta} }  \nu  (p,0,0) = \delta_{\alpha \beta}$.
\item $\eta (p,0, 0) = \langle p , \nn \rangle$, $ \partial_s  \eta  (p,0,0) = 1$, and $ \partial_{  y_{\alpha}} \eta (p,0,0) = - p_{\alpha}$.
\end{itemize}
\end{Lem}

\begin{proof}
Let $(p,s)$ be Fermi coordinates on the normal tubular neighborhood of $\Sigma$, so that $s$ measures the signed distance to $\Sigma$.    If we fix an $s$ and a path $\gamma (t)$ in $\Sigma$, then applying the normal exponential map 
for time $s$ sends $\gamma (t)$ to $\gamma (t) + s \, \nn (\gamma (t))$.  It follows that the differential is given by the {\underline{symmetric}}  linear operator
\begin{equation}
	B(p,s) \equiv (\Identity - s \, A(p)): T_p \Sigma \to T_p \Sigma \, , 
\end{equation}
where we used that $-A$ is the differential of the Gauss map to differentiate $\nn$ and   the Gauss lemma to
  identify $T_p \Sigma$ with the tangent space to the level set of the distance to $\Sigma$.

We will  use this to compute the relative area element for the graph $\Sigma_u$. 
Pushing forward an orthonormal frame $e_i$  for $\Sigma$ at $p$  gives
  a frame $E_i$ for $\Sigma_u$ at $(p,u(p))$
\begin{equation}	\label{e:framei}
	E_i \equiv B(p,u)(e_i) + u_i (p) \, \partial_s \, .
\end{equation}
Thus, the metric on the graph is given in this frame  by
\begin{align}
	g^u_{ij} (p) \equiv \langle E_i , E_j \rangle =  \langle B(p,u)(e_i) , B(p,u) (e_j) \rangle + u_i \, u_j \, .
\end{align}
Since the $e_i$'s are orthonormal on $\Sigma$, we get 
\begin{align}	\label{e:rewri}
	\nu_u^2 (p) = \det \, \left( B^2(p,u(p)) + \nabla u \otimes \nabla u (p) \right)  \, .
\end{align}
Similarly, using the frame \eqr{e:framei}, we see that
the vector field 
\begin{equation}	\label{e:normalu}
	\partial_s - B^{-1}(p,u(p)) (\nabla u )(p) = e_{n+1} - B^{-1}(p,u(p)) (\nabla u )(p)
\end{equation}
 is normal to $\Sigma_u$.  It  follows that the speed function is given by
\begin{align}
	w_u (p) &= \langle e_{n+1} , \nn_u \rangle^{-1} = \frac{ \left|  e_{n+1} - B^{-1}(p,u(p)) (\nabla u )(p) \right| }{ 
	\langle e_{n+1} ,  e_{n+1} - B^{-1}(p,u(p)) (\nabla u )(p) \rangle } \notag \\
	 &= \sqrt{ 1 + \left| B^{-1}(p,u(p)) (\nabla u (p)) \right|^2 }  \, .
\end{align}
To rewrite the relative area element, we will need two elementary facts.  The first is that for $n \times n$ matrices $M_1$ and $M_2$,
we have
%\begin{equation}
	$\det (M_1 \, M_2) = \det (M_1) \, \det (M_2) $.
%\end{equation}
The second is that for a vector $v \in \RR^n$, we have
\begin{equation}
	\det (\Identity + v \otimes v ) = 
	%\left| e_1+v_1\,v,\cdots,e_n+v_n\,v\right|  = 
	1 + |v|^2 \, .
\end{equation}
Using these two facts, we now rewrite \eqr{e:rewri} as
\begin{align}
	\nu_u^2 (p)  &=   \det \, \left\{  B(p,u(p)) \, 
	\left( \Identity +
	B^{-1} (p,u(p)) (\nabla u(p)) \otimes B^{-1} (p,u(p)) (\nabla u(p)) \right) \, B(p,u(p))  \right\} 
	  \notag \\
	%  & = \left[  \det \, \left( B(p,u(p)) \right)     \right]^2 \, \det \, \left( \Identity +
	%B^{-1} (p,u(p)) (\nabla u(p)) \otimes B^{-1} (p,u(p)) (\nabla u(p)) \right) 
	%  \\
	& = \left[  \det \, \left( B(p,u(p)) \right)   \, w_u (p) \right]^2 \, . % \notag
\end{align}
 To compute the support function $\eta_u$, first use the
 formula \eqr{e:normalu} to get
 \begin{equation}	\label{e:normun}
 	\nn_u = \frac{ e_{n+1}  - B^{-1}(p,u(p)) (\nabla u )(p)}{   \left|  e_{n+1} - B^{-1}(p,u(p)) (\nabla u )(p) \right|
	} =  \frac{ e_{n+1}  - B^{-1}(p,u(p)) (\nabla u )(p)}{  w_u (p)
	} \, ,
\end{equation}
where $\nn_u$ is evaluated at  $p + u(p) \, \nn (p)$.  Thus, 
  the support function  is given by
\begin{align}
	w_u (p) \, \eta_u (p) &=  
	\langle p + u(p) \, \nn (p) ,    e_{n+1}  - B^{-1}(p,u(p)) (\nabla u )(p) \rangle     \\
	&=   \langle p , \nn (p) \rangle  + u(p) - 
	\langle p , B^{-1} (p,u(p)) (\nabla u (p)) \rangle  \, , \notag
\end{align}
where the last equality used that $\nn(p)$ is equal to $e_{n+1}$ at the point $p + s \, \nn(p)$ for any $s$.

We have now established the formulas \eqr{e:herew}, \eqr{e:hereG} and
\eqr{e:hereeta} for the functions $w$, $\nu$, and $\eta$.
It is clear from the expressions for $w$, $\nu$ and $\eta$ that they are smooth in the three variables provided that $s$ is sufficiently small.  

 The next thing is to establish the second set of three claims that give the second order Taylor expansions for $w$, $\nu$, and $\eta$.
The function $w$ appears in all three expressions, so it is convenient to start there.  It follows immediately that $w(p,s,0) = 1$.  To compute the partials involving $y_{\alpha}$'s, we get
\begin{equation}	\label{e:wap}
	\partial_{y_{\alpha}} \, w (p,s,y) = \frac{  \sum_{\beta} \left(B^{-2} \right)_{\alpha \beta} (p,s)\, y_{\beta} }{ w(p,s,y)} \, .
\end{equation}
It follows that $\partial_{y_{\alpha}} w (p,s,0) = 0$.  To get the Hessian, we differentiate \eqr{e:wap} again 
\begin{equation}	\label{e:wap2}
	\partial_{y_{\alpha}} \partial_{  y_{\beta}} w (p,0,0) = \frac{  \left( B^{-2} \right)_{\alpha \beta} (p,0)  }{ w(p,0,0)} = \delta_{\alpha \beta} \, ,
\end{equation}
where the last equality used that 
 $B(p,0) = \Identity$.
 
Using \eqr{e:hereG}, we have $\nu(p,s,y)  = w(p,s,y) \, \cB (p,s)$ where 
\begin{equation}
	\cB (p,s) = \det \, \left( B(p,s) \right) = \det  \, (\Identity - s \, A(p) ) \, .
\end{equation}
We have $\cB (p,0) \equiv 1$ and $\partial_s \cB  (p,0) = - \Tr  (A(p)) = H (p)$.  This also gives $\partial_{s} \partial_{ p_j} \, \cB (p,0)= H_j (p)$.
To get the second derivative in $s$, observe that
\begin{align}
	\partial_s \, \log \cB (p,s) =    \Tr \left[ B^{-1}(p,s) \,  \partial_s B(p,s) \right] =   - \Tr \left[ (\Identity - s \, A(p) )^{-1} A(p) \right] \, .  
\end{align}
Thus, we see that
\begin{equation}
	\partial_s^2 \, \cB  (p,0) = \left( \partial_s \,  \cB  (p,0) \right)  \, H(p) - \cB (p,0) \, |A|^2 (p) = H^2 (p) - |A|^2 (p) \, .
\end{equation}
Combining the calculations for $\cB$ with the earlier ones for $w$, we can compute the first three Taylor series terms for $\nu$.  The constant term is $\nu(p,0,0) = 1$.  The   first order terms are
\begin{align}
	\partial_{p_j} \, \nu(p,0,0) &= 0 \, , \\
	\partial_s \, \nu(p,0,0) &= \left( \partial_s \cB (p,0) \right) \, w(p,0,0) + \left( \partial_s w (p,0,0) \right) \, \cB(p,0) = H(p) \, , \\
	\partial_{y_{\alpha}} \, \nu(p,0,0) &=   \left( \partial_{y_{\alpha}} w (p,0,0) \right) \, \cB(p,0) = 0  \, .
\end{align}
The second order terms involving just $s$ and $p$ derivatives are simplified greatly since $w(p,s,0) \equiv 1$.  These are
$\partial_{p_j } \partial_{p_k} \, \nu(p,0,0) = 0$ and
\begin{align}
	\partial^2_s \, \nu(p,0,0) &= \left\{ \left( \partial^2_s \cB   \right) \, w +2\,  \left( \partial_s \cB   \right) \, \partial_s w + \left( \partial^2_s w   \right) \, \cB \right\} (p,0,0) =  \partial^2_s \cB (p,0) \notag \\
	&= H^2 (p) - |A|^2 (p)  \, , \\
	\partial_{p_j}   \partial_s \, \nu(p,0,0) &= \left\{ 
	\left( \partial_s \cB   \right) \, \partial_{p_j} w  + \left( \partial_{p_j} \partial_s w  \right) \, \cB +
	\left(\partial_{p_j}  \partial_s \cB   \right) \, w  + \left( \partial_s w  \right) \, \partial_{p_j}  \cB 
	\right\}  (p,0,0) \notag \\
	&= \partial_{p_j}  \partial_s \cB  (p,0) =  H_j (p) \, .
   \end{align}
   To compute the terms involving $y$ derivatives, it is useful to keep in mind that $\cB$ does not depend on $y$.  We get
   \begin{align}
		\partial_{p_j} \partial_{y_{\alpha}} \, \nu(p,0,0) &= \left\{  \left( \partial_{p_j}  \partial_{y_{\alpha}} w   \right) \, \cB +
		\left( \partial_{y_{\alpha}} w   \right) \, \partial_{p_j}  \cB  \right\} (p,0,0) = 0  \, , \\
		\partial_s \partial_{y_{\alpha}} \, \nu(p,0,0) &=  \left\{  \left( \partial_{s}  \partial_{y_{\alpha}} w   \right) \, \cB +
		\left( \partial_{y_{\alpha}} w   \right) \, \partial_s  \cB  \right\}(p,0,0) = 0   \, , \\
		\partial_{y_{\beta}} \partial_{y_{\alpha}} \, \nu(p,0,0) &=   \left( \partial_{y_{\beta}}  \partial_{y_{\alpha}} w (p,0,0) \right) \, \cB(p,0) = \delta_{\alpha \beta}  \, .
\end{align}
Finally, using  \eqr{e:hereeta} and the fact that the first derivatives of $w$ vanish at $(p,0,0)$, we get the first order expansion for $\eta$.
\end{proof}

\subsection{The mean curvature and its linearization via the first variation}

We  will compute the mean curvature $H_u$ using the first variation of the area of  $\Sigma_u$.  This  gives  a divergence form  equation in  $u$.

\begin{Cor}	\label{c:graphue}
The mean curvature $H_u$ of $\Sigma_u$ is given by
\begin{align}	\label{e:ymc}
	H_u (p) &= \frac{w}{ \nu } \left[ \partial_s \nu       - \dv_{\Sigma} 
	\left(  \partial_{y_{\alpha}} \nu   \right) \right]  \\
	&=\frac{w}{ \nu } \, \left( \partial_s \nu      -  
	   \partial_{ p_{\alpha}} \partial_{y_{\alpha}}  \nu
	  - \left(  \partial_s  \partial_{y_{\alpha}} \nu \right) u_{\alpha}(p)
	  -   \left( \partial_{ y_{\beta} }\partial_{ y_{\alpha}}  \nu \right)  u_{\alpha \beta}(p)
	 \right)
	\, , \notag
\end{align}
where $w$, $\nu$ and their derivatives  are all evaluated at $(p,u(p) , \nabla u(p))$.
\end{Cor}

\begin{proof}
By Lemma \ref{l:areau}, the area of the graph $\Sigma_u$ is
\begin{equation}
	\Area (\Sigma_u) = \int_{\Sigma} \nu_u \, \, dp_{\Sigma}  =  \int_{\Sigma} \nu(p,u(p), \nabla u(p)) \, dp_{\Sigma} \, .
\end{equation}
 Given a one-parameter family of graphs $\Sigma_{u+tv}$ with $v$ compactly supported,   differentiating the area gives
 \begin{align}
	\frac{d}{dt}  \, \big|_{t=0} \, \Area (\Sigma_{u+tv}) &= \int_{\Sigma}  \left\{
	\partial_s \nu (p,u(p),\nabla u(p))  \, v(p) + \partial_{y_{\alpha}} \nu  (p,u(p),\nabla u(p)) \, v_{\alpha}  (p) \right\} \,  dp_{\Sigma}  \notag \\
	&= \int_{\Sigma}  \left\{ 
	\partial_s \nu  (p,u(p),\nabla u(p))   - \dv_{\Sigma} 
	\left( \partial_{y_{\alpha}} \nu (p,u(p),\nabla u(p)) \right) \right\} \,  v  (p) \, dp_{\Sigma}  \, .
\end{align}
On the other hand, the variation vector field on $\Sigma_u$ is given by $v \, e_{n+1}$ so the
first variation formula (see, e.g., ($1.45$) in \cite{CM3}) gives
\begin{align}
	\frac{d}{dt}  \, \big|_{t=0} \, \Area (\Sigma_{u+tv}) = \int_{\Sigma_u} H_u \,  \langle v \, e_{n+1} , \nn_u \rangle  =
	 \int_{\Sigma} H_u (p) \, \frac{v (p)\, \nu_u (p)}{w_u (p)}  \, dp_{\Sigma}  \,  ,
\end{align}
where the second equality used the definition of the speed function $w_u = \langle e_{n+1} , \nn_u \rangle^{-1}$.

Equating these two expressions for the derivative of area, we conclude that
\begin{align}
	H_u (p) \, \frac{\nu (p,u(p) , \nabla u(p))}{w (p, u(p), \nabla u (p))}   =  \partial_s \nu (p,u(p),\nabla u(p))   - \dv_{\Sigma} 
	\left( 
	\partial_{y_{\alpha}} \nu (p,u,\nabla u) \right) \, .
\end{align}
This gives the first equality in  \eqr{e:ymc}; the second equality follows from the chain rule.
\end{proof}

 \subsection{The $F$ functional near a cylinder}
 
 We now specialize to  where  $\Sigma$ is a cylinder in $\cC_k$ and
 $F(u)$ is the $F$ functional of the graph $\Sigma_u$.
 
 \begin{Lem}	\label{l:gradcm}
 If $\Sigma \in \cC_k$, then the gradient $\cM (u)$ of the $F$ functional is given by
 \begin{align}
 	\cM (u) =  \frac{ \nu}{w} \, 
	 \, \left( H_u - \frac{1}{2} \eta \right)\,   \e^{- \frac{ 2 \, \sqrt{2k} \, u + u^2}{4} }
	     \, ,
 \end{align}
 where $H_u$ is the mean curvature of $\Sigma_u$ and $\nu , w , \eta$ are all evaluated at
 $(p, u(p) , \nabla u (p))$.
 \end{Lem}
 
 \begin{proof}
 Since we are using the Gaussian $L^2$ inner product, $\cM (u)$ is defined by
 \begin{align}	\label{e:111}
 	\frac{d}{dt} \big|_{t=0} \, F (u+tv) = \left( 4\pi \right)^{ - \frac{n}{2} } \, \int_{\Sigma}  v \, \cM (u) \,  \e^{ - \frac{|p|^2}{4} } \, d\mu_{\Sigma} \, .
 \end{align}
 On the other hand, the first variation formula for the $F$ functional from \cite{CM1} gives
  \begin{align}	\label{e:rewr}
 	\frac{d}{dt} \big|_{t=0} \, F (u+tv) &=  \left( 4\pi \right)^{ - \frac{n}{2} } \,  \int_{\Sigma_u}  \langle v\, e_{n+1} , \nn_u \rangle \, \left( H_u - \frac{1}{2} \, \langle \nn_u , x \rangle \right)  \,  \e^{ - \frac{|x|^2}{4} } \, d\mu_{\Sigma_u} \, ,
 \end{align}
 where each quantity   is evaluated on $\Sigma_u$.
 Given $p \in \Sigma$, we have
\begin{align}	\label{e:gaussiansagree}
	\left| p + u(p) \, \nn (p) \right|^2 = |p|^2 + u^2 + 2 \, u \,  \langle p , \nn \rangle 
	=   |p|^2 + u^2 + 2\, \sqrt{2k} \, u  \, , 
\end{align}
 where the last equality used that $\Sigma \in \cC_k$.
Writing \eqr{e:rewr} as an integral over $\Sigma$ gives
  \begin{align}	\label{e:122}
 	\frac{d}{dt} \big|_{t=0} \, F (u+tv) &=  \left( 4\pi \right)^{ - \frac{n}{2} } \,  \int_{\Sigma}  \frac{v}{w} \, 
	 \, \left( H_u - \frac{1}{2} \eta \right)\,   \e^{- \frac{ 2 \, \sqrt{2k} \, u + u^2}{4} }
	   \,  \nu \, 
	 \e^{ - \frac{|p|^2}{4} } \, d\mu_{\Sigma} \, .
 \end{align}
 The lemma follows by equating \eqr{e:111} and \eqr{e:122}
 \end{proof}

 \begin{proof}[Proof of Lemma \ref{l:gradcm2}]
 By Lemma \ref{l:gradcm} and Corollary \ref{c:graphue},
  $\cM (u)$ can be written as
 \begin{align}	\label{e:listem}
 	\cM (u) \,  \e^{ \frac{ 2 \, \sqrt{2k} \, u + u^2}{4} }=   \partial_s \nu      -  
	   \partial_{ p_{\alpha}} \partial_{y_{\alpha}}  \nu
	  - \left(  \partial_s  \partial_{y_{\alpha}} \nu \right) u_{\alpha}(p)
	  -   \left( \partial_{ y_{\beta} }\partial_{ y_{\alpha}}  \nu \right)  u_{\alpha \beta}(p)
	   -
	\frac{ \nu}{2\,w} \, 
	 \eta  \,  
	     \, .
 \end{align}
 Since the exponential term  depends only on $u$, we have to show that each of the five terms on the right side can be expressed as either:
 \begin{align}
({\text{i}})\,  f(u , \nabla u), \, \, 
({\text{ii}})\,   \langle p , V (u, \nabla u) \rangle  {\text{ or }}
({\text{iii}})\,  \Phi^{\alpha \beta} (u, \nabla u) \, u_{\alpha \beta} \,  . \notag
 \end{align}
 The proof will repeatedly use the calculations from Lemma \ref{l:areau}.
 
 The key point is that 
 $A$ is  parallel on  cylinders and, thus, the linear operator $B(p,s)$ depends only on $s$ (and not   $p$).  In particular, the function $\nu$ depends only on $s$ and $y$ (and not   $p$).  Thus,  the first three terms on the right side of
  \eqr{e:listem} are type (i) and the fourth term is type (iii).  Similarly, $w$  depends only on $s$ and $y$, it suffices to show that $w\, \eta$ is a sum of terms of the three allowed types.  Lemma \ref{l:areau} gives
  \begin{align}
   w   \,  \eta =  \langle p , \nn (p) \rangle  + s - 
	\langle p , B^{-1} (p,s) (y) \rangle  
 \, .
  \end{align}
  The first term is constant (so trivially type (i)) and the second is also   type (i).  Finally, since $B$ depends only on $s$, the third term is   type (ii).  
 \end{proof}

 \subsection{Rescaled MCF near a shrinker}

Let $\Sigma \subset \RR^{n+1} $ be an    embedded shrinker and $u(p,t)$   a  smooth function on $\Sigma \times (-\epsilon , \epsilon)$, giving a one-parameter family of hypersurfaces $\Sigma_u$.  We next derive the graphical rescaled MCF equation.

\begin{Lem}	\label{l:normpart1}
The graphs $\Sigma_{u}$ flow by rescaled MCF   if and only if $u$ satisfies
\begin{align}
	 \partial_t u(p,t)  =  w (p, u(p,t) , \nabla u (p,t)) \, 
	 \left( \frac{1}{2} \, \eta (p, u(p,t) , \nabla u(p,t)) - H_u \right)   \, .
\end{align}
\end{Lem}

\begin{proof}
As in \cite{EH}, the rescaled MCF equation $x_t = \left( \frac{1}{2} \, \langle x , \nn \rangle - H \right) \, \nn$ is equivalent (up to  tangential diffeomorphisms) to the equation
\begin{equation}
	\left( x_t \right)^{\perp} = \frac{1}{2} \, \langle x , \nn \rangle - H \, .
\end{equation}
The variation vector field and unit normal for $\Sigma_u$ are $\partial_t u(p,t) \, \nn (p)$ and $\nn_u$, respectively, at the point $p+ u(p,t) \, \nn (p)$, so we get the equation
\begin{align}
	\langle  \nn (p) , \nn_u \rangle \, \partial_t u(p,t) = \langle \left(  \partial_t u(p,t) \right) \, \nn (p) , \nn_u \rangle = \frac{1}{2} \, \eta_u - H_u \, .
\end{align}
Finally, multiplying through by $w_u = \langle  \nn (p) , \nn_u \rangle^{-1}$  gives the lemma.
\end{proof}

We will use the following lemma bounding the distance between time slices of a rescaled MCF by the $L^1$ norm of
the gradient of the $F$ functional.

\begin{Lem}	\label{l:dist}
Given $n$, there exist $C$ and $\delta > 0$  so that if $\Sigma \in \cC_k$ and 
  $\Sigma_u $ is a graphical solution of rescaled MCF on $[t_1 , t_2 ]$ with $\| u (\cdot , t) \|_{C^1} \leq \delta$, then
\begin{align}
	\int_{\Sigma} \left| u (p, t_2) - u (p,t_1) \right| \, \e^{- \frac{|p|^2}{4}} \leq   C\, \int_{t_1}^{t_2} \int_{\Sigma_u (t=r)} \left| \frac{ \langle x , \nn \rangle}{2} - H \right| \, \e^{ - \frac{|x|^2}{4} }   \, dr\, .
\end{align}
\end{Lem}

\begin{proof}  %[Proof of Lemma \ref{l:dist}]   
By Lemma \ref{l:normpart1},  $u$ satisfies 
\begin{align}
	 \partial_t u(p,t)  =  w (p, u(p,t) , \nabla u (p,t)) \,  \left( \frac{1}{2} \, \eta (p, u(p,t) , \nabla u(p,t)) - H_u \right)    \, .
\end{align}
Since $|u|$ and $|\nabla u|$ are small,  Lemma \ref{l:areau} gives that both $w$ and the 
 relative area element $\nu_u$ are uniformly bounded and \eqr{e:gaussiansagree} relates the Gaussians on $\Sigma$
 and $\Sigma_u$, so we get
  \begin{align}	\label{e:utbound2}
 	\int_{\Sigma} \left| \partial_t u(p,t) \right|   \, \e^{ - \frac{|p|^2}{4} }   &\leq C   \, \int_{\Sigma}  \left| \frac{1}{2} \, \eta (p, u(p,t) , \nabla u(p,t)) - H_u \right| \, \e^{ - \frac{|p|^2}{4} } \notag \\
	&\leq C'   \, \int_{\Sigma}  \left| \frac{1}{2} \, \eta (p, u(p,t) , \nabla u(p,t)) - H_u \right| \, \nu_u \, \e^{ - \frac{| p + u(p,t) \, \nn|^2}{4} }   \\
	&=  C' \, \int_{\Sigma_u} \left| \frac{ \langle x , \nn \rangle}{2} - H \right| \e^{ - \frac{|x|^2}{4} } \, . \notag
 \end{align}
The lemma follows from integrating this with respect to $t$, using the fundamental theorem of calculus and    Fubini's theorem.

\end{proof}

\section{An interpolation inequality}

We will use the following interpolation inequality which  is well-known, but we are including 
the short proof since we do not have an exact reference.  Unlike the rest of this paper, the $L^1$ norms below 
are unweighted.

\begin{Lem}	\label{l:interp1}
There exists $C=C(k,n)$ so that if $u$ is a $C^k$ function on $B_{2r} \subset \RR^n$, then
\begin{align}
  \| u \|_{L^{\infty} (B_{r})}   &\leq C \, \left\{ r^{-n}\,  \| u \|_{L^1(B_{2r})} +    \| u \|_{L^1(B_{2r})}^{a_{k,n}} \, 
    \| \nabla^k u \|_{L^{\infty}(B_{2r})}^{1-a_{k,n}}  \right\} \, , 
     \label{e:i1} \\
      r\, \| \nabla u \|_{L^{\infty} (B_{r})} &\leq C \, \left\{ r^{-n} \, \| u \|_{L^1(B_{2r})} + 
        r \,  \| u \|_{L^1(B_{2r})}^{b_{k,n}} \, 
    \| \nabla^k u \|_{L^{\infty}(B_{2r})}^{1-b_{k,n}} \right\}  \, ,  \label{e:i2} \\
      r^2\, \| \nabla^2 u \|_{L^{\infty} (B_{r})} &\leq C \, \left\{ r^{-n} \, \| u \|_{L^1(B_{2r})} + 
        r^2 \,  \| u \|_{L^1(B_{2r})}^{c_{k,n}} \, 
    \| \nabla^k u \|_{L^{\infty}(B_{2r})}^{1-c_{k,n}} \right\}  \, ,  \label{e:i3}
\end{align}
where $a_{k,n} = \frac{k}{k+n}$, $b_{k,n} =  \frac{k-1}{k+n}$ and $c_{k,n} =  \frac{k-2}{k+n}$.
\end{Lem}

\begin{proof}
By scaling, it suffices to prove the case $r=1$.

\vskip2mm
 The starting point is the following standard consequence of the Bernstein/Kellogg inequality
for polynomials, \cite{K}:
\begin{itemize}
 \item[(K)] Given $n$ and $d$, there exists $C_{d,n}$ so that if $p$ is a polynomial of degree at most $d$ on
 a ball $B_{\delta} \subset \RR^n$ for some $\delta > 0$, then 
 \begin{align}
    \| p \|_{L^{\infty} (B_{\delta})}+ \delta \, \| \nabla  p \|_{L^{\infty} (B_{\delta})} + 
     \delta^2 \, \| \nabla^2  p \|_{L^{\infty} (B_{\delta})} 
      \leq C_{d,n} \, \delta^{-n} \, \int_{B_{\delta}} |p| \, .
 \end{align}

\end{itemize}
Set $m = \| \nabla^k u \|_{L^{\infty}(B_2)}$. 
Choose $x \in \overline{B_1}$ where $|u|$ achieves its maximum and let $p$ be the degree
$(k-1)$ polynomial giving the 
first $(k-1)$ terms of the Taylor series of $u$ at $x$.  In particular, given any $\delta \in (0,1]$, Taylor
expansion gives
\begin{align}
   \int_{ B_{\delta}(x)} |u-p| \leq C \, m \, \delta^{n+k} \, , 
\end{align}
where $C$ depends on $n$ and $k$.  Using this in (K) gives
 \begin{align}	\label{e:mykell}
     \| u \|_{L^{\infty} (B_{1})} &=  |p|(x)   \leq C  \, \delta^{-n} \, \int_{B_{\delta}(x)} |p|  
     \leq C  \, \delta^{-n} \, \left\{  \int_{B_{\delta}(x)} |u| + \int_{B_{\delta}(x)} |u-p| \right\} \notag \\
     &\leq C  \, \delta^{-n} \,  \left\{ \| u \|_{L^1(B_2)} +  C \, m \, \delta^{n+k} \right\}
     \, .
 \end{align}
 We now consider two cases.  First, if $m \leq \| u \|_{L^1(B_2)}$, then \eqr{e:mykell} with $\delta =1$ 
 gives  
 \begin{align}
    \| u \|_{L^{\infty} (B_{1})} \leq C \, \| u \|_{L^1(B_2)} \, .
 \end{align}
Next, if $m > \| u \|_{L^1(B_2)}$, then we 
set $\delta^{n+k} =   \frac{\| u \|_{L^1(B_2)}}{m} $ (which is less than one) and  \eqr{e:mykell}  gives
\begin{align}
    \| u \|_{L^{\infty} (B_{1})} & \leq C   \, \| u \|_{L^1(B_2)}^{\frac{k}{n+k}} \, m^{\frac{n}{n+k}}  \, .
\end{align}
Thus, we see that \eqr{e:i1} holds in either case.

We will argue similarly to get the $\nabla u$ bound.  This time, let $x \in \overline{B_1}$ be 
a point where $|\nabla u|$ achieves its maximum.  Given $\delta \in (0,1]$, using (K) gives
\begin{align}   \label{e:gradtime}
  |\nabla u|(x) = |\nabla p|(x) \leq C  \, \delta^{-n-1} \,
  \left\{  \| u \|_{L^1(B_2)}  +  C \, m \, \delta^{n+k} \right\}  \, .
\end{align}
In the case where $m \leq \| u \|_{L^1(B_2)}$, we get \eqr{e:i2} by setting $\delta =1$.  On the other hand,
when $m > \| u \|_{L^1(B_2)}$, then we 
set $\delta^{n+k} =   \frac{ \| u \|_{L^1(B_2)} }{m} $ (which is less than one) and
\eqr{e:gradtime} gives
\begin{align}   \label{e:gradtime2}
  |\nabla u|(x)  \leq C  \, \| u \|_{L^1(B_2)}^{ \frac{k-1}{n+k} }
  \, m^{ \frac{n+1}{n+k} } \, ,
\end{align}
   completing the proof of \eqr{e:i2}.  The last bound \eqr{e:i3} follows similarly.

\end{proof}

\end{document}